\documentclass[11pt,a4]{amsart}

\usepackage{amsmath}
\usepackage{amsfonts}
\usepackage{amssymb}
\usepackage{amsthm}
\usepackage{cancel}
\usepackage{epsfig}
\usepackage{graphicx}
\usepackage{url,hyperref}
\usepackage[style=english]{csquotes}
\usepackage{enumerate}
\usepackage{lscape}
\usepackage[width=\textwidth]{caption}
\usepackage{subcaption}
\DeclareCaptionSubType*[arabic]{figure}
\usepackage{own-macros}

\begin{document}

\title[A complete 3D BS-diagram]{A complete 3-dimensional Blaschke-Santal\'o diagram}

\author{Ren\'e Brandenberg and Bernardo Gonz\'alez Merino}
\address{Zentrum Mathematik, Technische Universit\"at M\"unchen, Boltzmannstr. 3, 85747 Garching bei M\"unchen, Germany} \email{brandenb@ma.tum.de}

\address{Departamento de Matem\'aticas, Universidad de Murcia, Campus
Espinar\-do, 30100-Murcia, Spain} \email{bgmerino@um.es}

\thanks{The second author was partially supported by MINECO (Ministerio de Econom\'{\i}a
y Competitividad) and FEDER (Fondo Europeo de Desarrollo Regional) project MTM2012-34037,
and La Fundaci\'on S\'eneca Regi\'on de Murcia, under the programme
{\it Becas-contrato predoctorales de formaci\'on del personal investigador}, 2008.}

\subjclass[2000]{Primary 52A10,52A40; Secondary 52A20}

\keywords{Inner and outer radii,
  geometric inequalities, Blaschke diagram, Blaschke-Santal\'o diagram, shape diagrams}

\begin{abstract}
  We present a complete 3-dimensional Blaschke-Santal\'o
  diagram for planar convex bodies with respect to the
  four classical magnitudes inner and outer radius, diameter and (minimal) width in
  euclidean spaces.
\end{abstract}

\maketitle

\section{Introduction} \label{s:intro}

The focus of the paper are the standard radii  measured for the family $\CK^n$ of
convex bodies $K \subset \R^n$, where a convex body is a compact and convex set in euclidean
$n$-space.
The \cemph{dred}{diameter} $D(K)$ of $K$ is
the biggest distance between two points of $K$, the \cemph{dred}{width}
$w(K)$ is the minimal \cemph{dred}{breadth}, \ie~the smallest
distance between any two different parallel supporting hyperplanes of $K$.
The \cemph{dred}{inradius} $r(K)$
is the radius of a biggest ball contained in $K$, and the
\cemph{dred}{circumradius} $R(K)$ is the radius of the (unique) smallest ball
containing $K$.

A natural and very intuitive question is the following: if
$K\in\CK^n$ is given and we have fixed values for some of the previous
radii (say \eg $r$, $D$ and $R$), which is the range of possible values
of $w$ depending on $r,D$ and $R$? A comprehensive solution of this task in
$\CK^2$ is presented in the following in form of a Blaschke-Santal\'o diagram
(sometimes also called shape diagram).

\medskip

Let us start with some historical and more general review:
In \cite{Bl} Blaschke proposed the study of possible values for the volume $V(K)$, surface area $S(K)$,
and integral mean curvature $M(K)$ for any $K \in \CK^3$.
For doing so, he considered the mapping
\[ h:\CK^3 \to [0,1]^2, \text{ with } h(K):=\left(\frac{4\pi S(K)}{M(K)^2},
 \frac{48\pi^2 V(K)}{M(K)^3}\right). \]
The image $h(\CK^3)$ is well known as \cemph{dred}{Blaschke diagram}. Blaschke
realized that the isoperimetric inequality and the geometric inequalities of Minkowski
were not sufficient for a complete description of $h(\CK^3)$. A \emph{complete system of
inequalities} needed additional geometric inequalities relating $V$, $S$ and $M$, still
a famous open problem in convex geometry, see \cite{HCSa,SY}.

Reviving the idea of Blaschke, Santal\'o proposed in \cite{Sa} the study
of such diagrams for all triples of the magnitudes $r$, $w$, $D$, $R$, $p$ (perimeter) and
$A$ (area), for a start, for planar sets.
Once a triple is
fixed, say $(r,D,R)$, the function \[ g:\CK^2 \to [0,1]^2, \text{ with }
g(K):=\left(\frac{r(K)}{R(K)},\frac{D(K)}{2R(K)}\right)\] is considered, and its image $g(\CK^2)$ is
called a \cemph{dred}{Blaschke-Santal\'o diagram}.
Full descriptions of those diagrams for the triples $(A,p,w)$, $(A,p,r)$,
$(A,p,R)$, $(A,w,D)$, $(p,w,D)$ , and $(r,D,R)$
are already derived in \cite{Sa}.

An important first ingredient in the full description of the diagram for
$(r,D,R)$ in \cite{Sa} are the well known (and easy to prove)
inequalities
\begin{equation}\label{eq:1}
2r(K) \le w(K) \le D(K) \le 2R(K)
\end{equation}
as well as the inequality of Jung \cite{Ju}
\begin{equation}\label{eq:6}
 R(K) \le \sqfr{n}{2(n+1)} D(K)
\end{equation}
which are true for all $K\in \CK^n$.

Moreover the validity of the inequalities
\begin{equation}\label{eq:2}
w(K) \le r(K) + R(K) \le D(K)
\end{equation}
was shown in \cite{Sa} (there only for $n=2$, but easy to see to be true in general dimensions,
see \eg \cite{Br}).
Equality in \eqref{eq:2} holds true simultaneously, if $K$ is of constant width (\ie if $w(K) = D(K)$).
See \cite{BrK2} for more details and extensions on \eqref{eq:2}.

A final inequality derived in \cite{Sa} holds true (in the given form) for $K \in \CK^2$ only:
\begin{equation*}
  2R(K) \left(2R(K) + \sqrt{4R(K)^2-D(K)^2}\right)r(K) \ge D(K)^2\sqrt{4R(K)^2-D(K)^2},
\end{equation*}
with equality, if $K$ is an isosceles triangle.

This inequality, together with Jung's inequality \eqref{eq:6} and the relevant parts of \eqref{eq:1} and \eqref{eq:2}
forms a complete system of inequalities for $(r,D,R)$.

Moreover, Santal\'o observed that previously known inequalities did not form complete systems of inequalities
in any case of changing one of $(r,D,R)$ to the width $w$.

In \cite{HC} and \cite{HCS} Hern\'andez Cifre and Segura Gomis could state the missing inequalities:
\begin{align} \label{eq:3}
  (4R(K)^2-D(K)^2)D(K)^4 & \le 4w(K)^2R(K)^4 , \\
  \notag  (4r(K)-w(K))(w(K)-2r(K))R(K) &\le 2r(K)^3, \\
  \notag  D(K)^4(w(K)-2r(K))^2(4r(K)-w(K)) &\le 4r(K)^4w(K), \text{and} \\
  \notag  \sqrt{3} \left(w(K) - r(K)\right) & \le  D(K).
\end{align}
Moreover, they showed that those of the new inequalities together with those in \eqref{eq:1} to \eqref{eq:2},
which do only involve the appropriate radii, form complete systems of inequalities for the
triples $(w,D,R)$, $(r,w,R)$ and $(r,w,D)$. In \cite{HC2} Hern\'andez Cifre computed
complete systems of inequalities for the triples $(A,D,R)$ and $(p,D,R)$ and
B\"or\"oczky Jr., Hern\'andez Cifre, and Salinas gave complete diagrams for the triples $(A,r,R)$ and $(p,r,R)$ in \cite{BHCS}.

Thus all 4 Blaschke-Santal\'o diagrams of $\CK^2$ involving only the classical radii
$r,w,D$, and $R$ could be completed, but there are still 7 out of the 20 possible triples
involving also $A$ and $p$ where a full description
of the diagram is still missing or at least unproven.

\medskip

Recently, Blaschke-Santal\'o diagrams have been used in pattern recognition
and image analysis (see \cite{RDP5,RDP6,Pr}), as they help in the prediction
of the size and shape of 3-dimensional sets from their
2-dimensional projections. In \cite{RDP5,Pr} for example,
the diagrams (in this context mostly called shape-diagrams) have been combined with
probabilistic methods, such as maximum likelihood estimation in the second of them.

\medskip

Once there are complete systems of inequalities for some of the possible triples of magnitudes
amongst $A,p,r,w,D,R$, it is a natural step to consider complete systems of inequalities
for even more than three of those magnitudes, \eg to obtain stronger inequalities or for an even more
accurate classification of convex sets in the mentioned application in image analysis.
In \cite{TK} the quadruple $(A,p,w,D)$ has been considered, without deriving a complete
description (which is not so much surprising, as even for the triple $(A,p,D)$ a complete description
is still missing). We consider the case $(r,w,D,R)$, which is the unique diagram involving
four of the above magnitudes, \st~for all choices of triples of them
complete descriptions of the diagrams are known. To the best of our knowledge, besides \cite{TK} and
the preliminary work of this paper in \cite{Br}, it is the only paper considering more than three of these magnitudes.

The study shows the necessity of \emph{new inequalities} relating all four radii at
once. This is done by describing the diagram's skeleton, \ie~its
boundary structure consisting of 0-,1-, and 2-dimensional differential manifolds
(see below for a proper definition).

\medskip

Further notation is the following: If $l\in\N$, we abbreviate $[l] := \{1,\dots,l\}$.
For a general set $C \subset \R^n$, we write $\aff(C)$ and $\conv(C)$ for the
\cemph{dred}{affine hull} and the \cemph{dred}{convex hull} of $C$, respectively,
and for any $x,y\in\R^n$ we denote by $[x,y]$ the \cemph{dred}{segment} $\conv\{x,y\}$
whose endpoints are $x$ and $y$.
If $K\in\CK^n$ we write $\ext(K)$ for the set of \cemph{dred}{extreme points} of $K$
and any $x \in \ext(K)$ is said to be \cemph{dred}{exposed}, if there exists a hyperplane
  $H$ supporting $K$, \st~$K \cap H = \{x\}$.

The euclidean unit ball and unit sphere are denoted by $\B,\S \subset \R^n$,
respectively, and
the closed (open) \cemph{dred}{semisphere} $\{x\in\S : u^Tx \ge 0\}$ ($\{x\in\S : u^Tx > 0\}$)
with $u\in\R^n\setminus \{0\}$, by $\S_u^\ge$ ($\S_u^>$).
By $\dist{A}{B}$ we denote the usual euclidean distance between two closed sets $A,B \subset \R^n$,
and write $\dist{a}{B}$ or $\dist{A}{b}$ if one of the sets is a singleton.

For a pair of bodies $K,L \in \CK^n$, the \cemph{dred}{Minkowski sum} of $K$ and $L$ is defined as
\[K+L:=\{x_1+x_2:x_1\in K,x_2\in L\},\]
and $\lambda K:=\{\lambda x:x\in K\}$ with $\lambda\in\R$ is the $\lambda$
\cemph{dred}{dilatation} of $K$. We abbreviate $K-L := K + (-L)$ and $K+x:=K+\{x\}$
for $x \in \R^n$.
A body $K$ is \cemph{dred}{$0$-symmetric} if $K = -K$ and \cemph{dred}{centrally symmetric} if
there exists $c \in \R^n$, \st~$K-c$ is 0-symmetric.

For any body $K\in\CK^n$, a \cemph{dred}{completion} of $K$ is defined
as a set $C_K$ satisfying $K \subset C_K$ and $D(K)=D(C_K)=w(C_K)$.
Moreover, $C_K$ is called a \cemph{dred}{Scott-completion} of $K$ if, in addition, $R(K)=R(C_K)$
(it was shown in \cite{Sc81} that in euclidean space such a completion always exists).

\medskip

We go on with a little series of well known propositions, which we will use later. The first
collects results taken from \cite{GK}:

\begin{prop}\label{prop:exposed}
  For any $K \in \CK^n$
  \begin{enumerate}[a)]
  \item every diametral pair of points in $K$ is a pair of exposed
    (and therefore extreme) points in $K$.
  \item every pair $L_1,L_2$ of parallel supporting hyperplanes of $K$
    at distance $w(K)$, supports a segment with endpoints in $K\cap L_1$ and $K\cap L_2$
        perpendicular to both hyperplanes. Moreover, if $K\in\CK^n$ is a polyhedron, then
        $\dim(K\cap L_1) + \dim(K\cap L_2) \ge d-1$ (which in case of $n=2$ means that
        at least one of the intersections $K\cap L_i$, $i=1,2$ contains a segment).
    \end{enumerate}
\end{prop}

The first part of the following proposition about the euclidean outer radius was shown already in
\cite{BF}. For the part about the inner radius we refer to the general optimality conditions
for containment under homothetics given in \cite{BrK}.

\begin{prop}\label{prop:CONT}
  Let $K\in\CK^n$ and $c\in\R^n$, \st~$c+\rho\B \subseteq K\subset \B$. Then
  \begin{enumerate}[a)]
  \item $R(K)=1$, iff there exist $k \in \{2,\dots,n+1\}$ and
    $p^1,\dots,p^k \in \bd(K\cap \S)$ \st~$0\in\conv\{p^1,\dots,p^k\}$.
  \item $r(K)=\rho$, iff there exists $l \in \{2,\dots,n+1\}$, $q^1,\dots,q^l \in \bd(K-c) \cap \rho\S$
    and $u^1,\dots,u^l \in\S$, \st~$(u^i)^Tq^i = \rho$, $i \in [l]$,
    $K-c\subset \bigcap_{i \in [l]}\{x\in\R^n:(u^i)^Tx \le \rho \}$, and $0\in\conv\{u^1,\dots,u^l\}$.
  \end{enumerate}
\end{prop}

The last proposition is ancient and best known as the \enquote{inscribed angle theorem}:

\begin{prop}\label{lem:angle}
  For any triangle $T:=\conv\{p^1,p^2,p^3\}$ with $p^1,p^2,p^3 \in \S$ and 0 and $p^3$ on
  the same side of $\aff\{p^1,p^2\}$
  let $\gamma$ denote the angle of $T$ at $p^3$. Then the angle
  of the triangle $\conv\{p^1,p^2,0\}$ at 0 is $2 \gamma$ (independently of the position of $p^3$).

  Moreover, if $p^3 \in \inte(\B)$ (or $p^3 \notin \B$), but still
  on the same side of $\aff\{p^1,p^2\}$ than 0,
  the angle in 0 is strictly smaller (or, respectively, strictly greater) than $2\gamma$.
\end{prop}

\medskip

In Section \ref{s:diagram} the way we proceed for the description of
the whole diagram is explained.

In Section \ref{s:main ineq} we present a collection of nine (generally)
valid inequalities
completely describing the diagram in Section \ref{s:diagram}.
Six of these inequalities were known before but three of them are totally new, relating
all four basic radii at once (in a non-trivial way).

Every family of sets attaining equality in one of the
inequalities above is mapped onto a compact connected subset of a $2$-dimensional differential
manifold within $\R^3$.
We call them \cemph{dred}{facets} of the diagram. Moreover, the
common boundaries between any two facets are called \cemph{dred}{edges} of the diagram
and the bodies appearing in the intersection of three (or more) facets
are called \cemph{dred}{vertices} of the diagram.
The families forming the facets and edges, as well as all vertices are described
in Section \ref{s:skeleton}.

Section \ref{s:proofs} is devoted to the proofs of the new inequalities presented in Section \ref{s:main ineq},
while the paper is finished with a short outview in Section \ref{s:outview}.

\section{Main ideas for explaining the diagram}\label{s:diagram}

Following the idea of Blaschke and Santal\'o, we define
\begin{equation}\label{def:function}
f : \CK^n \to [0,1]^3,\quad
f(K)=\left(\frac{r(K)}{R(K)},\frac{w(K)}{2R(K)},\frac{D(K)}{2R(K)}\right)
\end{equation}
and call $f(\CK^n)$ a \cemph{dred}{3-dimensional Blaschke-Santal\'o diagram}.
The following Lemma is taken from \cite{Br}:

\begin{lem} \label{lem:starshaped}
 $f(\CK^n) = f\left(\{K \in \CK^n : \B \text{ is the circumball of } K\}\right)$
 is starshaped with respect to $f(\B)=(1,1,1)$.
\end{lem}

\begin{proof}
  If $K \in \CK^n$, $c \in \R^n$, $\lambda \in [0,1]$ and $q$ any of the four radii functionals
  $r,w,D,R$,
  it obviously holds $q(\lambda (K+c)) = \lambda q(K)$ and
  $q(\lambda K+(1-\lambda)\B)=\lambda q(K)+(1-\lambda)q(\B)$.
\end{proof}

The latter part of Lemma \ref{lem:starshaped} means that the diagram has no \enquote{holes}
and therefore it suffices to describe the sets $K \in \bd(\CK^n)$, with
$\B$ being their circumball, mapped to the boundary of the diagram.

\begin{lem} \label{lem:Completion}
Let $K,C_K\in\CK^n$, \st $C_K$ is a completion of $K$
and $K_{\lambda}:=\lambda K+(1-\lambda)C_K$, $\lambda \in [0,1]$.
Then $D(K_{\lambda})=D(K)$ and $w(K_{\lambda})=\lambda w(K) + (1-\lambda)w(C_K)$ for all $\lambda\in[0,1]$.
\end{lem}

\begin{proof}
  Since $K\subset C_K$ and $D(K)=D(C_K)$ it immediately follows from $K \subset K_{\lambda} \subset C_K$ that
    $D(K_{\lambda})=D(K)$ for every $\lambda\in[0,1]$.
    Now, let $s^* \in \S$, \st~the breadth $b_{s*}(K)$ of $K$
    in direction of $s^*$ is $b_{s^*}(K)=w(K)$.
    Since $C_K$ is of constant width and because the breadth is linear with respect to the Minkowski sum,
    we obtain
    \begin{equation*}
      \begin{split}
        w(K_{\lambda})&=\min_{s\in\S}b_s(K_{\lambda})=\min_{s\in\S}(\lambda b_s(K) + (1-\lambda)w_s(C_K)) \\
        &\le \lambda b_{s^*}(K) + (1-\lambda) w(C_K) = \lambda w(K) + (1-\lambda) w(C_K) \le w(K_{\lambda}).
      \end{split}
    \end{equation*}
\end{proof}

\begin{lem} \label{lem:CompletionsW=r+R}
  For any $K\in\CK^n$ satisfying the left inequality in \eqref{eq:2} with equality (\ie~$w(K)=r(K)+R(K)$),
  $C_K$ being its Scott-completion,
  and $K_{\lambda}:=\lambda K+(1-\lambda)C_K$, $\lambda\in[0,1]$, it holds
  $f(K_{\lambda})=\lambda f(K)+(1-\lambda)f(C_K)$
  and therefore $w(K_{\lambda})=r(K_{\lambda})+R(K_{\lambda})$ for all $\lambda\in[0,1]$.
\end{lem}

\begin{proof}
  Using Lemma \ref{lem:Completion} it immediately follows that
  $D(K_{\lambda})=D(K)$, $R(K_\lambda)=R(K)$, and $w(K_{\lambda})=\lambda w(K) + (1-\lambda)w(C_K)$
  for all $\lambda\in[0,1]$.

  In \cite{Br} it was shown that $w(K)=r(K)+R(K)$ implies that $K$
  has a unique inball being concentric with the circumball.

  Hence, assuming as always that $\B$ is the circumball of $K$, we have $r(K)\B\subset K \subset \B$
  with $R(K)=1$.

  Observe that if $s\in\S$ with $-r(K)us\in\bd(K)$, then $s\in K$.
  This follows because the (unique) supporting hyperplane
  in $-r(K)s$ of $K$ has to support $r(K)\B$ too, and therefore
  this hyperplane has to be $-r(K)s+\lin\{s\}^{\bot}$. Thus
  the breadth of $K$ in the direction $s$ is at most $r(K)+R(K)$,
  with equality iff $s\in K$.

  Using Proposition \ref{prop:CONT}
  there exist $u^1,\dots,u^j\in\S$, $2\le j\le n+1$,
  \st~$-r(K)u^i\in\bd(K)$, $i \in [j]$ with $0\in\conv\{u^1,\dots,u^j\}$, which together with
  the observation above yields that
  $u^i\in K\cap\S$ and that the hyperplanes
  $-r(K)u^i+\lin\{u^i\}^{\bot}$ support $K$ in $-r(K)u^i$, $i \in [j]$.

  Now, since $C_K$ is a Scott-completion of $K$ it obviously holds
  $r(C_K)\B \subset C_K\subset \B$, $u^i\in C_K$, and since $w(C_K)=r(C_K)+R(C_K)$ also
  that the hyperplanes $-r(C_K)u^i + \lin\{u^i\}^{\bot}$
  support $C_K$ and its inball in $-r(C_K)u^i$, $i \in [j]$.

  Altogether we obtain $-r(K)u^i+\lin\{u^i\}^{\bot}$ and
  $-r(C_K)u^i+\lin\{u^i\}^{\bot}$ support $K$ and $C_K$
  in the points $-r(K)u^i$ and $-r(C_K)u^i$, respectively. Hence
  the hyperplanes $-(\lambda r(K)+(1-\lambda)r(C_K))u^i+\lin\{u^i\}^{\bot}$
  support $\lambda K+(1-\lambda)C_K$ in the points $-(\lambda r(K)+(1-\lambda)r(C_K))u^i$,
  $i \in [j]$. Using again Proposition \ref{prop:CONT}, it follows
  $r(\lambda K+(1-\lambda)C_K)=\lambda r(K)+(1-\lambda)r(C_K)$ and from this the rest of the statement
  in the lemma.
\end{proof}

As mentioned in Section \ref{s:intro} the inequalities
\begin{equation}\label{eq:rdR diagram}
  \begin{split}
    &D(K) \le 2R(K), \quad D(K) \ge r(K) + R(K), \quad D(K) \ge \sqrt{3} R(K), \quad \text{and}\\
    &2R(K)\left(2R(K)+\sqrt{4R(K)^2-D(K)^2}\right)r(K) \ge D(K)^2\sqrt{4R(K)^2-D(K)^2},
  \end{split}
\end{equation}
give a full description of \[ g:\CK^2 \to [0,1]^2, \text{ with }
g(K):=\left(\frac{r(K)}{R(K)},\frac{D(K)}{2R(K)}\right)\] (see Figure \ref{fig:2dimdiagram}).

\begin{figure}
  \begin{center}
    \includegraphics[trim = 0cm 9.5cm 0cm 8cm, width=14cm]{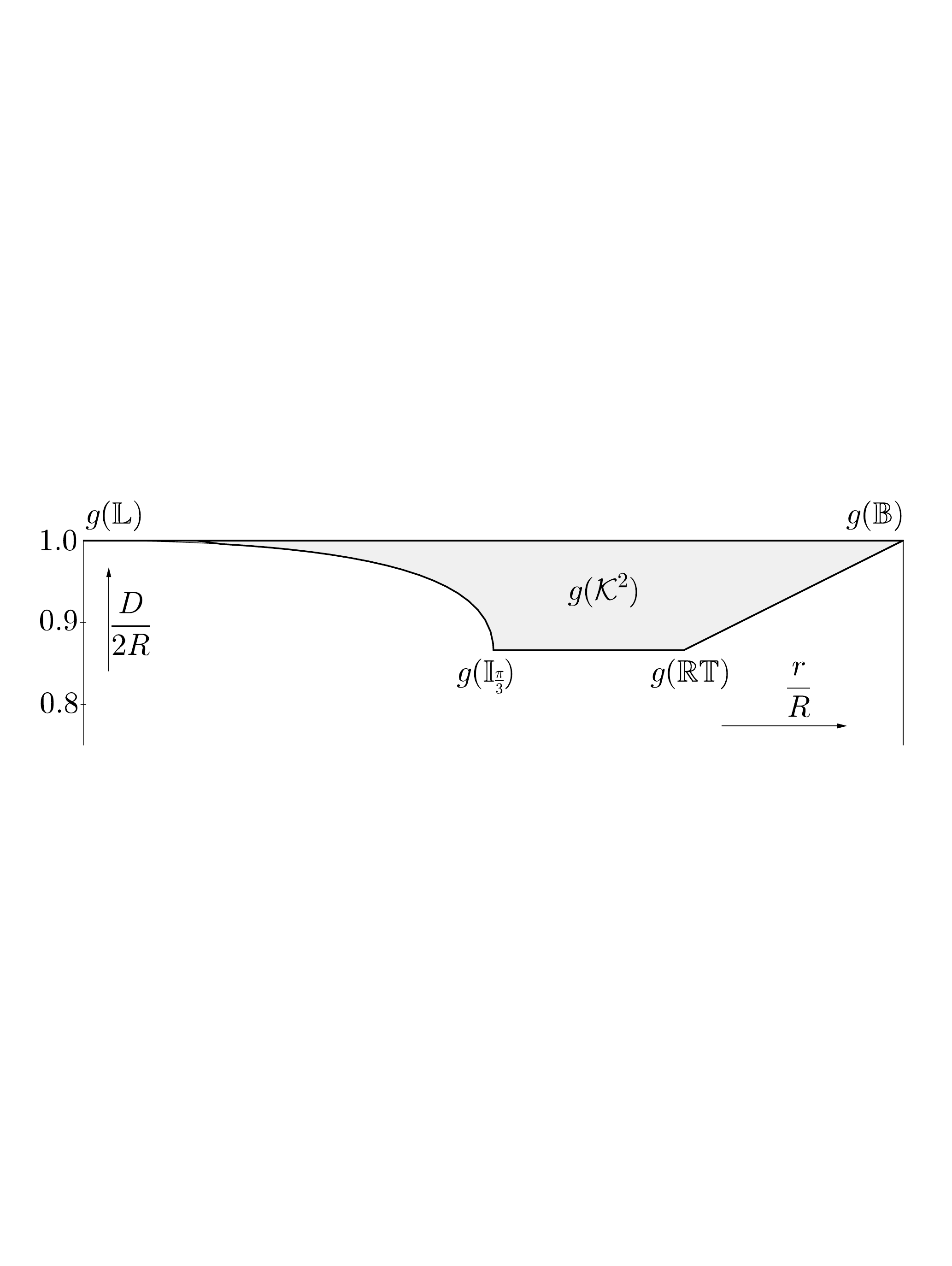}
    \caption{The diagram $g(\CK^2)$ with $x$-axis $\nicefrac{r}{R}$ and $y$-axis
      $\nicefrac{D}{2R}$. The boundaries are given via the inequalities
      collected in \eqref{eq:rdR diagram}.
      The vertices are the euclidean ball $\B$,
      the line segment $\L$, the equilateral triangle $\EqT$ and the Reuleaux triangle $\ReT$
      (see Subsection \ref{ss:vertices} for their explanation).}\label{fig:2dimdiagram}
  \end{center}
\end{figure}

\begin{figure}
\begin{center}
\includegraphics[trim = 0mm 11.5cm 0mm 11.5cm, width=16cm]{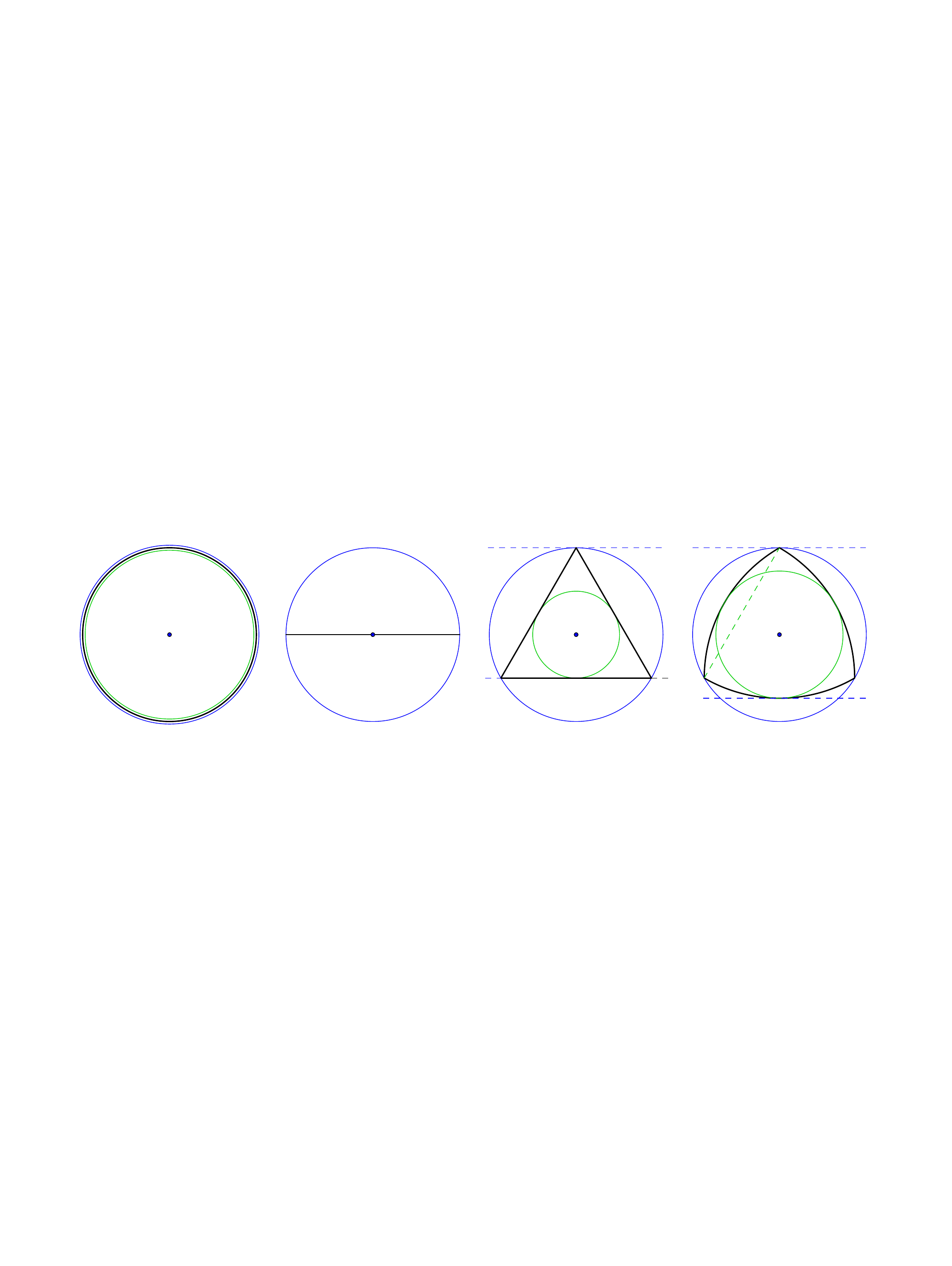}
\caption{From left to right: the euclidean ball $\B$, the line $\L$,
  the equilateral triangle $\EqT$, and the Reuleaux triangle $\ReT$.
  Here and in all the forthcoming figures, the inballs are drawn
    in green, the circumballs in blue, the diameters in dashed green, and the widths in dashed blue.}
  \label{fig:fourVertices}
\end{center}
\end{figure}

Since $g(\CK^2)$ is just the projection of $f(\CK^2)$ onto the first and last coordinate,
we may consider any valid pair of values $(r,D) \in g(\CK^2)$
and solve
\begin{equation*}
  \begin{array}{crcl}
    \max_{K\in\CK^2} & w(K) \\
    & r(K) & = & r \\
    & D(K) & = & D \\
    & R(K) & = & 1
  \end{array}
  \qquad \text \quad \text{as well as} \quad \text \qquad
  \begin{array}{crcl}
    \min_{K\in\CK^2} & w(K) \\
    & r(K) & = & r \\
    & D(K) & = & D \\
    & R(K) & = & 1 \ .
  \end{array}
\end{equation*}
Calling the solution of the maximization problem $w^*(r,D)$ for any
given pair $(r,D)$ and the solution of the minimization problem $w_*(r,D)$,
the family $\{ w^*(r,D) : (r,D) \in g(\CK^2) \}$ describes
the upper boundary of $f(\CK^2)$ and $\{w_*(r,D) : (r,D) \in g(\CK^2)\}$
describes the lower boundary of $f(\CK^2)$.
    To complete the full
diagram it then suffices to check which of the inequalities in
\eqref{eq:rdR diagram} still describe facets of $f(\CK^2)$ (\ie there exists a
pair $(r,D) \in \bd(g(\CK^2))$, \st~the corresponding inequality is
fulfilled with equality and $w^*(r,D) \neq w_*(r,D)$) and which describe only
edges (which is the case if $w^*(r,D) = w_*(r,D)$ for all $(r,D) \in g(\CK^2)$
fulfilling the inequality with equality).

\section{Main inequalities} \label{s:main ineq}

In this section we describe nine valid inequalities.
Three of them are of the form $w \le w^*(r,D)$, thus describing the upper boundary of the diagram;
we call them $(ub_j)$, $j=1,2,3$. Analogously,
we write $(lb_j)$, $j=1,2,3$ for the three inequalities $w \ge w_*(r,D)$ (giving the lower boundary)
and $(ib_j)$, $j=1,2,3$ for the inequalities which are independent of $w$.

We start with those inequalities which are a-priori known:
\begin{prop}\label{prop:1}
  Let $K\in\CK^2$. Then
  \begin{align}
    \tag{$lb_1$}  \label{lb_1} 2r(K) &\le w(K) \\
    \tag{$ib_1$}  \label{ib_1} D(K) &\le 2R(K) \\
    \tag{$ub_1$}  \label{ub_1} w(K) &\le R(K)+r(K) \\
    \tag{$ib_2$}  \label{ib_2} R(K) + r(K) &\le D(K)\\
    \tag{$ib_3$}  \label{ib_3} \sqrt{3}R(K) &\le D(K) \\
    \tag{$lb_2$}  \label{lb_2} (4R(K)^2-D(K)^2)D(K)^4 &\le  4w(K)^2R(K)^4
  \end{align}
\end{prop}

The remaining three inequalities for a complete description of $f(\CK^2)$ are new.
Surely, each of them involves all four radii $r,w,D$, and $R$ simultaneously
as otherwise it would have been neccessary for the description of
the according 2-dimensional Blaschke-Santal\'o diagram.

\begin{thm} \label{thm:3}
  Let $K\in\CK^2$. Then
     \begin{equation}\tag{$lb_3$}\label{eq:3PT}
     \begin{split}
    w(K) & \ge 2D(K) \sqrt{1-\left(\frac{D(K)}{2R(K)}\right)^2} \\
    & \cos\left[\arccos\left(\frac{D(K)}{2(D(K)-r(K))}\right) +
          \arccos\left(\frac{D(K)}{2R(K)}\right)
    -\arcsin\left(\frac{r(K)}{D(K)-r(K)}\right)\right]
    \end{split}
    \end{equation}
\end{thm}

\begin{rem}
An algebraic representation of \eqref{eq:3PT}
can easily be calculated using a computer algebra tool and looks like the following:
\begin{equation*}
\begin{split}
w(K) &\ge 2D(K) \sqrt{1-\left(\frac{D(K)}{2R(K)}\right)^2} \left[ \sqrt{1-\frac{r(K)^2}{(D(K)-r(K))^2}} \right.\\
&      \left(\frac{D(K)^2}{4R(K)(D(K)-r(K))}-\sqrt{\left(1-\frac{D(K)^2}{4(D(K)-r(K))^2}\right)\left(1-\frac{D(K)^2}{4R(K)^2}\right)}\right)\\
&      +\left. \frac{r(K)}{D(K)-r(K)} \left(\frac{D(K)}{2R(K)}\sqrt{1-\frac{D(K)^2}{4(D(K)-r(K))^2}}-\frac{D(K)}{2(D(K)-r(K))}\sqrt{1-\frac{D(K)^2}{4R(K)^2}}\right) \right]
\end{split}
\end{equation*}
\end{rem}

\begin{thm}\label{thm:2}
  Let $K\in\CK^2$. Then
  \begin{equation} \tag{$ub_2$} \label{eq:sailing boats}
    w(K) \le r(K) \left( 1 + \frac{2\sqrt{2}R(K)}{D(K)}
      \sqrt{1 + \sqrt{1-\left(\frac{D(K)}{2R(K)}\right)^2}}\right).
  \end{equation}
\end{thm}

\begin{rem} \label{rem:steinhagen}
One may recognize that
\[1 + \frac{(2\sqrt{2}R(K))}{D(K)} \sqrt{1 + \sqrt{1-\left(\frac{D(K)}{2R(K)}\right)^2}} \leq 3,\]
which shows that \eqref{eq:sailing boats} is a direct strengthening of the 2-dimensional version
of Steinhagen's inequality (\cf~\cite{St21}). However, this is not the case for \eqref{eq:triangles}
(even so containing $\L$ and $\EqT$, the two sets fulfilling Steinhagen's inequality with equality) as,
\eg~evaluating
\eqref{eq:triangles} at $\B$ gives a value obviously bigger than 3.

It is also quite easy to see that for a pendant of our diagram for higher dimensional sets Steinhagen's inequality
induces a facet.
\end{rem}

\begin{thm} \label{thm:4}
Let $K\in\CK^2$. Then
  \begin{equation} \tag{$ub_3$} \label{eq:triangles}
    w(K) \le 2r(K) \left(1 + \frac{2r(K)R(K)}{D(K)^2}
      \left(1 + \sqrt{1-\left(\frac{D(K)}{2R(K)}\right)^2}\right)
      \right)
  \end{equation}
\end{thm}

\section{The skeleton of the 3-dimensional diagram}\label{s:skeleton}

This section is devoted to the description of the families of bodies filling the faces
of $\bd(f(\CK^2))$.

We start in Subsection \ref{ss:vertices}
describing the sets fulfilling three or more inequalities with equality,
the vertices of $\bd(f(\CK^2))$.
In Subsection \ref{ss:edges} we discuss the families of sets fulfilling two inequalities with equality,
the edges of $\bd(f(\CK^2))$.
Finally, in Subsection \ref{ss:facets} the sets filling the different facets
are explained. For the description of these sets we always assume that $\B$ is the circumball,
but for a better understanding of the geometric inequalities we will keep the value $R(K)$ in each description.

In case there does not exist a unique set, which is mapped to a boundary point of
the diagram, we will usually describe in some way the range of sets mapped to that point,
\eg by giving maximal and minimal sets (with respect to set inclusion) if appropriate. However,
as this is not our major topic, we neither claim completeness nor present proper proofs.

\subsection{Vertices of the diagram}\label{ss:vertices}

The vertices, including their radii, and for each the inequalities which are fulfilled with
equality are listed in the following:

\begin{itemize}
\item[$\B$] Obviously, the \cemph{dblue}{euclidean ball} $\B$ is the unique set mapped to
  $f(\B)=(1,1,1)$ in the diagram. It is extreme for the inequalities \eqref{lb_1}, \eqref{ib_1},
  \eqref{ub_1}, and \eqref{ib_2}.

\item[$\L$] The radii of the \cemph{dblue}{line segment} $\L$ are easy to see, too:
  $f(\L)=(0,0,1)$ and also it is the only set mapped to this coordinates.
  The inequalities it fulfills with equality are \eqref{lb_1}, \eqref{lb_2}, \eqref{ib_1},
  and \eqref{eq:triangles}. It also fulfills \eqref{eq:sailing boats} with equality,
  but this is an artefact which will be explained in Remark \ref{artefact}.

\item[$\EqT$] The radii of the \cemph{dblue}{equilateral triangle} $\EqT$ are well known:
  $f(\EqT)=(\nicefrac 1 2, \nicefrac 3 4, \nicefrac {\sqrt{3}} 2)$. It is the unique set with these
  coordinates and extreme for the
  inequalities \eqref{ub_1}, \eqref{eq:sailing boats}, \eqref{eq:triangles}, \eqref{lb_2}, and \eqref{ib_3}.

\item[$\ReT$] The \cemph{dblue}{Reuleaux triangle} $\ReT$ is the intersection of
  three euclidean balls of radius $\sqrt{3}$ centered in the vertices of
  $\EqT$. On the one hand it has the same diameter and circumradius as
  $\EqT$. On the other hand it is of constant width, thus \eqref{lb_1} and \eqref{ib_2} imply
  $w(\ReT) = r(\ReT) + R(\ReT) = D(\ReT)$.
  Hence $f(\ReT)=(\sqrt{3}-1,\nicefrac {\sqrt{3}} 2,\nicefrac{\sqrt{3}} 2)$ and $\ReT$ is
  the unique set mapped to this point of the diagram. It is
  extreme for the inequalities \eqref{ub_1}, \eqref{ib_2} and \eqref{ib_3}.

\item[$\RAT$] The \cemph{dblue}{(isosceles) right-angled triangle} $\RAT$
    (for short we will sometimes ommit the term \enquote{isosceles})
    has diameter
    $D(\RAT)=2R(\RAT)$ and its width coincides with its height above the diameter edge,
    thus $w(\RAT)=R(\RAT)$. Using the semiperimeter formula for triangles, we obtain that the inradius is
  \[ r(\RAT) = \frac{D(\RAT) w(\RAT)}{D(\RAT) + 2\sqrt{2}R(\RAT)} =  \frac{R(\RAT)}{1+\sqrt{2}}
  = (\sqrt{2}-1) R(\RAT). \]
  Thus $f(\RAT)=(\sqrt{2}-1,\nicefrac 1 2,1)$ and there is no different $K$ mapped to this coordinates
  (as one may easily see in proceeding the construction of a set mapped to this coordinates).
  $\RAT$ is extreme for the inequalities \eqref{eq:sailing boats}, \eqref{eq:triangles} and \eqref{ib_1}.

\item[$\SB$] The \cemph{dblue}{(right-angled concentric) sailing boat} $\SB$ is the
  intersection of $\B$ and a homothetic of $\RAT$ with incenter at 0 and a vertex $v$ located where
  the two edges of equal length intersect on $\S$ (see Figure \ref{fig:SB}).
  Hence the in- and circumball of $\SB$ are concentric and one
  can easily see from the construction, that
    $\nicefrac12 \, D(\SB) = \sqrt{2} r(\SB) =  R(\SB)$.
  Its width is attained in any of the orthogonal directions to any three of the edges of $\RAT$.
  Especially from measuring between $v$ and its opposite edge, we obtain
  \[w(\SB) = r(\SB) + R(\SB) = \left(\nicefrac 1 {\sqrt{2}} + 1 \right) R(\SB). \]
  Thus $f(\SB) = \left(\nicefrac 1 {\sqrt{2}}, \nicefrac 1 2 (\nicefrac 1 {\sqrt{2}} + 1) ,1\right)$.
  The sailing boat fulfills inequalities \eqref{ub_1}, \eqref{eq:sailing boats} and \eqref{ib_1} with equality.
  Finally, denoting the (circumspherical) pentagon formed from the five vertices of $\SB$ by
  $\mathds{CP}$, we obtain
  $f(K)=f(\SB)$ for any $K \in \CK^2$, iff $\mathds{CP} \subset K \subset \SB$.

  \begin{figure}
    \centering
    \begin{subfigure}[b]{0.45\textwidth}
      \includegraphics[trim = 1cm 3.5cm 1cm 2cm, width =\textwidth]{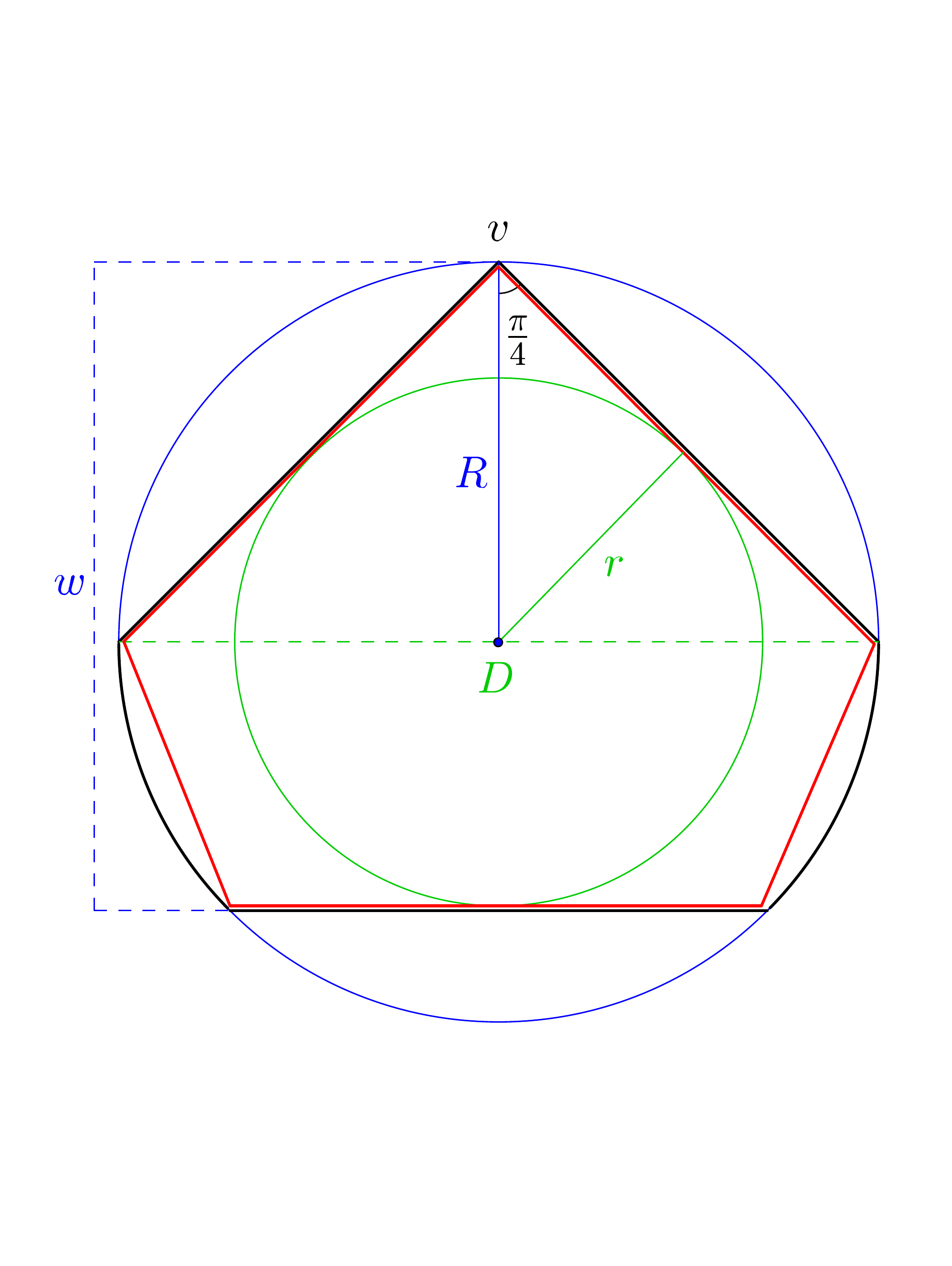}
      \caption{The sailing boat $\SB$ in black and the pentagon $\mathds{CP}$
        (sharing the vertices with $\SB$) in red.}
      \label{fig:SB}
    \end{subfigure} \hfill 
    \begin{subfigure}[b]{0.45\textwidth}
      \includegraphics[trim = 1cm 3cm 1cm 1.5cm, width = 0.93\textwidth]{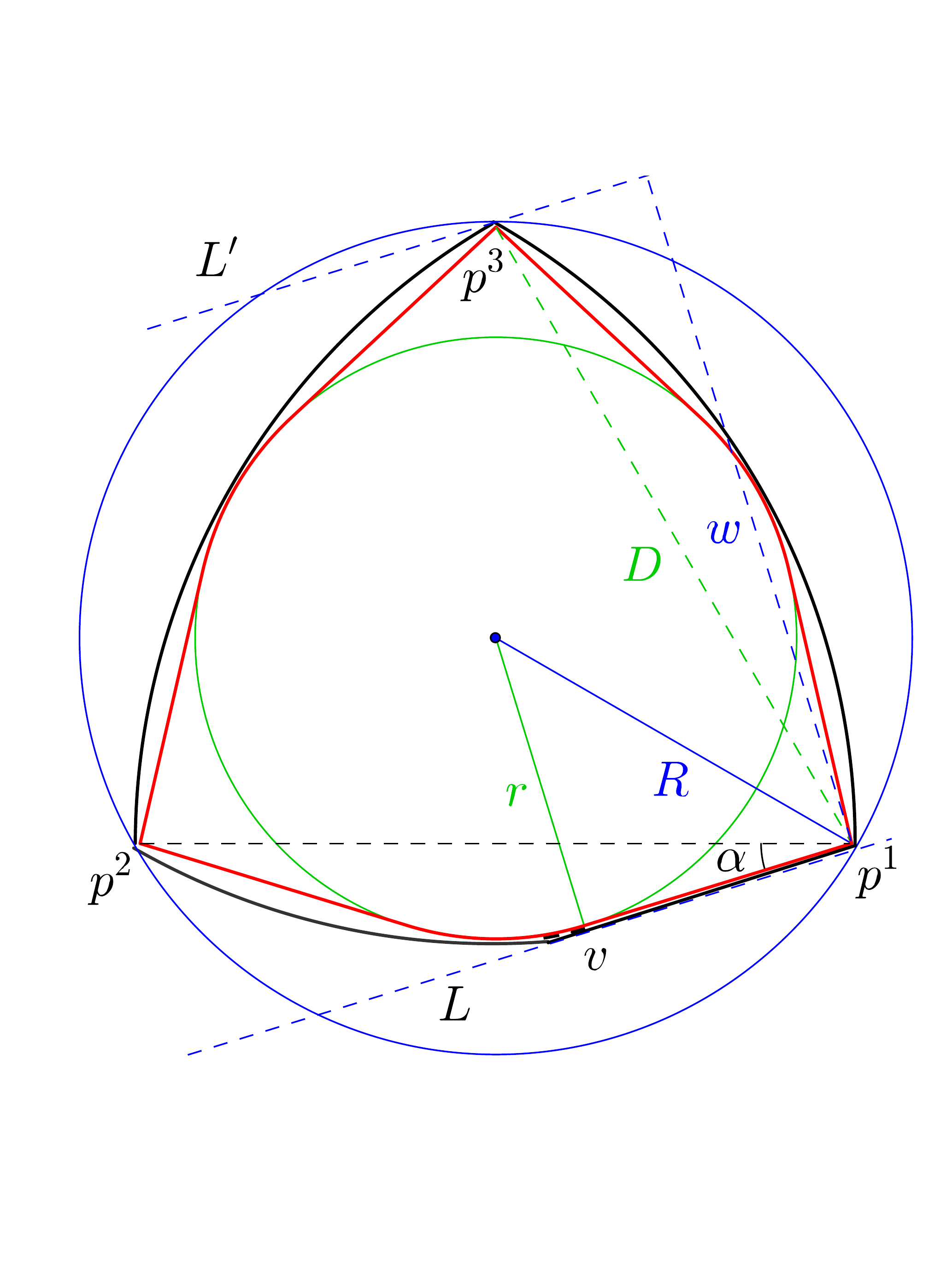}
      \caption{The sliced Reuleaux triangle $\SRT$ in black
      and $\SRT_{\min}$ (the minimal
      set sharing all radii with $\SRT$) in red.}
      \label{fig:SRT}
    \end{subfigure}
    \caption{Sailing boats and sliced Reuleaux triangles.}
    \label{fig:SB+SRT}
  \end{figure}

\item[$\SRT$] The \cemph{dblue}{sliced Reuleaux triangle} $\SRT$ is
  the intersection of a Reuleaux triangle $\ReT$ and a halfspace $H$ which supports a vertex
  of $\ReT$, say $p^1$, and its inball in a point $v$ (see Figure \ref{fig:SRT}).
  By construction it keeps the same diameter,
  in- and circumradius as $\ReT$. The width of $\SRT$ is attained between the parallel lines $L=\bd(H)$,
  and $L'$ supporting $\SRT$ in the vertex $p^2$, which is furthest from $v$.

  Defining $\alpha$ to be the angle between $L$ and the line segment $[p^2,p^3]$,
  where $p^3$ is the remaining vertex of $\EqT$, one easily computes
  \[r(\SRT) =  R(\SRT) \sin\left(\nicefrac{\pi}{6}+\alpha\right) \quad \text{and} \quad
  w(\SRT) =  D(\SRT) \cos\left(\nicefrac{\pi}{6}-\alpha\right). \]
  Hence $\alpha=\arcsin\left(\nicefrac{r(\SRT)}{R(\SRT)}\right) - \, \nicefrac \pi 6$ and thus
  \[w(\SRT) =  D(\SRT) \cos\left(\nicefrac{\pi}{3}-\arcsin\left(\nicefrac{r(\SRT)}{R(\SRT)}\right)\right).\]
  We obtain $f(\SRT) = \left(\sqrt{3}-1, \nicefrac {\sqrt{3}} 2 \cos(\nicefrac \pi 3-\arcsin(\sqrt{3}-1)),
  \nicefrac{\sqrt{3}} 2\right)$
  and extremality for the inequalities \eqref{eq:3PT}, \eqref{ib_2} and \eqref{ib_3}.

  Denoting by $\SRT_{\min}$ the convex hull of the vertices and the inball of $\ReT$,
    one may easily verify that $f(K)=f(\SRT)$ for any $K \in \CK^2$, iff
  $\SRT_{\min} \subset K \subset \SRT$.

\item[$\FRT$] Let $\FRT$ be the \cemph{dblue}{flattened Reuleaux triangle}, obtained by replacing
  two of the three edges of $\EqT$ by the according arcs of $\ReT$.
  It has the same circumradius, diameter and width than $\EqT$ and
  defining $a$ to be the distance from the center of the inball to each vertex incident with the linear
  edge, it follows $a^2=r(\FRT)^2 + \nicefrac{1}{4} \, D(\FRT)^2$ and $D(\FRT) = a + r(\FRT)$
  (see Figure \ref{fig:FRT}).
  Hence $4(D(\FRT) - r(\FRT))^2 = 4r(\FRT)^2 + D(\FRT)^2$ and,
  after dividing by $D(\FRT)$, we obtain $3D(\FRT) = 8 r(\FRT)$.
  Thus $f(\FRT)=(\nicefrac {\sqrt{27}} 8, \nicefrac 3 4, \nicefrac {\sqrt{3}} 2)$
  and equality holds in the inequalities \eqref{lb_2}, \eqref{eq:3PT} and \eqref{ib_3}.

  Denoting the convex hull of $\EqT$ and the inball of $\FRT$
  by $\FRT_{\min}$, it holds $f(K)=f(\FRT)$, iff $\FRT_{\min} \subset K \subset \FRT$
  (see Figure \ref{fig:FRT}).

  \begin{figure}
    \centering
    \begin{subfigure}[b]{0.45\textwidth}
      \includegraphics[trim = 1cm 3.5cm 1cm 2cm, width =\textwidth]{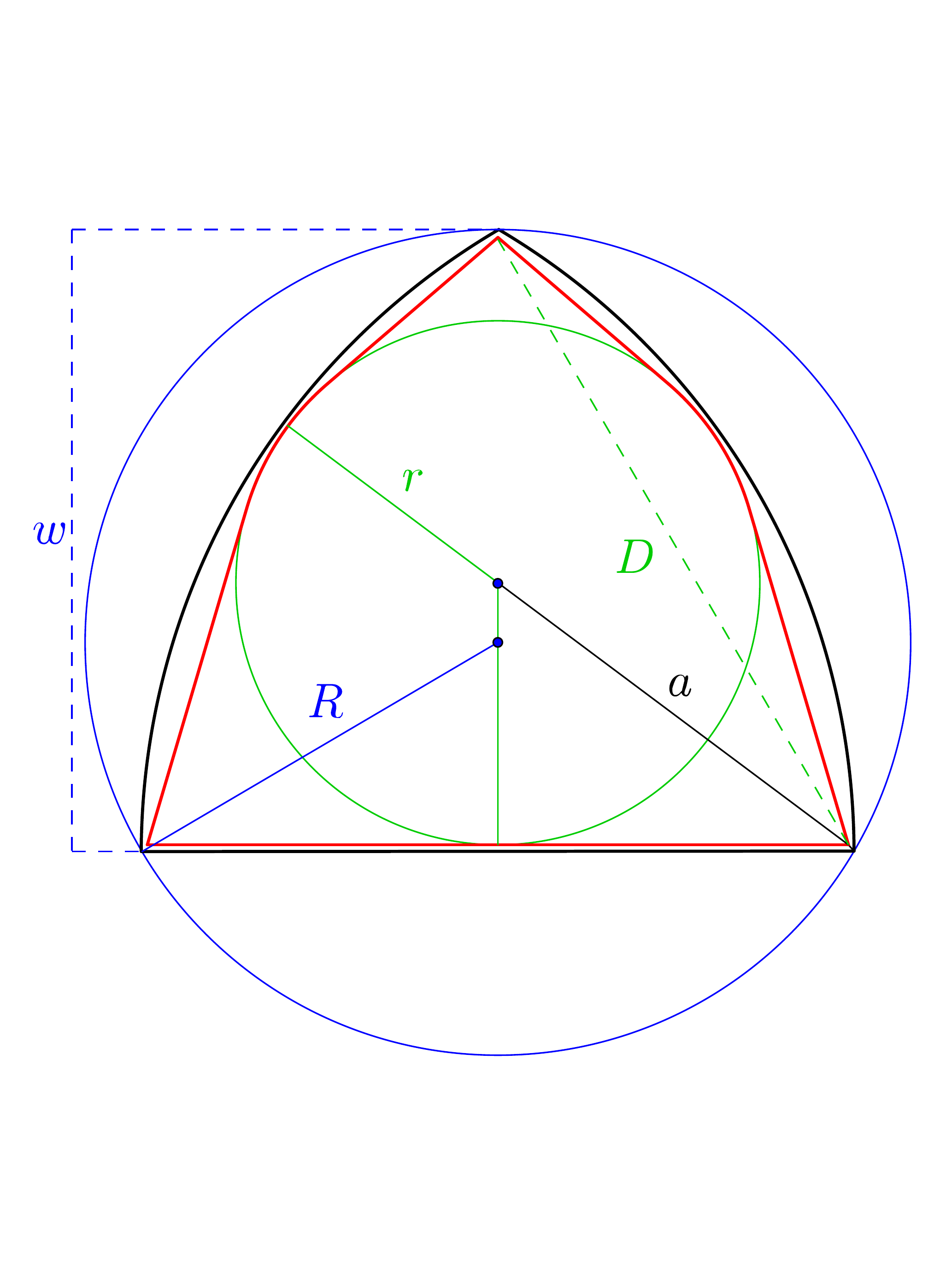}
      \caption{In black $\FRT$, in red $\FRT_{\min}$.}
      \label{fig:FRT}
    \end{subfigure} \hfill 
    \begin{subfigure}[b]{0.45\textwidth}
      \includegraphics[trim = 1cm 3.5cm 1cm 2cm, width = \textwidth]{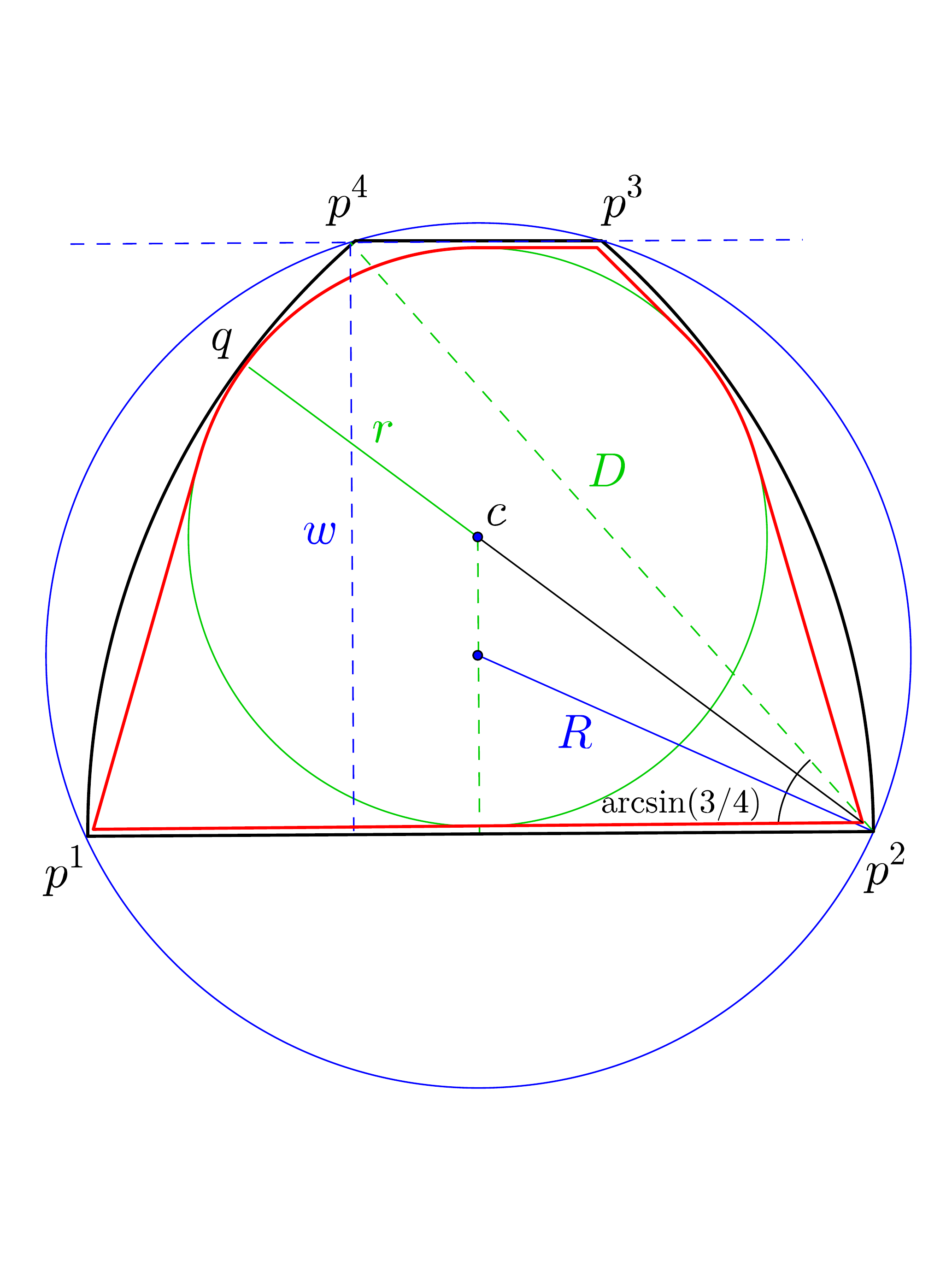}
      \caption{In black $\BT$, in red $\BT_{\min}$.}
      \label{fig:BT}
    \end{subfigure}
    \caption{The flattened Reuleaux triangle and the bent trapezoid.}
  \end{figure}

\item[$\BT$] Let $p^1,p^2,p^3,p^4 \in \S$ be, \st~$\conv\{p^1,p^2,p^3,p^4\}$ is a trapezoid
  with $[p^1,p^2]$ the longer and $[p^3,p^4]$ the shorter parallel line, and \st~$\conv\{p^1,p^2,p^3\}$
  as well as $\conv\{p^1,p^2,p^4\}$ are isosceles triangles (the first
  with $[p^1,p^2]$, $[p^1,p^3]$ the edges of equal length, the second with $[p^1,p^2]$, $[p^2,p^4]$)
  and $\arcsin(\nicefrac{3}{4})$ in both cases the angle between the two equal edges
  (see Figure \ref{fig:BT}). We write $\Iso[\arcsin(\nicefrac{3}{4})]$ for such an
  isosceles triangle (\cf Subsection \ref{ss:edges}).
  Substituting
  the edges $[p^1,p^4]$ and $[p^2,p^3]$ by two arcs of circumference of radius
  $\norm[p^1-p^2] = D(\Iso[\arcsin(\nicefrac{3}{4})])$ and centers $p^1$ and $p^2$, respectively,
  we obtain the \cemph{dblue}{bent trapezoid} $\BT$.

  By construction $\BT$ and
  $\Iso[\arcsin(\nicefrac{3}{4})]$ have the same width $w$, diameter $D$, and circumradius $R=1$.
  We prove that the inball of $\BT$ is tangent to the two parallels and the two arcs:
  Let $B$ be a ball of radius $r=\nicefrac12 \, w$ and center
  $c=\nicefrac14 (p^1+p^2+p^3+p^4)$, and denote the intersection point of the line through $p^2$
  and $c$ with the bow between $p^1$ and $p^4$ by $q$.
  We show that $\norm[c-q]=r$ which then implies $r(\BT)=r$:
    \[\text{(i)} \;~ \norm[c-q] = D - \norm[c-p^2], \quad \text{(ii)} \;~
    r^2+\nicefrac{1}{4} \, D^2 = \norm[c-p^2]^2, \quad \text{(iii)} \;~
    w=\nicefrac34 \, D. \]
    From (iii) we obtain $D=\nicefrac{8}{3} \, r$, and using (i) combined with (ii) gives
    \[
    \norm[c-q] = D - \sqrt{r^2 + \nicefrac{1}{4} \, D^2} = \nicefrac{8}{3} \, r -
    \sqrt{r^2+\nicefrac{16}{9} \, r^2} = r,
    \]
    as we wanted to show. Thus $r(\BT)=r=\nicefrac12 \, w$.

  For computing $D$,
    we use the fact that the line from $p^2$ through 0
    is the bisecting line of the angle $\arcsin(\nicefrac{3}{4})$ between $[p^1,p^2]$
    and $[p^2,p^4]$ at $p^2$ which means
    \[
    \frac{D}{2R} =\frac{D(\Iso[\arcsin(\nicefrac{3}{4})])}{2R(\Iso[\arcsin(\nicefrac{3}{4})])}
    = \cos\left(\frac 1 2 \arcsin\left(\frac{3}{4}\right)\right).
    \]
    This implies
    \[ D= 2\cos\left(\nicefrac{1}{2}\arcsin(\nicefrac{3}{4})\right)R
    = \sqrt{2+\nicefrac{\sqrt{7}}{2}} \, R,\]
    and using the above properties on the radii of $\BT$ we obtain that
    \[ 2r=w=\nicefrac{3}{4} \, D=\nicefrac{3}{4}\sqrt{2+\nicefrac{\sqrt{7}}{2}} \, R.\]

  Hence
  \[
  f(\BT)=\left(\nicefrac{3}{8}\sqrt{2+\nicefrac{\sqrt{7}}{2}},
    \nicefrac{3}{8}\sqrt{2 + \nicefrac{\sqrt{7}}{2}},
    \nicefrac{1}{2}\sqrt{2 + \nicefrac{\sqrt{7}}{2}}\right),
  \]
  and one may easily check that it fulfills the inequalities \eqref{lb_1}, \eqref{lb_2},
  and \eqref{eq:3PT} with equality.

  Denoting the convex hull of $p^1,p^2,p^3$ and $B$ by $\BT_{\min}$,
  it holds $f(K)=f(\BT)$, iff $\BT_{\min} \subset K \subset \BT$
  (\cf~Figure \ref{fig:BT}).

\item[$\H$] The last vertex satisfies \eqref{lb_1}, \eqref{ib_2} and \eqref{eq:3PT} with equality, whereby its shape
  is determined as follows: from \eqref{lb_1} there must exist two parallel
  lines supporting the inball of the set and because of \eqref{ib_2} it must have concentric in-
  and circumball.
  The two parallels supporting the inball contain two separated arcs
  of the circumsphere between them.
  Let $p^1,p^2,p^3$ be points, \st~$p^2$ and $p^3$ lie in one
  arc and each in one of the supporting lines, while
  $p^1$ lies in the other arc and at the same distance from $p^2$ and $p^3$.
  Finally, we connect $p^2$ and $p^3$ by an arc centered
  in $p^1$, its radius as well as the
  radius of the inball chosen, \st~the arc is tangent to the inball (\cf Figure \ref{fig:H}).
  The convex set bounded by the two supporting parallel lines and the three arcs
  with centers $p^1,p^2,p^3$ and radius $\norm[p^1-p^2]$
  is called the \textcolor{dblue}{hood} and denoted by $\H$.

  Remember that we always assume $0$ to be the circumcenter and let $\gamma$ be
  \st~$\Iso=\conv\{p^1,p^2,p^3\}$ is the isosceles triangle built by $p^1,p^2,p^3$. Thus
    $R(\H)=R(\Iso)$, $D(\H)=D(\Iso)=r(\H)+R(\H)$
  and $2r(\H)=w(\H)$.

  For the computation of $r(\H)$ let $\zeta$
  denote the distance from $0$ to $[p^2,p^3]$. Considering the two right-angled triangles
  $\conv\left\{0,p^2,\nicefrac{1}{2} \, (p^2+p^3)\right\}$ and
  $\conv\left\{p^1,p^2,\nicefrac{1}{2} \, (p^2+p^3)\right\}$
  we obtain
    \[\text{(i)} \quad  r(\H)^2+\zeta^2=R(\H)^2 \qquad \text{and} \qquad \text{(ii)}
    \quad D(\H)^2=(\zeta+R(\H))^2+r(\H)^2 \]
  (\cf Figure \ref{fig:H}).
  \begin{figure}
    \begin{center}
      \includegraphics[trim = 0mm 4cm 0mm 3cm, width=8cm]{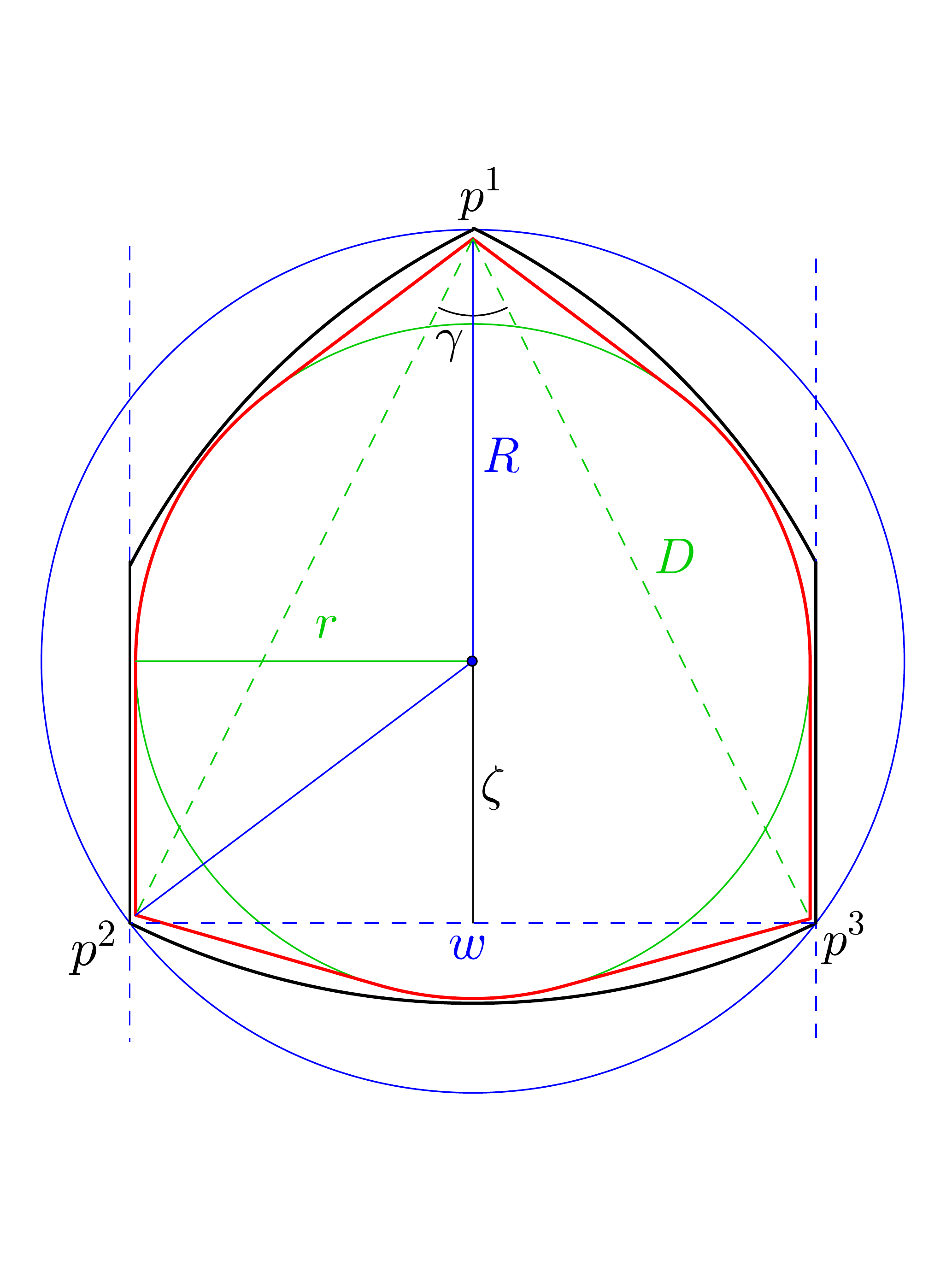}
      \caption{The hood $\H$ in black and $\H_{\min}$ in red.}
      \label{fig:H}
    \end{center}
  \end{figure}
  Solving (i) for $\zeta$ and inserting it into (ii), keeping into account that
  $D(\H)=r(\H)+R(\H)$, we obtain
  \[
  (r(\H)+R(\H))^2=D(\H)^2=(\sqrt{R(\H)^2-r(\H)^2}+R(\H))^2+r(\H)^2.
  \]
    Solving for $r(\H)$ gives the unique positive real solution
    \[
    r(\H)=\left(\frac12\sqrt{\varsigma+\xi}+\sqrt{-\varsigma-\xi+\frac{16}{\sqrt{\varsigma+\xi}}}-1\right)R(\H),
    \]
  where $\varsigma=\nicefrac13 \, (864-96\sqrt{69})^{\nicefrac13}$ and
  $\xi=2(\nicefrac23)^{\nicefrac23}(9+\sqrt{69})^{\nicefrac13}$.
  Thus
  \[
  f(\H)=\left(r(\H),r(\H),\nicefrac12(r(\H) + 1)\right) \approx (0.7935,0.7935,0.8967).
  \]
  Denoting the convex hull of the inball of $\H$ and $p^1,p^2,p^3$,
  by $\H_{\min}$,
  it holds $f(K)=f(\H)$, iff $\H_{\min} \subset K \subset \H$
  (\cf~Figure \ref{fig:H}).

\end{itemize}

\begin{table}[hbt]
  \centering 
  \begin{tabular}{l|c|c|c}
    Name & Symbol & Approximate Coordinates & $lb_{1,2,3} \ ib_{1,2,3} \ ub_{1,2,3} $ \\ \hline
    Ball & $\B$   & $(1,1,1)$
    & $+  -  -  +  +  -  +  -  -$ \\[2ex]
    Equilateral triangle & $\EqT$ & $(0.5,0.75,0.8660)$
    & $-  +  -  -  -  +  +  +  +$ \\[2ex]
    Line segment & $\L$ & $(0,0,1)$
    & $+  +  -  +  -  -  -  \pm  +$ \\[2ex]
    Reuleaux triangle & $\ReT$ & $(0.7321,0.8660,0.8660)$
    & $-  -  -  -  +  +  +  -  -$ \\[2ex]
    Right-angled triangle & $\RAT$ & $(0.4142,0.5,1)$
    & $-  -  -  +  -  -  -  +  +$ \\[2ex]
    Sailing boat & $\SB$ & $(0.7071,0.8536,1)$
    & $-  -  -  +  -  -  +  +  -$ \\[2ex]
    Sliced Reuleaux triangle & $\SRT$ & $(0.7321,0.8440,0.8660)$
    & $-  -  +  -  +  +  -  -  -$ \\[2ex]
    Flattened Reuleaux triangle & $\FRT$ & $(0.6495,0.75,0.8660)$
    & $-  +  +  -  -  +  -  -  -$ \\[2ex]
    Bent trapezoid & $\BT$ & $(0.6836,0.6836,0.9114)$
    & $+  +  +  -  -  -  -  -  -$ \\[2ex]
    Hood & $\H$ & $(0.7935, 0.7935, 0.8967)$
    & $+  -  +  -  +  -  -  -  -$\\
    \multicolumn{4}{c}{\text}
  \end{tabular}

  \caption{The table lists the planar sets mapped to vertices of the 3-dimensional
    Blaschke-Santal\'o diagram, their (approximate) radii, and the inequalities
    they fulfill with equality ($+$) or not ($-$).
    The $\pm$ for the line segment in the \eqref{eq:sailing boats}-column
    is explained in Remark \ref{artefact}}
  \label{table}
\end{table}

\begin{rem} \label{artefact}
Considering Table \ref{table} we may observe the following: all inequalities besides
\eqref{eq:triangles} are fulfilled with equality by
exactly four vertices. Moreover, since all three vertices of \eqref{eq:triangles} also
fulfill \eqref{eq:sailing boats} with equality (and since we will later prove these two inequalities
more or less within one proof), we may understand them as one inequality in two parts.
Doing so all inequalities are fulfilled by exactly four vertices,
a fact which in a polytopal setting would be quite exceptional. (To be honest, accepting the two inequalities
to be a joint one, the right-angled triangle would not be a vertex anymore due to our definition,
but nevertheless we think the whole matter is remarkable.)
\end{rem}

\subsection{Edges of the diagram}\label{ss:edges}

Next we give constructions of explicit families of convex sets mapped onto the intersection of two of the surfaces
obtained from the equality cases of the inequalities collected in Section \ref{s:main ineq}. In particular,
every family of sets $\{K_t\}_{t\in[t_1,t_2]}$ described, induces a closed differentiable curve $f(\{K_t:t\in[t_1,t_2]\})$
in $\R^3$. In our nomenclature they form the edges of the diagram. Each edge is named via its two endpoints, \eg
$(\EqT,\B)$ denotes the edge between $\EqT$ and $\B$.

\begin{itemize}
\item[$(\ReT,\B)$] It is a well known property that $w(K) = r(K) + R(K) = D(K)$, iff $K$ is of constant width.
  Thus all \emph{sets of constant width} fulfill \eqref{ub_1} and \eqref{ib_2} with equality.
  Essentially all edges with $\B$ as an endpoint are real linear edges of the diagram:
  because of Lemma \ref{lem:starshaped} we may pass the full edge from $\ReT$ to $\B$
  with \cemph{dblue}{rounded Reuleaux triangles}, \ie
  the outer parallel bodies $(1-\lambda) \ReT + \lambda \B$, $\lambda \in [0,1]$ of the Reuleaux triangle.

\item[$(\L,\B)$] Whenever $K$ is \emph{centrally symmetric} it satisfies the equations
  $D(K)=2R(K)$ and $w(K)=2r(K)$. Thus $f$ maps $K$ onto the linear edge formed from
  the equality cases of \eqref{lb_1} and \eqref{ib_1}. Again, because of Lemma \ref{lem:starshaped},
  the outer parallel bodies $(1-\lambda) \L + \lambda \B$, $\lambda \in [0,1]$, of $\L$
  (called \cemph{dblue}{sausages}) already fill the whole edge.

\item[$(\SB,\B)$] Lemma \ref{lem:starshaped} implies that all \cemph{dblue}{rounded sailing boats}
    $(1-\lambda) \SB + \lambda \B$, $\lambda \in [0,1]$
    satisfy the inequalities \eqref{ub_1} and \eqref{ib_1} with equality and fill the corresponding edge of the diagram.

\item[$(\H,\B)$] Because of Lemma \ref{lem:starshaped}
  the \cemph{dblue}{rounded hoods} $(1-\lambda) \H + \lambda \B$, $\lambda \in [0,1]$
    satisfy the inequalities
  \eqref{lb_1} and \eqref{ib_2} with equality and their images through $f$ fill the corresponding edge.

\item[$(\L,\EqT)$] $\Iso$ denotes an \cemph{dblue}{isosceles triangle} with an angle $\gamma$ between
  the two edges of equal length (see Figure \ref{fig:IsoA}).
  If $\gamma \in [0,\nicefrac \pi 3]$, the two edges of equal length attain its diameter
  $D = D(\Iso) = 2R \cos(\nicefrac{\gamma}{2})$, where $R=R(\Iso)=1$.
    Abbreviating also $r=r(\Iso)$ and $w=w(\Iso)$, it was shown in \cite{HCS} and \cite{Sa} that
    \[ \left(2 + \sqrt{4 -\left(\nicefrac{D}{R}\right)^2}\right) r = w
    \quad \text{and} \quad 2 w R = D^2\sqrt{4-\left(\nicefrac{D}{R}\right)^2}.\]

  Thus one may check that $\Iso$ fulfills \eqref{eq:triangles}
    and \eqref{lb_2} with equality for any $\gamma \in [0,\nicefrac \pi 3]$.

  \begin{figure}
    \centering
    \begin{subfigure}[b]{0.45\textwidth}
      \includegraphics[trim = 1cm 3.5cm 1cm 2cm, width =\textwidth]{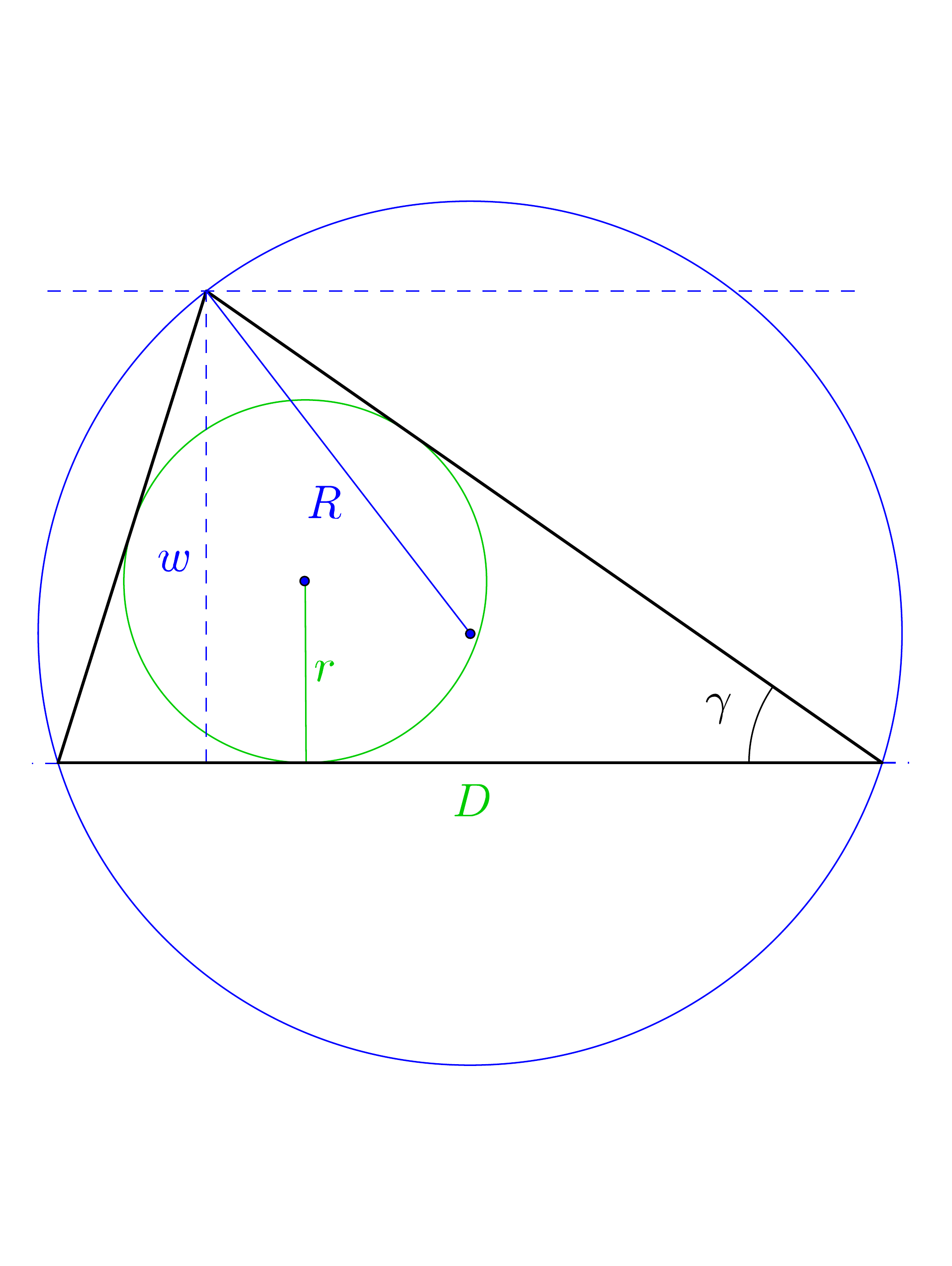}
      \caption{$\Iso$, $\gamma\in[0,\nicefrac{\pi}{3}]$.}
      \label{fig:IsoA}
    \end{subfigure} \hfill 
    \begin{subfigure}[b]{0.45\textwidth}
      \includegraphics[trim = 1cm 3.5cm 1cm 2cm, width = \textwidth]{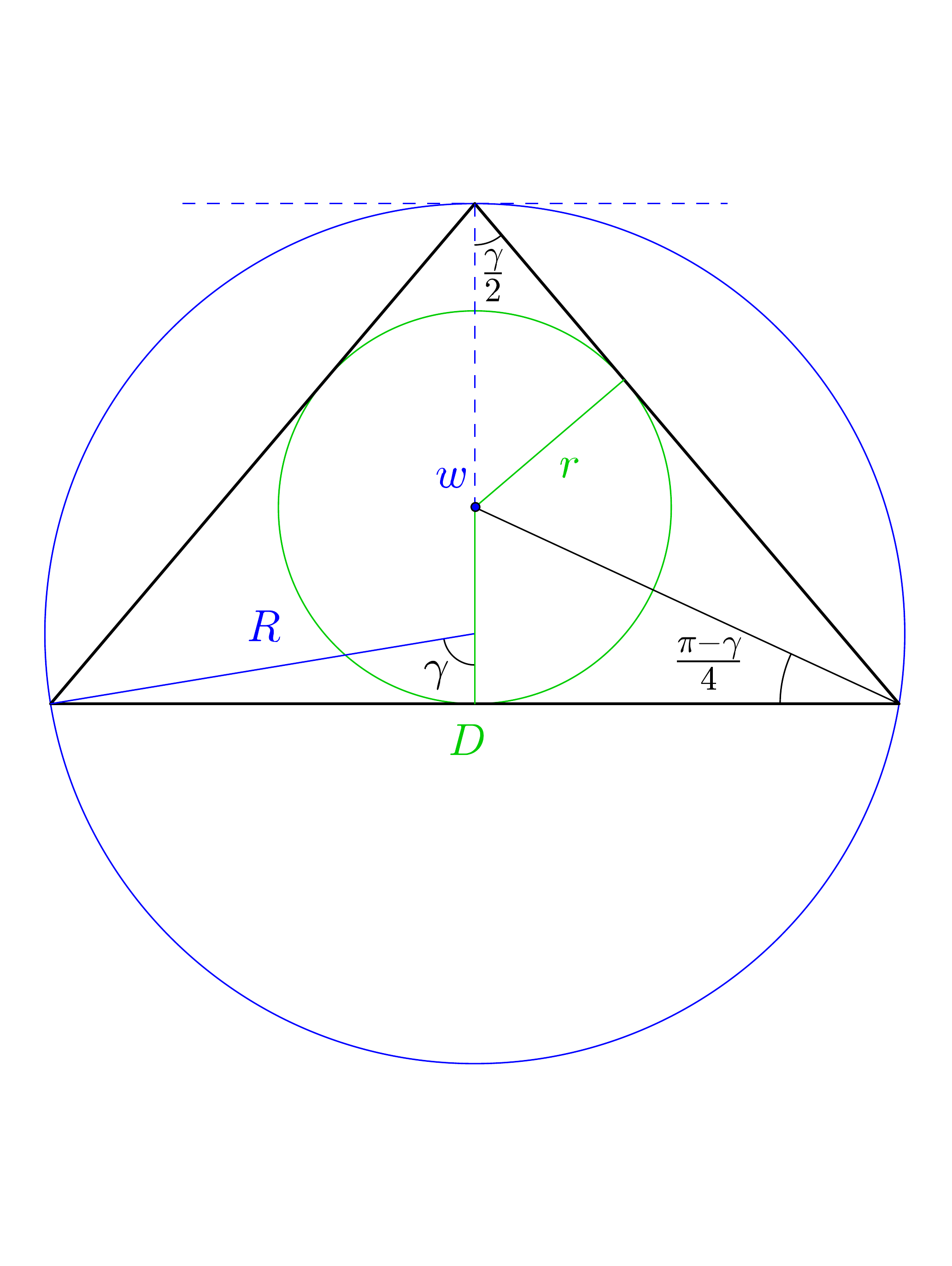}
      \caption{$\Iso$, $\gamma\in[\nicefrac{\pi}{3},\nicefrac{\pi}{2}]$.}
      \label{fig:IsoB}
    \end{subfigure}
    \caption{(Acute) isosceles triangles.}
  \end{figure}

\item[$(\RAT,\EqT)$] Consider the family of \cemph{dblue}{isosceles triangles} $\Iso$ as described above,
  but now with $\gamma \in [\nicefrac \pi 3, \nicefrac \pi 2]$. Obviously their diameter $D(\Iso)$
  is attained by the edge opposite to $\gamma$. Using Lemma \ref{lem:angle} we
  obtain that the angle at the circumcenter between the height onto the diametral edge and the radiusline
  from the center to one of the diametral vertices is again
  $\gamma$ (\cf Figure \ref{fig:IsoB}).
  The width is obviously attained orthogonal to the diametral edge and thus it is the sum of the inradius
  and the distance from the incenter to the opposing vertex. Considering the right angled triangle
  with the incenter, the midpoint of the diametral edge, and one of its endpoints as vertices,
  it is easy to check that the interior angle in that endpoint is $\nicefrac{(\pi-\gamma)}{4}$.
  Hence $2r(\Iso)=D(\Iso)\tan(\nicefrac{(\pi-\gamma)}{4})$ and using
  trigonometric identities it follows
  \[
  \tan\left(\frac{\pi -\gamma}{4}\right)=\frac{1-\cos\left(\nicefrac{(\pi-\gamma)}{2}\right)}
  {\sin\left(\nicefrac{(\pi-\gamma)}{2}\right)}
  =\frac{1-\sin\left(\nicefrac{\gamma}{2}\right)}{\cos\left(\nicefrac{\gamma}{2}\right)}.
  \] Altogether, omitting arguments we have
    \begin{align*}
      D &= 2R \sin(\gamma), \qquad w = r \left(1+\frac{1}{\sin\left(\nicefrac{\gamma}{2}\right)}\right),
      \quad \text{and} \\
      r & = \frac{D}{2} \left(\frac{1}{\cos(\nicefrac{\gamma}{2})}-\tan\left(\frac{\gamma}{2}\right)\right).
    \end{align*}

  Finally, again using trigonometric identities, we may remove $\gamma$ from the width formula in two ways
    \begin{align*}
      w & =r \left( 1 + \frac{1}{\sin(\nicefrac{\gamma}{2})}\right) = r\left(1 + \sqrt{\frac{2}{1-\cos(\gamma)}}\right)\\
      & = r\left(1 + \frac{\sqrt{2}}{\sin(\gamma)} \sqrt{1+\cos(\gamma)}\right) =
      r\left(1 + \frac{2\sqrt{2}R}{D}\sqrt{1+\sqrt{1-\left(\frac{D}{2R}\right)^2}}\right)
      \intertext{or}
      & = r \left( 1 + \frac{1}{\sin(\nicefrac{\gamma}{2})}\right) =
      2r\left(1+\frac{1}{2}\left(\frac{1}{\sin(\nicefrac{\gamma}{2})}-1\right)\right)
      = 2r\left(1+\frac{\nicefrac{1}{\cos(\nicefrac{\gamma}{2})} -
          \tan(\nicefrac{\gamma}{2})}{2\tan(\nicefrac{\gamma}{2})}\right)\\
      & = 2r\left(1+\frac{r}{D\tan(\nicefrac{\gamma}{2})}\right)=2r\left(1+\frac{r(1+\cos(\gamma))}{D\sin(\gamma)}\right)\\
      &=2r\left(1+\frac{2rR}{D^2}\left(1+\sqrt{1-\left(\frac{D}{2R}\right)^2}\right)\right).
    \end{align*}
    proving that $\Iso$, $\gamma \in [\nicefrac \pi 3, \nicefrac \pi 2]$ is extreme for \eqref{eq:sailing boats} and
    \eqref{eq:triangles}.

\item[$(\L,\RAT)$] The next family we consider are the \cemph{dblue}{right-angled triangles} $\RecT$, where
  $r\in[0,r(\RAT)]$ denotes their inradius. Naming the edges as their lenghts
  $a,b$ and $D=D(\RecT)=2R(\RecT)$, abbreviating $w=w(\RecT)$
    and recognizing that the inball touches $D$, \st it is split into two segments of
    lengths $a-r$ and $b-r$ (see Figure \ref{fig:RecT}), we easily see that the perimeter $p$ of $\RecT$
    is $2r+2D$ (or $2r+4R$). Thus using the semiperimeter formular for the width, we obtain
    \[wD = 2A = rp = 2r(r+D).\]

  One may easily calculate that the right-angled triangles are extreme for the
  inequalities \eqref{eq:triangles} and \eqref{ib_1}.

  \begin{figure}
    \centering
    \begin{subfigure}[b]{0.45\textwidth}
      \includegraphics[trim = 1cm 3.5cm 1cm 1cm, width =\textwidth]{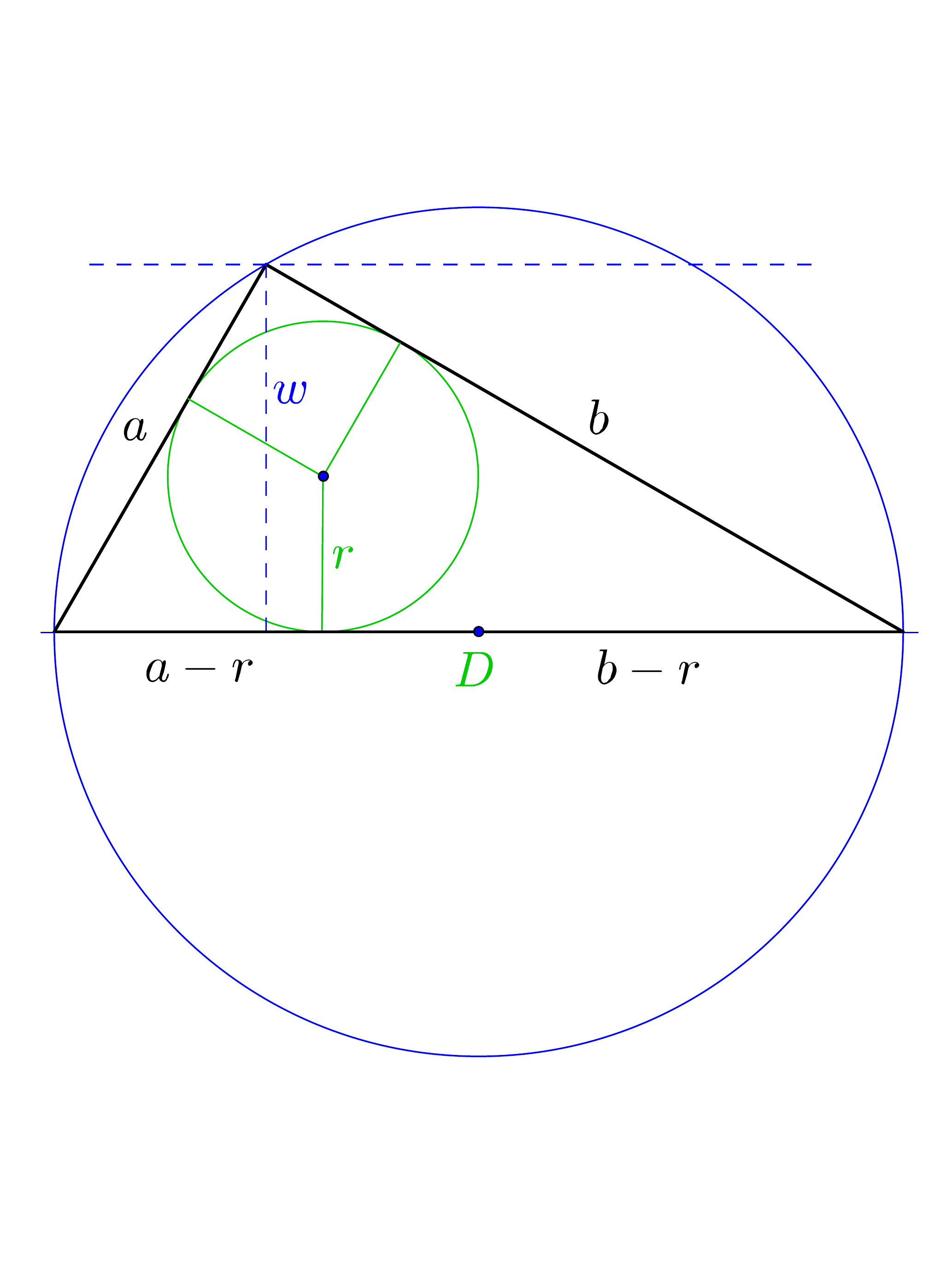}
      \caption{$\RecT$.\\ \hfill}
      \label{fig:RecT}
    \end{subfigure} \hfill
    \begin{subfigure}[b]{0.45\textwidth}
      \includegraphics[trim = 1cm 3.5cm 1cm 2cm, width =\textwidth]{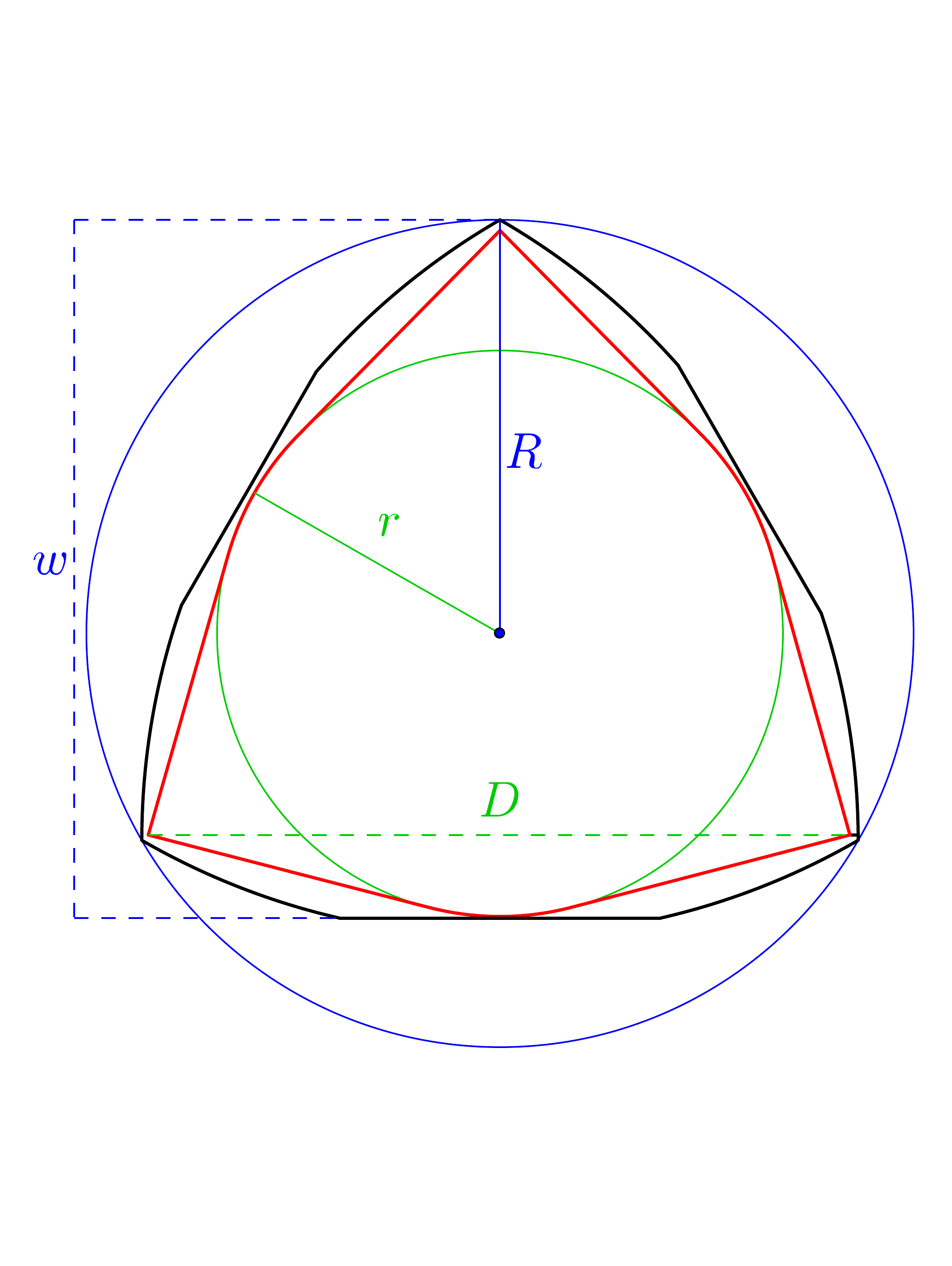}
      \caption{In black $\ReB$
        and in red a Yamanouti set with the same radii.}
      \label{fig:ReB}
    \end{subfigure}
    \caption{A right-angled triangle and a Reuleaux blossom.}
  \end{figure}

\item[$(\EqT,\ReT)$] For any $r \in [r(\EqT),r(\ReT)]$ we call $\ReB=(\nicefrac{r}{r(\EqT)} \, \EqT)\cap\ReT$ a
  \cemph{dblue}{Reuleaux blossom}, \st~$\ReB[r(\EqT)]=\ReB[\nicefrac 1 2]=\EqT$ and
  $\ReB[r(\ReT)]=\ReB[\sqrt3-1]=\ReT$
  (see Figure \ref{fig:ReB}). Obviously $r(\ReB)=r$, $D(\ReB)=\sqrt{3}R(\ReB)$, and
  $w(\ReB) = r+R(\ReB)$. Hence they are extreme for the inequalities \eqref{ub_1} and \eqref{ib_3}.

  A \cemph{dblue}{Yamanouti set} of inradius $r$ is mapped onto the same coordinates in the 3-dimensional
  Blaschke-Santal\'o diagram as the Reuleaux blossom $\ReB$.
  They are the convex hull of $\EqT$ and the intersection of
  three balls with centers in the vertices of $\EqT$ and radius taken in $[w(\EqT),w(\ReT)]$
  (see \cite{Sa} and \cf Figure \ref{fig:ReB}).
  While the Yamanouti set is a unique minimal set (with respect to set inclusion) mapped to these
  coordinates, the corresponding Reuleaux Blossom is maximal but not unique (as one may support the inball in different
  points than the chosen ones). However, the Reuleaux blossoms are the only maximizers which possess
  the same symmetry group as $\EqT$.

\item[$(\EqT,\SB)$] Let $\gamma\in[\nicefrac{\pi}{3},\nicefrac{\pi}{2}]$ and $c$
  the incenter of $\Iso$. Now rescale $\Iso-c$ by a factor $\rho$, \st~the vertex $p$ of
    $\rho(\Iso-c)$
    between the two edges of equal length touches the boundary of $\B$.
    Then the \cemph{dblue}{concentric sailing boat} is defined as $\CSB=\rho(\Iso-c)\cap\B$
    (see Figure \ref{fig:CSB}).
  Obviously, $R=R(\CSB)=1$ and
  $D=D(\CSB)= D(\Iso) = R\sin(\gamma)$.
    Moreover, since $\CSB$ is concentric and the distance of the center from $p$ is $R$ we obtain
    \[r=r(\CSB)= R \sin\left(\frac{\gamma}{2}\right) \quad \text{and} \quad w = r+R =
    r \left(1+\frac{1}{\sin(\nicefrac{\gamma}{2})}\right)\]
    However, since $D=R \sin(\gamma)$ it follows exactly in the same ways as shown for $w(\Iso)$ in the
    $(\EqT,\RAT)$-edge that
    \[w = r \left( 1 + \frac{1}{\sin(\nicefrac{\gamma}{2})}\right) =
    2r\left(1+\frac{2rR}{D^2}\left(1+\sqrt{1-\left(\frac{D}{2R}\right)^2}\right)\right).\]

   \begin{figure}
    \centering
    \begin{subfigure}[b]{0.45\textwidth}
      \includegraphics[trim = 1cm 3.5cm 1cm 2cm, width =\textwidth]{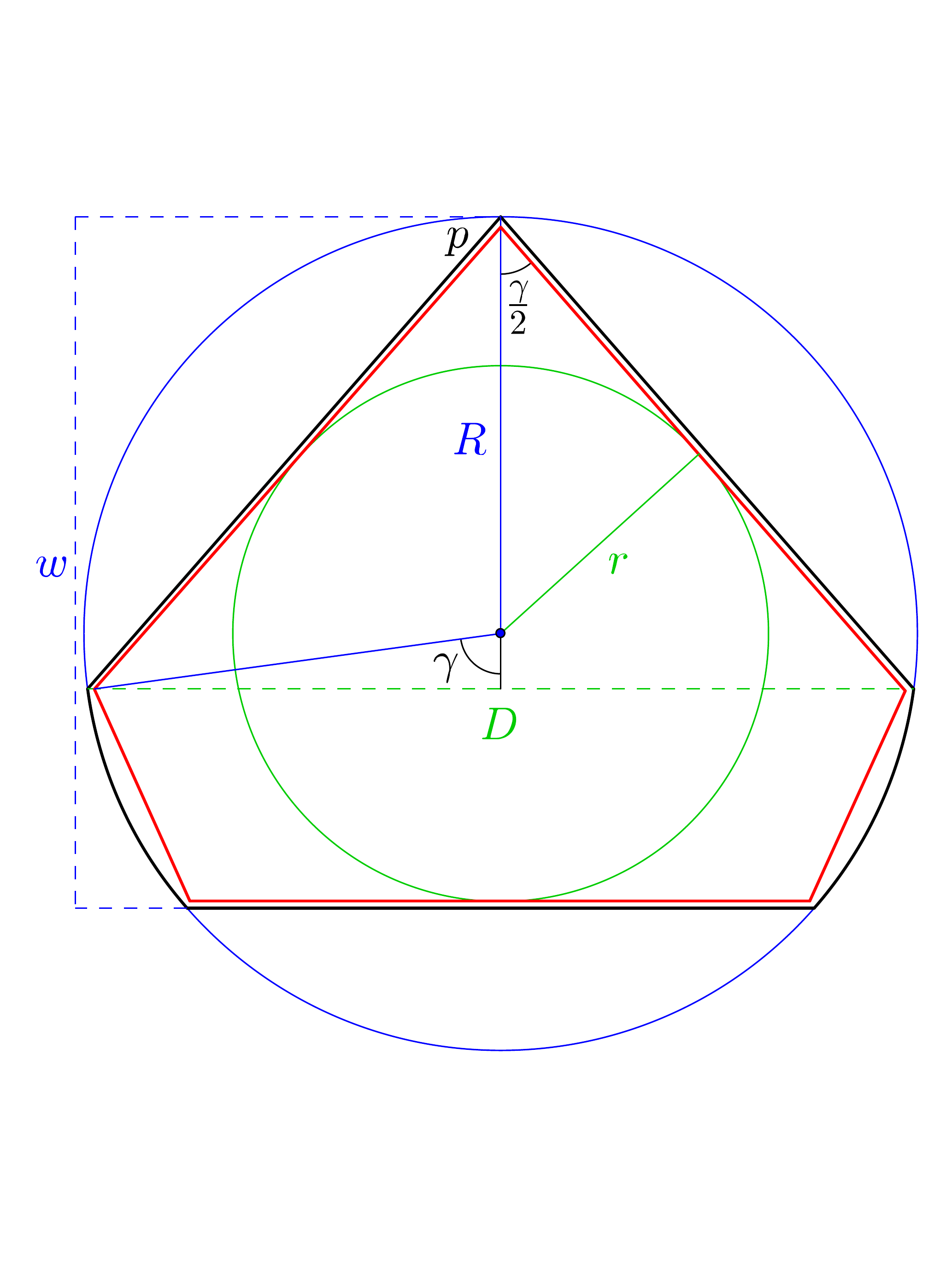}
      \caption{In black $\CSB$, in red the pentagon $\mathrm{CP}^{\circledcirc}_\gamma$.\\\hfill \\\hfill}
      \label{fig:CSB}
    \end{subfigure} \hfill 
        \begin{subfigure}[b]{0.45\textwidth}
      \includegraphics[trim = 1cm 3.5cm 1cm 2cm, width = \textwidth]{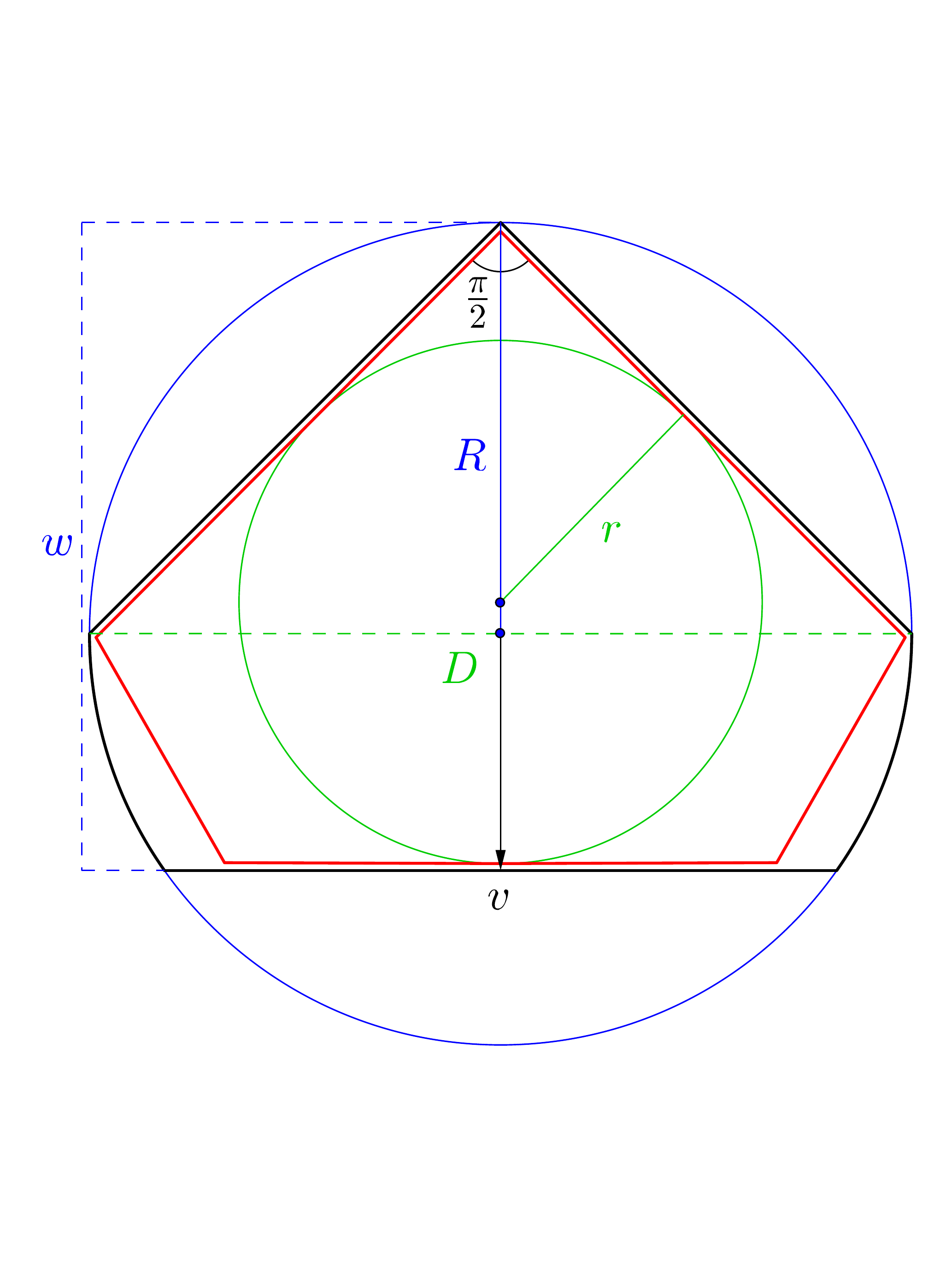}
      \caption{In black $\RSB$, in red the pentagon $\RSB^{\min}$ in case that
          $w(\RSB) > \sqrt2 R(\RSB)$.}
      \label{fig:RSB}
    \end{subfigure}
    \caption{A concentric and a right-angled sailing boat.}
  \end{figure}

  Hence the concentric sailing boats are extreme for the inequalities \eqref{ub_1} and \eqref{eq:sailing boats}.
  Denoting the concentric pentagon built from the vertices of $\CSB$ by
  $\mathrm{CP}^{\circledcirc}_\gamma$, it holds
  $f(K)=f(\CSB)$, iff $\mathrm{CP}^{\circledcirc}_\gamma \subset K \subset \CSB$ (\cf~Figure \ref{fig:CSB}).

\item[$(\RAT,\SB)$] Let $r\in[r(\RAT),r(\SB)]$ and $v\in\R^2$,
  \st~the vertex of $v+(\nicefrac{r}{r(\RAT)})\RAT$ between the two edges of equal length
  belongs to $\S$ and the edges of
  equal length induce equal caps in $\B$
  (\cf~Figure \ref{fig:RSB}).
  Then $\RSB=(v+(\nicefrac{r}{r(\RAT)})\RAT)\cap\B$ is a \cemph{dblue}{right-angled sailing boat}.
  Hence $D(\RSB)=2R(\RSB)$, $r(\RSB)=r(v+(\nicefrac{r}{r(\RAT)})\RAT)=r$, and
  $w(\RSB)=\nicefrac{r}{r(\RAT)}w(\RAT)=(\sqrt{2}+1)r$.
  Thus they are extreme for the inequalities \eqref{eq:sailing boats} and \eqref{ib_1} and it holds  $K \subset \RSB$
  for any set $K$ with $f(K)=f(\RSB)$.
  Concerning possible minimal sets mapped to the same coordinates in the diagram,
    let $p^1,p^2,p^3 \in \S$, \st~$\conv\{p^1,p^2,p^3\} = \RAT$, with the right-angle at $p^3$.
    Now, if $w(\RSB) \le \sqrt{2} R(\RSB)$, the set $\RSB^{\min} :=
    \conv\left(\RAT, (p^3 + w(\RSB)\B)\cap\RSB \right)$ fulfills $\RSB^{\min} \subset K$, for all
    $K$ with $f(K)=f(\RSB)$.
    In case of $w(\RSB) > \sqrt{2} R(\RSB)$ (\ie $r>(2-\sqrt2)R(\RSB)$),
    call $L$ the supporting line
    to the inball in $v$, and let $p^4,p^5 \in L$ at distance $w(\RSB)$
    from the segments $[p^1,p^3]$, $[p^2,p^3]$, respectively. Then the pentagon
    $\RSB^{\min}:=\conv\{p^i,i \in [5]\}$ is a minimal set mapped to the same coordinates as
    $\RSB$. However, one should recognize that $\RSB^{\min}$ is only one (maybe the \enquote{nicest})
    possible choice for such a set (\cf~Figure \ref{fig:RSB}).

\item[$(\EqT,\FRT)$] For any $r\in[r(\EqT),r(\FRT)]$ there exists $c\in\R^2$, \st~$c+r\B$
  is contained in $\FRT$ (by definition of inradius) and tangent to the linear edge of
  $\FRT$ (\cf Figure \ref{fig:BEqONE}).
  Assuming $c$ to be equidistant from the endpoints of that linear edge
  the sets $\BEq=\conv(\EqT,v+r\B)$,
  $r\in[r(\EqT),r(\FRT)]$ are called \cemph{dblue}{bent equilaterals}
  and they satisfy
  $r(\BEq)=r$, $D(\BEq)=\sqrt{3}R(\BEq)$ and $w(\BEq)=w(\EqT)$.
  Thus the bent equilaterals with $r\in[r(\EqT),r(\FRT)]$ are extreme for the inequalities \eqref{lb_2} and \eqref{ib_3}.

  With respect to set inclusion $\BEq$ is a minimal set mapped onto these coordinates.
  However, since there is some freedom in placing $c$, it is not a unique minimal set.

  Choosing two common supporting half-spaces $H_i$, $i=1,2$ of $\BEq$ and its inball,
  \st~$c+ r\B$ is the inball of $\FRT \cap H_1 \cap H_2$,
  one gets a maximal set containing $\BEq$
  (but neither the choice of the half-spaces $H_i$, $i=1,2$ is unique nor is
  the choice of $c$, \cf~Figure \ref{fig:BEqONE}).

  \begin{figure}
    \includegraphics[trim = 1cm 3.5cm 1cm 2cm, width = 7cm]{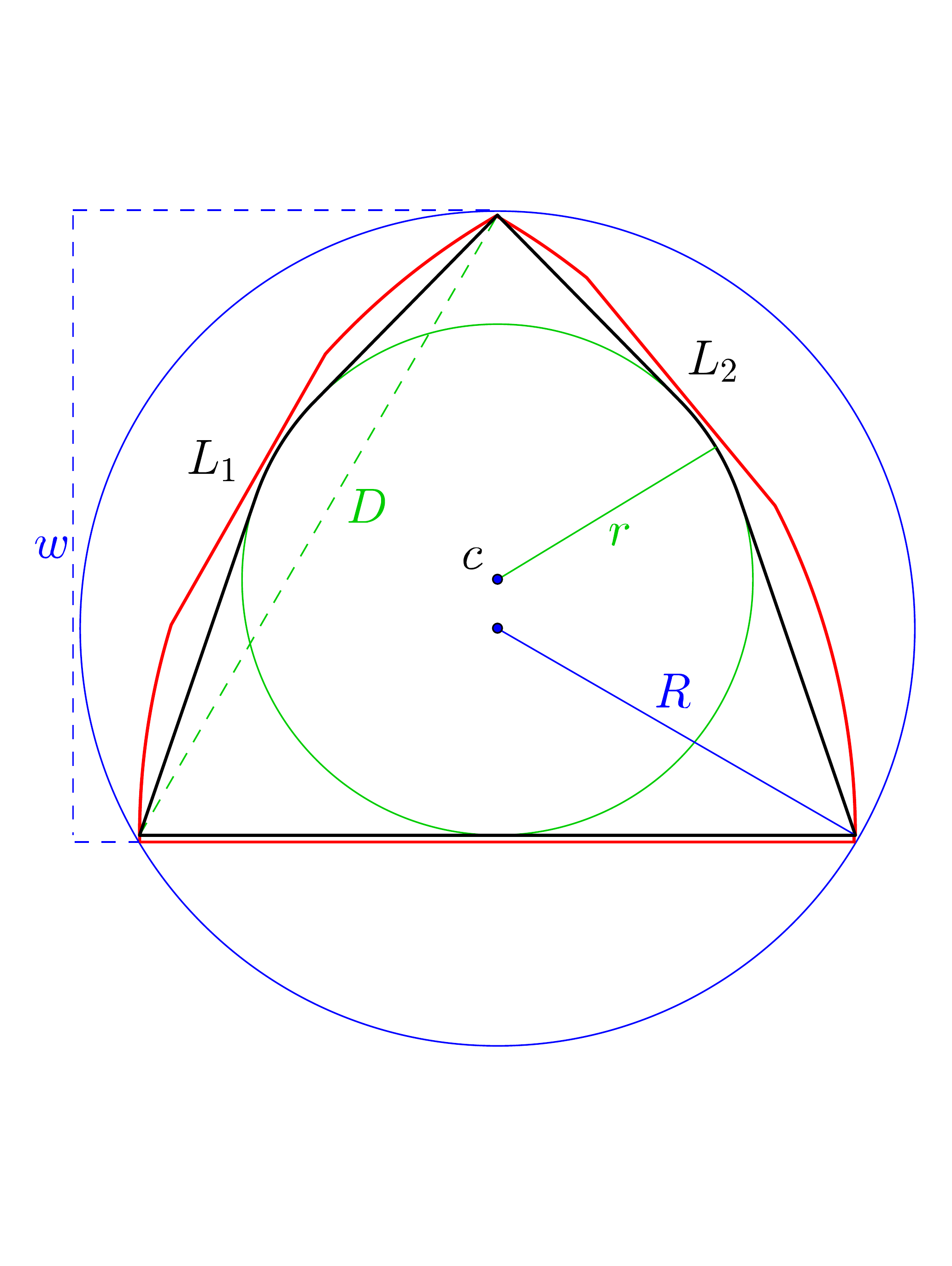}
    \caption{In black a bent equilateral $\BEq$, $r\in[r(\EqT),r(\FRT)]$
      (for which all radii keep constant moving the inball horizontally),
      in red one possible maximal set containing $\BEq$.}
    \label{fig:BEqONE}
  \end{figure}

\item[$(\FRT,\SRT)$] On the contrary, for any $r\in[r(\FRT),r(\SRT)]$, let
  $c\in\R^2$, \st~$c+r\B$ is tangent to the two (non-linear) arcs of $\FRT$
  (see Figure \ref{fig:BEqTWO}). Then we define the
  \cemph{dblue}{maximally-sliced Reuleaux triangle} $\SliRT$
  to be the intersection of $\ReT$ with a halfspace supporting $c+r\B$ and containing
  a vertex of $\FRT$, which is adjacent to its linear edge, on the boundary line of the halfspace.
  Abbreviating $D=D(\SliRT)$ and $R=R(\SliRT)=1$ again, it holds $r(\SliRT)=r$
  and $D = \sqrt{3}R $. Considering the angles
  $\alpha,\beta,\gamma$ inside $\SliRT$ (as given in Figure \ref{fig:BEqTWO}), we have
  \[\text{(i)} \;~ \cos(\alpha)=\nicefrac{D}{2(D-r)}, \quad \text{(ii)} \;~
  \sin(\alpha+\beta) = \nicefrac{r}{(D-r)}, \quad \text{(iii)} \;~ \cos(\gamma)=\nicefrac{w}{D}.
  \]
  Passing (i) into (ii) one obtains
  \[ \beta=\arcsin\left(\frac{r}{D-r}\right)-\arccos\left(\frac{D}{2(D-r)}\right) \]
  and since $\gamma = \nicefrac{\pi}{6}-\beta$ it follows from (iii) for the width $w=w(\SliRT)$
  that
  \[ w = D\cos\left(\frac{\pi}{6}-\arcsin\left(\frac{r}{D-r}\right)
    + \arccos \left(\frac{D}{2(D-r)}\right)\right) \]

  The sliced Reuleaux triangles fulfill inequalities \eqref{eq:3PT} and \eqref{ib_3} with equality. Moreover,
  denoting $\BEq:=\conv(\EqT, c+r\B)$, with $r\in[r(\FRT),r(\SRT)]$ again a
  \cemph{dblue}{bent equilateral}
  it holds $f(K)=f(\SliRT)$, iff $\BEq \subset K \subset \SliRT$ (\cf~Figure \ref{fig:BEqTWO}).

  \begin{figure}
    \centering
    \begin{subfigure}[b]{0.45\textwidth}
      \includegraphics[trim = 1cm 3.5cm 1cm 2cm, width =\textwidth]{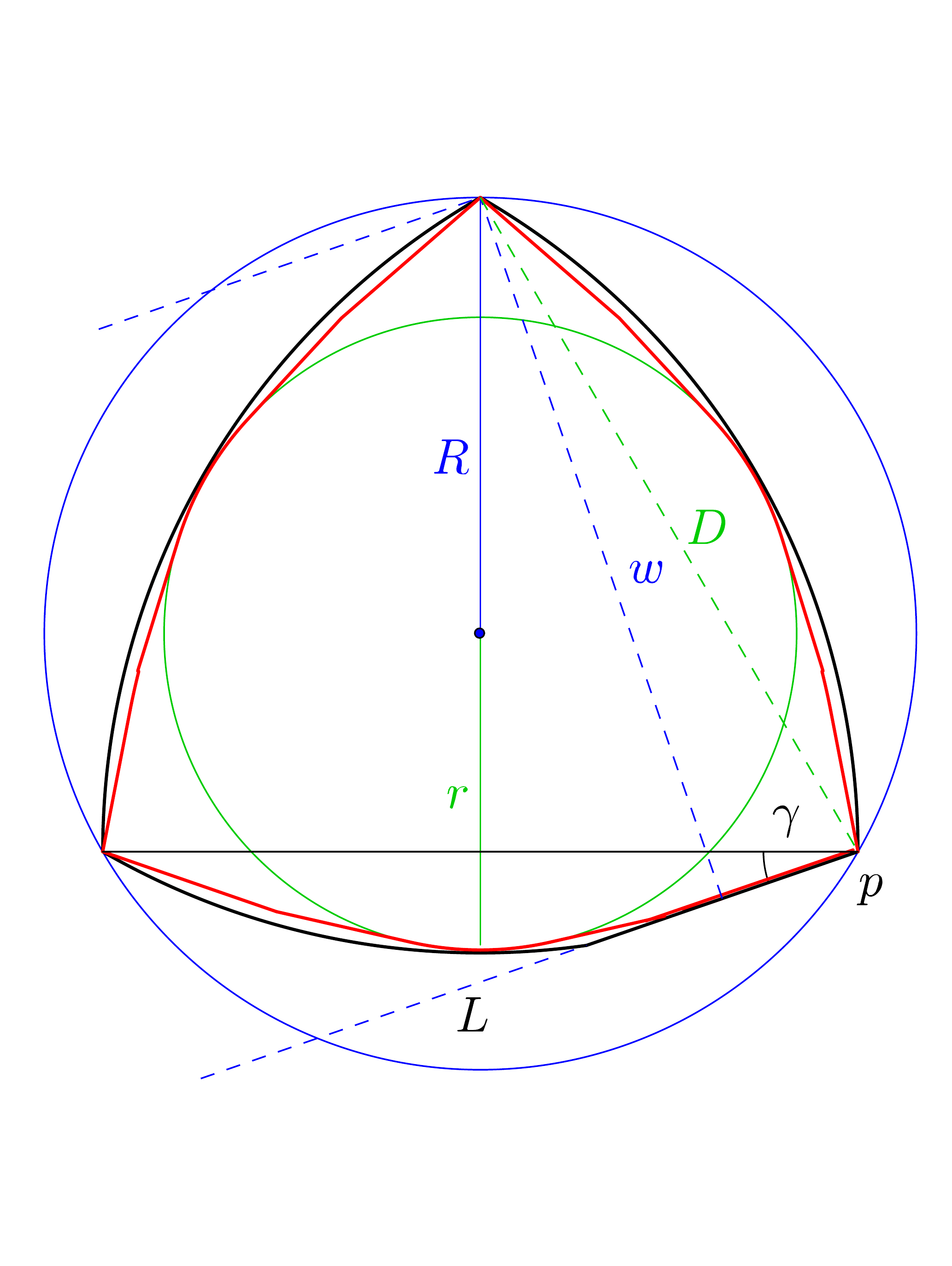}
      \caption{In black a concentric sliced Reuleaux triangle $\CSRT$ and in red
      $\mathrm{BY_{\gamma}}$.\\\hfill}\label{fig:CSRT}
    \end{subfigure} \hfill
    \begin{subfigure}[b]{0.45\textwidth}
      \includegraphics[trim = 1cm 3.5cm 1cm 2cm, width = \textwidth]{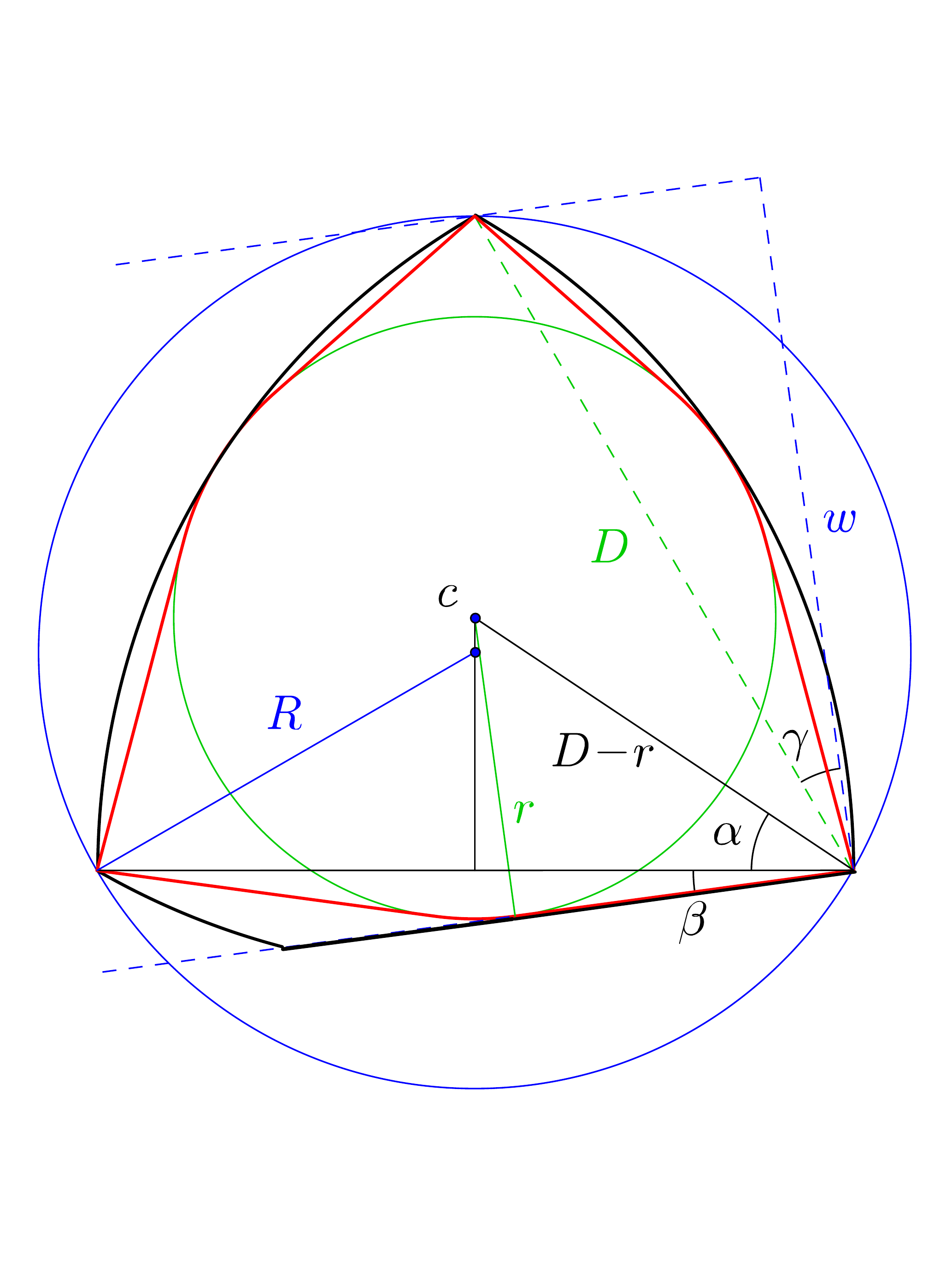}
      \caption{In black a maximally-sliced Reuleaux triangle $\SliRT$,
          in red a bent equilateral $\BEq$, $r\in[r(\FRT),r(\SRT)]$.}
      \label{fig:BEqTWO}
    \end{subfigure}
    \caption{Sliced Reuleaux triangles}
  \end{figure}

\item[$(\SRT,\ReT)$] Let $L$ be a line containing a vertex of $\ReT$, say $p$, not cutting the
  interior of the inball of $\ReT$, and
  $\gamma\in[\arcsin(\sqrt{3}-1)-\nicefrac{\pi}{6},\nicefrac{\pi}{6}]$ the
  angle between $L$ and one of the segments joining $p$ with one of the other two vertices (see Figure \ref{fig:CSRT}).
  The family of \cemph{dblue}{concentric sliced Reuleaux triangles} $\CSRT$
  are obtained from intersecting $\ReT$ with the halfspace induced by $L$.
  The concentric sliced Reuleaux triangles have the same diameter, in-, and circumradius as $\ReT$, while the width
  is attained in the orthogonal direction to the line $L$. Hence
  $w(\CSRT) = D(\ReT) \sin\left(\nicefrac{\pi}{3}+\gamma\right)$.
  Concentric sliced Reuleaux triangles are extreme for the inequalities \eqref{ib_2} and \eqref{ib_3}.

  Denoting the convex hull of $r(\CSRT)\B$ and the Yamanouti set sharing width, diameter and circumradius with $\CSRT$ by
  $\mathrm{BY_{\gamma}}$, then we have $f(K)=f(\CSRT)$, iff $\mathrm{BY_{\gamma}} \subset K \subset \CSRT$ (see Figure \ref{fig:CSRT}).

\item[$(\L,\BT)$] The construction of the sets in this edge is a generalization of that of the bent trapezoid
  $\BT$ in Subsection \ref{ss:vertices}.
  Let $\Iso=\conv\{p^1,p^2,p^3\}$ with $\gamma \in[0,\arcsin(\nicefrac{3}{4})]$, \st~$\gamma$
  is the angle at $p^1$. Moreover let $p^4 \neq p^3$ in the circumsphere
  of $\Iso$, \st~$\conv\{p^1,p^2,p^4\}$ is congruent with $\conv\{p^1,p^2,p^3\}$ and possesses its angle $\gamma$ at $p^2$.
  Substituting the two edges $[p^1,p^4]$ and $[p^2,p^3]$ by two arcs of radius $D(\Iso)$ whose centers are $p^1$
  and $p^2$, respectively, the resulting set is a \cemph{dblue}{(general) bent trapezoid} $\BTrap$,
  $\gamma \in[0,\arcsin(\nicefrac{3}{4})]$ (see Figure \ref{fig:BTrapONE}).
  It holds $D(\BTrap) = D(\Iso) = 2 R(\Iso) \cos\left(\nicefrac{\gamma}{2}\right)$
  and $w(\BTrap) = w(\Iso) = D(\Iso) \sin(\gamma)$ and since they possess two parallel edges touching
  the inball in antipodal points $w(\BTrap)=2r(\BTrap)$.

  The bent trapezoids are extreme for the inequalities \eqref{lb_1} and \eqref{lb_2}.
  While $\BTrap$ is the unique maximal set with respect to set inclusion, which is mapped onto these coordinates in the diagram,
  there does not exist a unique minimal set. Essentially, the convex hull of $\Iso$
  and any of the possible inballs of $\BTrap$
  shares all four radii with $\BTrap$ and is minimal in that sense.

  \begin{figure}
    \centering
    \begin{subfigure}[b]{0.45\textwidth}
      \includegraphics[trim = 1cm 3.5cm 1cm 2cm, width =\textwidth]{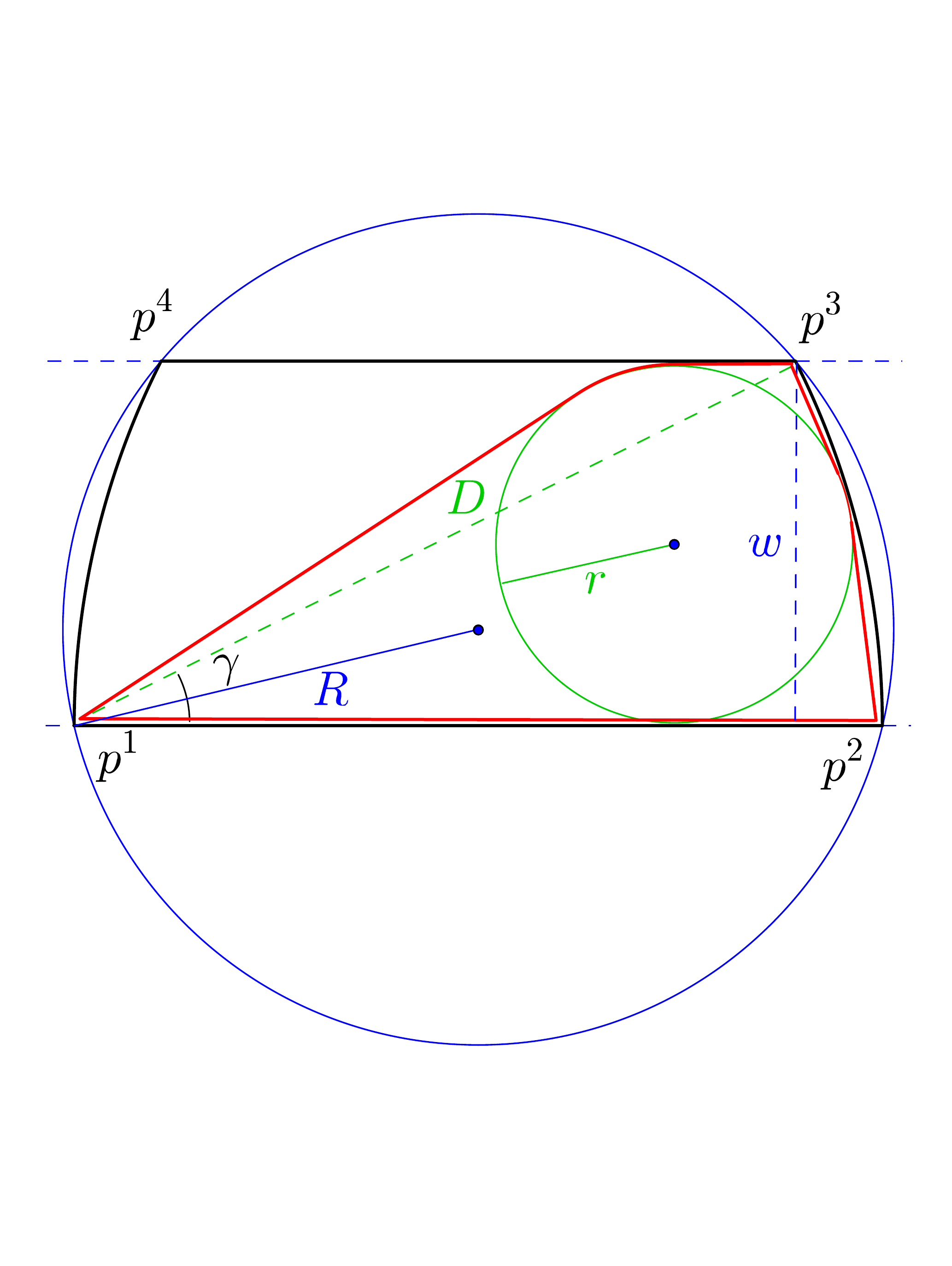}
      \caption{In black $\BTrap$ with $\gamma < \arcsin(\nicefrac{3}{4})$,
        in red a minimal set
        whose inball is tangent to one of the curved edges of $\BTrap$.}
      \label{fig:BTrapONE}
        \end{subfigure} \hfill
    \begin{subfigure}[b]{0.45\textwidth}
      \includegraphics[trim = 1cm 3.5cm 1cm 2cm, width = \textwidth]{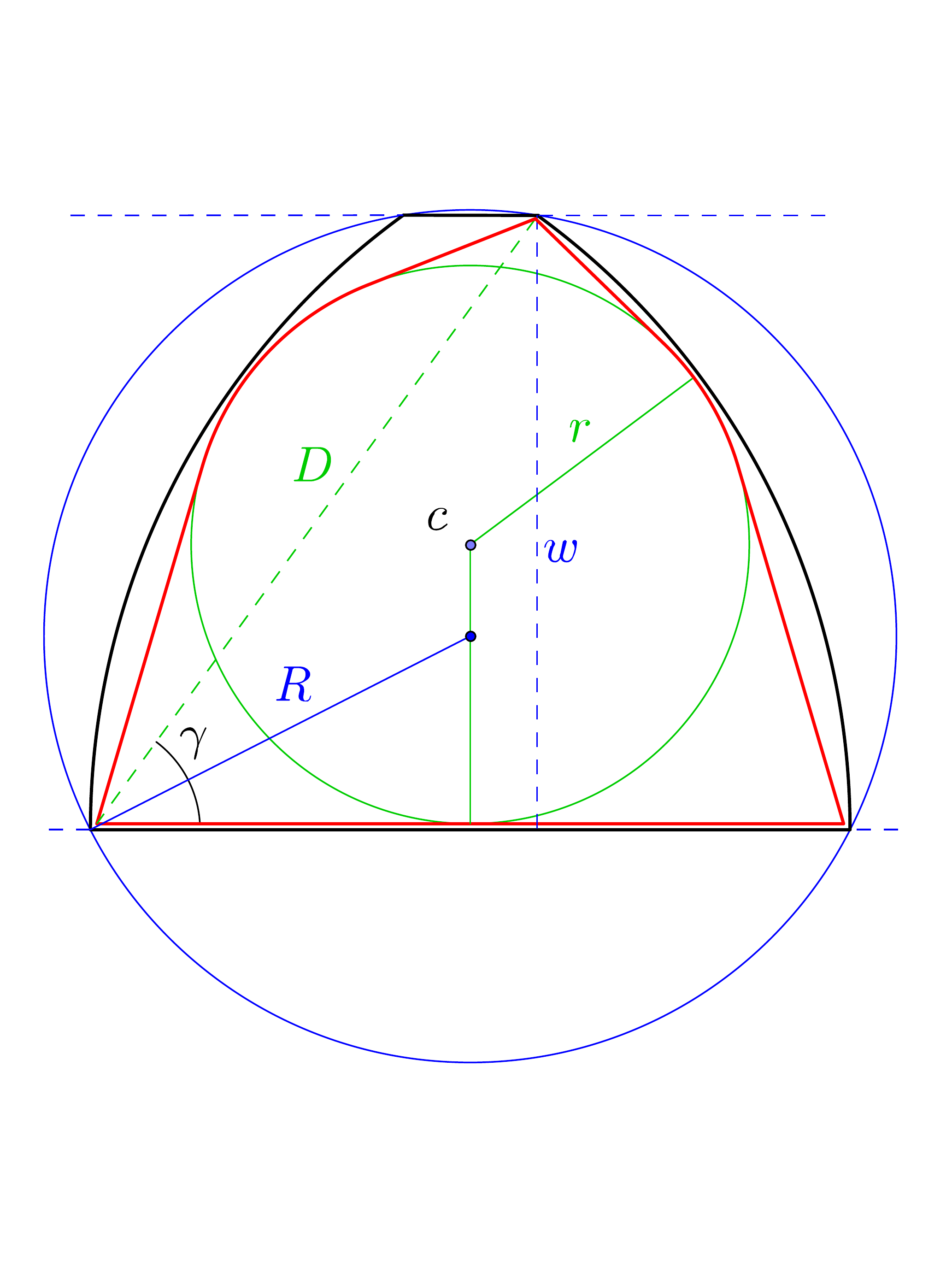}
      \caption{In black $\BTrap$ with $\gamma > \arcsin(\nicefrac{3}{4})$,
        in red an according bent isosceles.\\\hfill}
      \label{fig:BTrapTWO}
    \end{subfigure}
    \caption{Bent trapezoids.}
  \end{figure}

\item[$(\BT,\FRT)$] Adopting the construction of the bent trapezoids with
  $\gamma\in[0,\arcsin(\nicefrac{3}{4})]$ above, we define the \cemph{dblue}{(general) bent trapezoid}
  $\BTrap$ with $\gamma\in[\arcsin(\nicefrac{3}{4}),\nicefrac{\pi}{3}]$ (see Figure \ref{fig:BTrapONE}).
  They keep $D(\BTrap) = D(\Iso) = 2 R(\Iso) \cos\left(\nicefrac{\gamma}{2}\right)$,
  and $w(\BTrap)=w(\Iso) = D(\Iso) \sin(\gamma)$, but in difference to the bent trapezoids before, those with
  $\gamma > \arcsin(\nicefrac{3}{4})$
  have an inball touching the two arcs of
  circumference and only the longer of the parallels. Thus it holds
  $\nicefrac 14 \, D(\BTrap)^2 + r(\BTrap)^2 = (D(\BTrap)-r(\BTrap))^2$
  from which we obtain $\nicefrac34 \, D(\BTrap)^2 - 2D(\BTrap)r(\BTrap) = 0$ or
  $8r(\BTrap)=3D(\BTrap)$.
  Hence, the sets $\BTrap$, $\gamma\in[\arcsin(\nicefrac{3}{4}),\nicefrac{\pi}{3}]$,
  fulfill inequalities \eqref{lb_2} and \eqref{eq:3PT} with equality.

  While $\BTrap$ is the unique maximal set with respect to set inclusion,
  the bent isosceles given by the convex hull of one of the two possible copies of
    $\Iso$ inside $\BTrap$ and the inball of $\BTrap$ is a minimal set with respect to set inclusion
    mapped to the same coordinates,
    which is unique (up to mirroring along the symmetry axis of $\BTrap$,
    \cf~Figure \ref{fig:BTrapTWO}).

\item[$(\SRT,\H)$] Now, we generalize the hood $\H$ as constructed
  in Subsection \ref{ss:vertices}:
  For any $\gamma \in\left[2\arcsin\left(\nicefrac{r(\H)}{D(\H)}\right),\nicefrac{\pi}{3}\right]$
  let $\Iso=\conv\{p^1,p^2,p^3\}$ with $D(\Iso) = \norm[p^1-p^2]=\norm[p^1-p^3]$
  and define the space contained between
  \begin{itemize}
  \item the arcs with centers in the vertices of $\Iso$ and radius $D(\Iso)$,
  \item a line $L$ through $p^2$ supporting the ball $\left(D(\Iso)-R(\Iso)\right)\B$ and
    the smaller angle $\beta$ between it and $[p^1,p^2]$, as well as
  \item the parallel line $L'$ to $L$ supporting $\Iso$ in $p^3$,
  \end{itemize}
  as the \cemph{dblue}{(general) hood} $\Hood$ (see Figure \ref{fig:Hood}).

  One can easily see that
  $D(\Hood) = D(\Iso)=2R(\Iso)\cos\left(\nicefrac{\gamma}{2}\right)$ and
  $r(\Hood) =  D(\Hood)-R(\Hood)$.

  Observing that the angle between $[p^1,p^2]$ and $[0,p^2]$ in $p^2$ is $\nicefrac{\gamma}{2}$
  let $\alpha$ be the angle between $[p^2,p^3]$ and the perpendicular of $L$
  and $\beta = \nicefrac{\gamma}{2} - \alpha$ the angle between $[p^1,p^2]$ and $L$, both in $p^2$.
  Then, omitting the argument $\Hood$, we get
  \begin{equation*}
    \begin{split}
      & \text{(i)} \quad D = 2 R \cos(\nicefrac{\gamma}{2}), \qquad
      \text{(ii)} \quad r = R \sin(\nicefrac{\gamma}{2}+\beta) = R \sin(\gamma-\alpha) \\
      & \text{(iii)} \quad w = \norm[p^2-p^3] \cos(\alpha) = 2 D \sin(\nicefrac{\gamma}{2}) \cos(\alpha).
    \end{split}
   \end{equation*}
  From (i) and (ii) one
  immediately obtains
  $\nicefrac{\gamma}{2}=\arccos\left(\nicefrac{D}{2R}\right)$ and
  $\alpha=\gamma-\arcsin\left(\nicefrac{r}{R}\right)$. Thus (iii)
  can be rewritten as
  \begin{equation*}
    \begin{split}
      w & = 2 D \sin\left(\arccos\left(\nicefrac{D}{2R}\right)\right)
      \cos\left(\gamma-\arcsin\left(\nicefrac{r}{R}\right)\right)\\
      & = 2 D \sqrt{1-\left(\frac{D}{2R}\right)^2}
      \cos\left(2 \arccos\left(\frac{D}{2(D-r)}\right)-\arcsin\left(\frac{r}{D-r}\right)\right).
    \end{split}
  \end{equation*}

  The hoods $\Hood$ with $\gamma \in \left[2\arcsin\left(\nicefrac{r(\H)}{D(\H)}\right),\nicefrac{\pi}{3}\right]$,
  are extreme for the inequalities \eqref{eq:3PT} and  \eqref{ib_2}.
  While $\Hood$ is maximal with respect to set inclusion,
  the bent isosceles $\conv(\Iso, (D(\Hood)-R(\Hood))\B)$
  is minimal sharing all radii with $\Hood$ (see Figure \ref{fig:Hood}).
  \begin{figure}
    \centering
    \begin{subfigure}[b]{0.45\textwidth}
      \includegraphics[trim = 1cm 3.5cm 1cm 2cm, width =\textwidth]{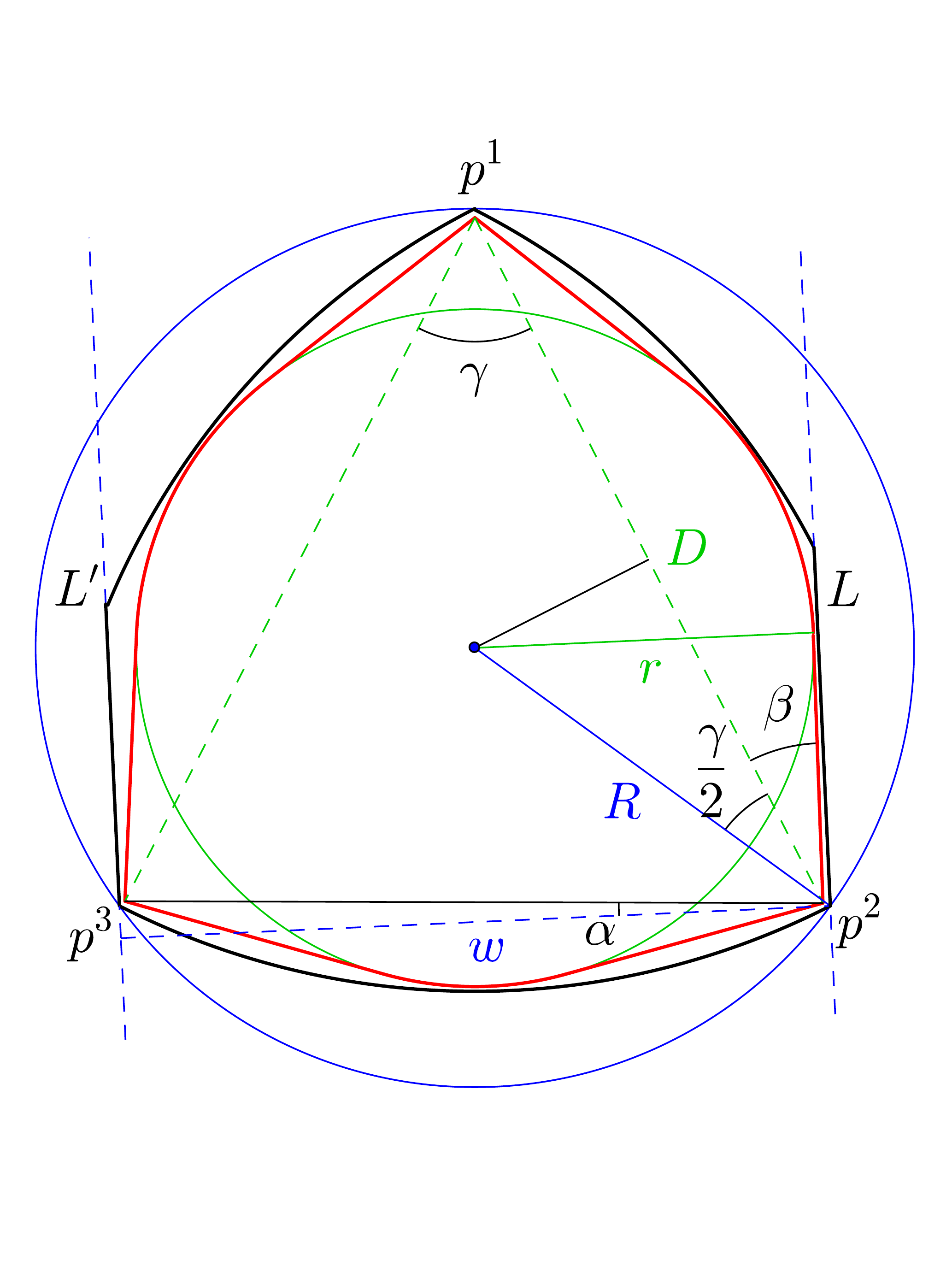}
      \caption{In black a general hood $\Hood$, in red a bent isosceles.}
      \label{fig:Hood}
    \end{subfigure} \hfill
    \begin{subfigure}[b]{0.45\textwidth}
      \includegraphics[trim = 1cm 3.5cm 1cm 2cm, width = \textwidth]{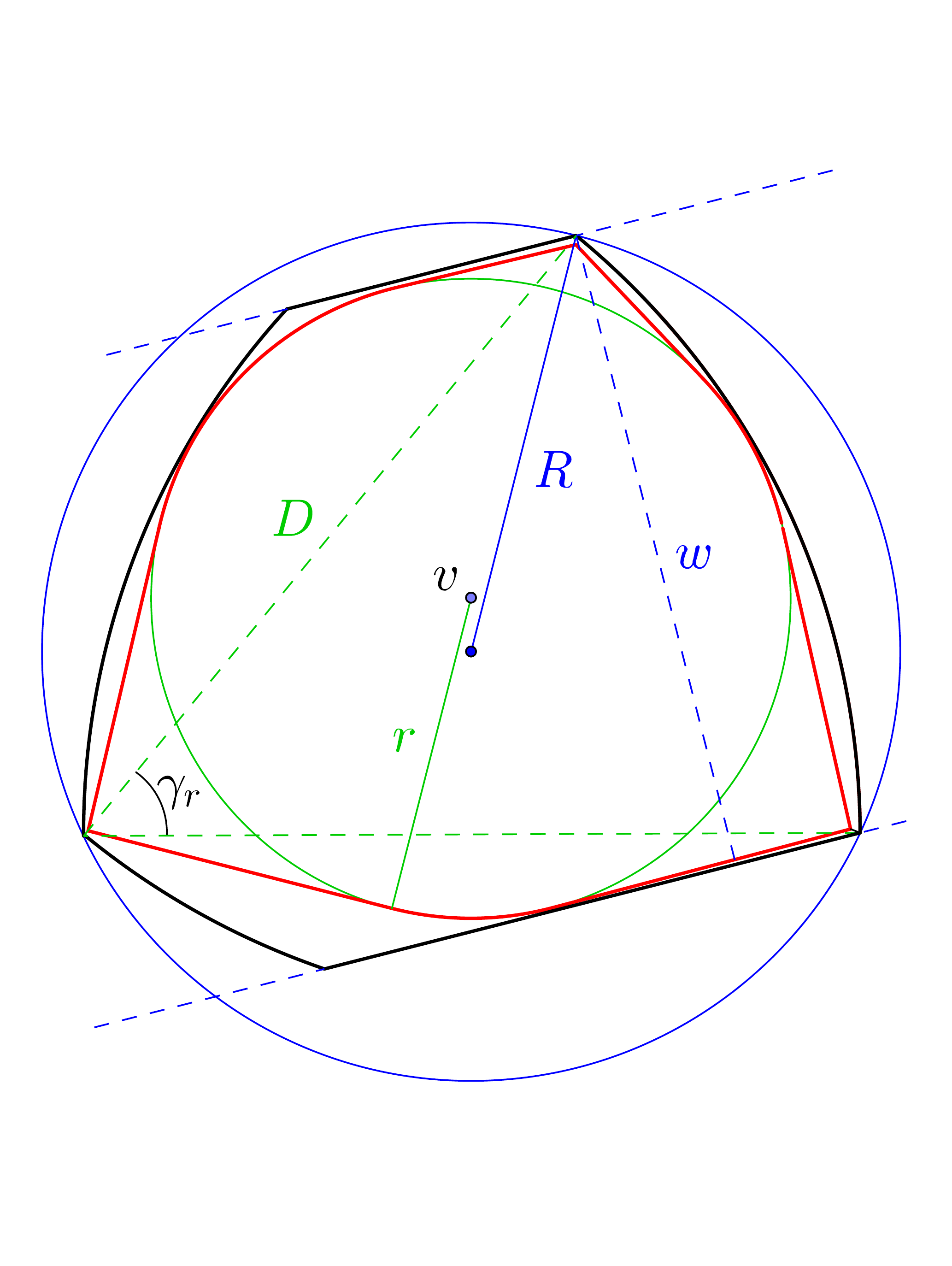}
      \caption{In black a bent pentagon $\BPen[r][\gamma_r]$
          and in red a bent isosceles
            $\BIso[r][\gamma_r]$, $r \in [r(\BT),r(\H)]$.}
      \label{fig:IRRBIso}
    \end{subfigure}
    \caption{Sets from the two edges meeting in $\H$ and bounding \eqref{eq:3PT}.}
  \end{figure}

\item[$(\BT,\H)$] Let $r\in[r(\BT),r(\H)]$ and $\gamma_r$ the maximal $\gamma \in[0,\nicefrac{\pi}{3}]$,
  \st~we can find $c\in \R^2$ for which
  \begin{enumerate}[(i)]
  \item $c+r\B$ is tangent to the two arcs of circumference 
    with centers $p^1$ and $p^2$ and radius $D(\Iso[\gamma_r])$
    above the segments $[p^2,p^3]$ and $[p^1,p^3]$, respectively, as well as
  \item two parallel lines $L$ and $L'$ both supporting $c+r\B$, support $\Iso[\gamma_r]$
    in, respectively, $p^2$ and $p^3$  (\cf~Figure \ref{fig:IRRBIso}).
  \end{enumerate}

  The \cemph{dblue}{bent pentagon} $\BPen[r][\gamma_r]$ is
  defined as the space contained between the lines $L, L'$ and the
  three arcs with radius $D(\Iso[\gamma_r])$ around the vertices of $\Iso[\gamma_r]$.
  They satisfy
  $D(\BPen[r][\gamma_r])=D(\Iso[\gamma_r])=2R(\Iso[\gamma_r])\cos(\nicefrac{\gamma_r}{2})$,
  $r(\BPen[r][\gamma_r])=r$, and $w(\BPen[r][\gamma_r])=2r(\BPen[r][\gamma_r])$ and are
  extreme for the inequalities \eqref{lb_1} and \eqref{eq:3PT}.

  Defining the bent isosceles $\BIso[r][\gamma_r]:=\conv(\Iso[\gamma_r], c+r\B)$
  (as we will do for $(lb_2)$), we obtain $f(K)=f(\BPen[r][\gamma_r])$, iff
  $\BIso[r][\gamma_r] \subset K \subset \BPen[r][\gamma_r]$.
\end{itemize}

\subsection{Facets of the diagram}\label{ss:facets}

In this section families of sets $\CK^2$ are described for each of the inequalities stated in Section \ref{s:main ineq},
\st~for every point $x \in[0,1]^3$ in the induced facet,
there exists a set $K_x$ within the family with $f(K_x)=x$.

\begin{itemize}
\item[$(lb_1)$] Due to Lemma \ref{lem:starshaped}, all
  outer parallel bodies $K$ of the bent trapezoids $\BTrap$
    of the $(\L,\BT)$-edge
  and the bent pentagons $\BPen$
    of the $(\BT,\H)$-edge
  fulfill the equation \[2r(K) = w(K).\]
  This means \eqref{lb_1} induces a linear facet of the diagram, which is bounded by the edges
  $(\L,\BT)$ (bent trapezoids with $\gamma \le \nicefrac{3}{4}$), $(\BT,\H)$
  (bent pentagons),
  $(\L,\B)$ (sausages) and $(\H,\B)$
  (rounded hoods).

\item[$(ib_1)$] If $K$ is an outer parallel body of a right-angled triangle $\RecT$
  or a right-angled sailing-boat $\RSB$ as described in Section \ref{ss:edges},
  then Lemma \ref{lem:starshaped} ensures
  \[D(K)=2R(K),\]
  which means equality in \eqref{ib_1}. Hence it is a linear facet of the diagram
  bounded by the edges $(\L,\B)$ (sausages),
  $(\L,\RAT)$ (right-angled triangles),
  $(\RAT,\SB)$ (right-angled sailing-boats), and $(\SB,\B)$ (rounded sailing boats).

\medskip

To both facets, $(lb_1)$ and $(ib_1)$, much more sets are mapped. Remember that, \eg, all symmetric
  sets are mapped to the edge obtained from the intersection of the two facets.

\medskip

\item[$(ub_1)$] \label{ub1-facet}
 Because of Lemma \ref{lem:starshaped} and Lemma \ref{lem:CompletionsW=r+R}
 any outer parallel body $K$ of a Reuleaux blossom $\ReB$ or a concentric sailing boat $\CSB$,
 as well as any of the sets $K_\lambda:=\lambda K+(1-\lambda)C_K$, $\lambda\in[0,1]$,
   where $C_K$ is a Scott-completion of a concentric sailing boat $\CSB$ fulfills $w(K)=r(K)+R(K)$.
 Thus \eqref{ub_1} defines a linear facet of the diagram, bounded by the edges $(\EqT,\ReT)$
 (Reuleaux blossoms) and $(\ReT,\B)$ (rounded Reuleaux triangles),
 as well as $(\EqT,\SB)$ (concentric sailing boats) and $(\SB,\B)$ (rounded sailing boats).

\item[$(ib_2)$] Lemma \ref{lem:starshaped} ensures that any outer parallel body $K$ of a
  general hood $\Hood$ or a concentric sliced Reuleaux triangle $\CSRT$ fulfills
  \[D(K)=r(K)+R(K),\]
  filling the linear facet from the star-shapedness with respect to $\B$.
  Moreover, the sets $K_\lambda:=\lambda K+(1-\lambda)C_K$, $\lambda\in[0,1]$,
  where $K$ is a set from the edges $(\SRT,\H)$ or $(\H,\B)$ and
  $C_K$ its Scott-completion, all fulfill \[D(K)=r(K)+R(K)\] too, filling the facet in horizontal
  lines with respect to the inradius-axis.

  Hence \eqref{ib_2} induces the fourth linear facet. Its boundary edges are  $(\SRT,\H)$ (general hoods),
  $(\H,\B)$ (rounded hoods), $(\SRT,\ReT)$ (concentric sliced Reuleaux triangles),
  and $(\ReT,\B)$ (rounded Reuleaux triangles).

\item[$(ib_3)$] As shown in \cite{Ju} a set $K \in \CK^2$ fulfills
  \[D(K)=\sqrt{3}R(K),\]
  iff $K$ contains an equilateral triangle $\EqT$ of the same circumradius.
  Since $\ReT$ is the unique Scott-completion of $\EqT$,
  we obtain $\EqT\subseteq K\subseteq\ReT$.

  Consider a Reuleaux blossom $\ReB=2r\EqT\cap\ReT$ with $r\in[r(\EqT),r(\ReT)]$.
  We describe a continuous transformation of $\ReB$, keeping
  its inradius, diameter, and circumradius constant and decreasing its width until
  it becomes a set from the edge $(\EqT,\FRT)$ or $(\FRT,\SRT)$.
  Let $p^i$, $i=1,2,3$, \st~$\EqT=\conv\{p^1,p^2,p^3\}$.
  While the transformation ending in the sets from the edge $(\EqT,\FRT)$
  can be done within one step (Step (i) below)
  the transformation of the sets which should approach the edge $(\FRT,\SRT)$
  must be done in two steps (Step (i) and (ii) below):
  \begin{enumerate}[(i)]
  \item We translate $2r\EqT$ in direction of $p^1$,
    until either its inball becomes tangent to $[p^2,p^3]$ (when $r(\EqT)\leq r\leq r(\FRT)$)
    or tangent to both
    arcs of $\ReT$ intersecting in $p^1$ (when $r(\FRT)\leq r\leq r(\SRT)$, see Figure \ref{fig:ib3ONE}).
    We define the \cemph{dblue}{(non-concentric) Reuleaux blossom} by $\ReB[r,v]=(v+2r\EqT)\cap\ReT$,
    where $v$ is a
    point on the segment $[0,tp^1]$ with $0 \le t <1$ chosen, \st~in case of $v=tp^1$ one
    of the two stopping reasons for the translation is reached (\cf~Figure \ref{fig:ib3ONE}).

    Observe that when $r(\EqT)\leq r\leq r(\FRT)$ all radii of $\ReB[r,tp^1]$ coincide with the ones
    of a bent equilateral $\BEq$ (\cf~Figure \ref{fig:BEqONE}),
    which means that we finished the transformation.

  \item In case of $r(\FRT)\leq r\leq r(\SRT)$ we further need to reduce the width.
    However, since the tangent lines to the inball do not support the diameter arcs of
    $\ReT$ intersecting in $p^1$, we first
    \enquote{fill} the space between $\ReB[r,tp^1]$ and these arcs, keeping all radii constant,
    but obtaining a maximal set.
    Afterwards let $L$
    be a line containing $p^2$ and cutting the extreme Reuleaux blossom $\ReB[r,tp^1]$,
    \st~the distance of $p^1$ and $L$ is the same as the width of $\ReB[r,tp^1]$.
    Then we rotate $L$ continuously until it becomes tangent to the inball of $\ReB[r,tp^1]$
    (see Figure \ref{fig:ib3TWO}).
    Denoting the halfspace induced by $L$ containing the inball $L^-$, we define the
    \cemph{dblue}{general sliced Reuleaux triangle} as
    $\Sliced= \ReB[r,tp^1] \cap L^-$.
    Finally, when $L^-$ becomes tangent to the inball, we need to \enquote{fill} again, this time
    all the space of $\ReT$ inside $L^-$.
    Observe that in that moment the general sliced Reuleaux triangle reaches the edge $(\FRT,\SRT)$
    becoming a maximally sliced Reuleaux triangle
    and that starting with the Reuleaux triangle the general sliced Reuleaux triangles
    get concentric ones and approach $\SRT$.
  \end{enumerate}
  Observe that in that moment the general sliced Reuleaux triangle reaches the edge $(\FRT,\SRT)$
  becoming a maximally sliced Reuleaux triangle
  and that starting with the Reuleaux triangle the general sliced Reuleaux triangles
  get concentric ones and approach $\SRT$.

  Both, non-concentric Reuleaux blossoms and general sliced Reuleaux triangles are maximal sets
  with respect to set inclusion.
  The \emph{corresponding} minimal sets are the convex hull
  of $\conv(\EqT, v+r\B)$  with the intersection of the three balls with radius $w(\ReB[r,v])$ or $w(\Sliced)$,
  depending if we are in case (i) or (ii), around the vertices of $\EqT$.

  \begin{figure}
    \centering
    \begin{subfigure}[b]{0.45\textwidth}
      \hspace{-0.5cm}\includegraphics[trim = 1cm 4cm 1cm 2.5cm, width =1.1\textwidth]{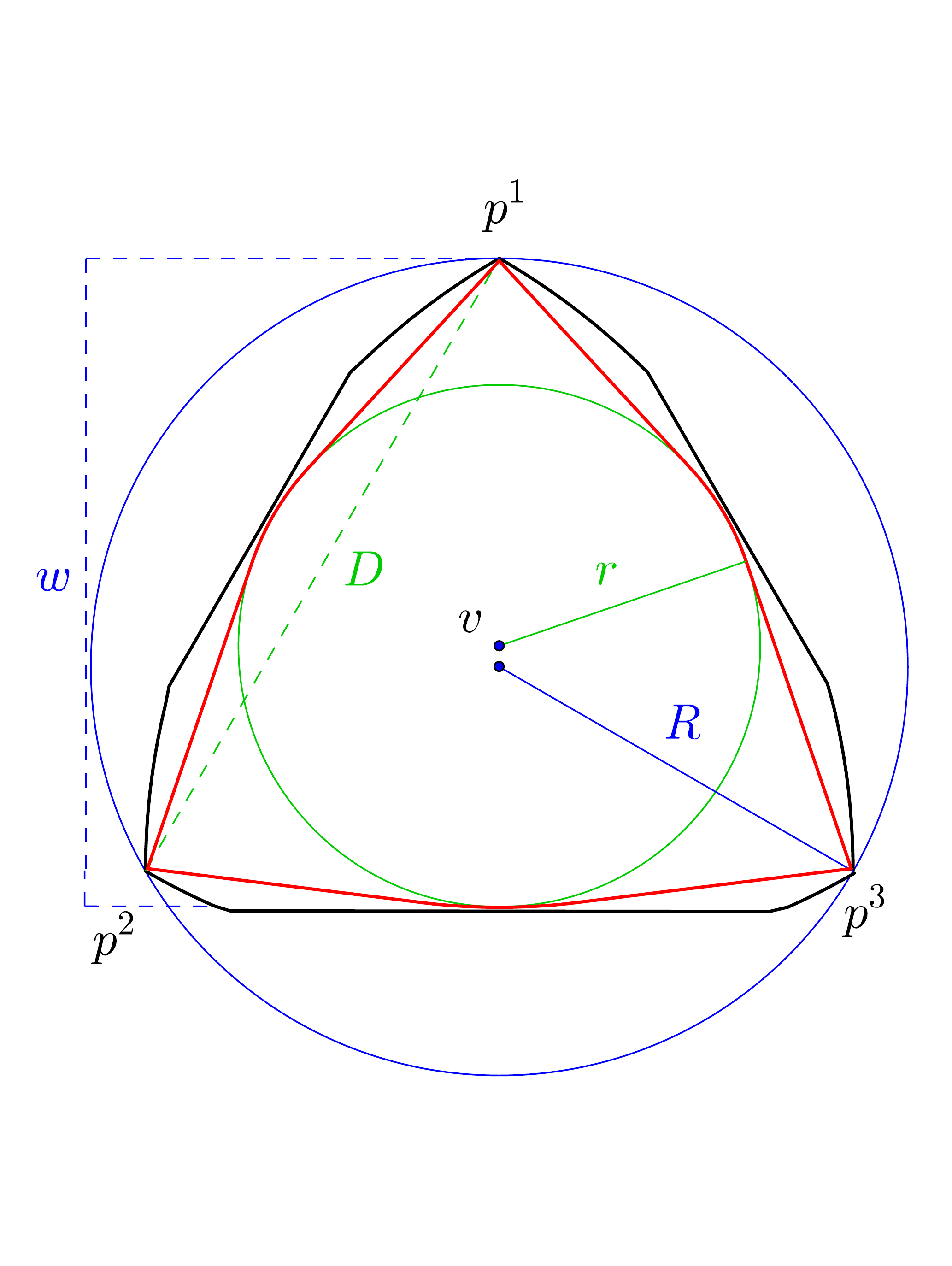}
      \caption{In black a non concentric Reuleaux blossom and in red the corresponding minimal set.}
      \label{fig:ib3ONE}
    \end{subfigure} \hfill
    \begin{subfigure}[b]{0.45\textwidth}
      \hspace{-0.5cm}\includegraphics[trim = 1cm 4cm 1cm 2.5cm, width = 1.1\textwidth]{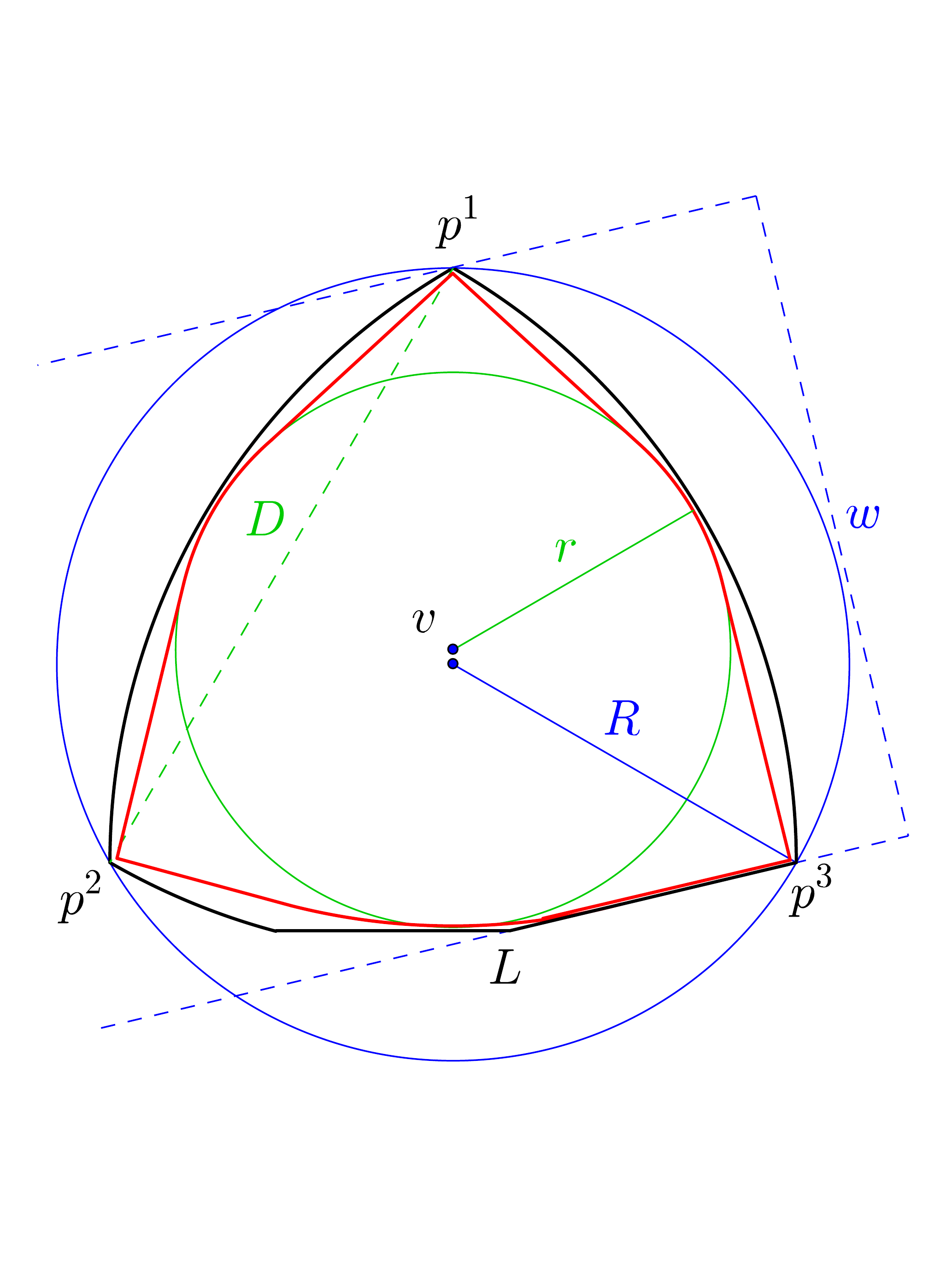}
      \caption{In black a sliced Reuleaux triangle and the corresponding minimal set in red.\\\hfill}
      \label{fig:ib3TWO}
    \end{subfigure}
    \caption{Examples for the sets, which are maped onto $(ib_3)$, corresponding to the cases
        (i) and (ii) in the description.}
  \end{figure}

\item[$(lb_2)$] It was shown in \cite{HCS} that every isosceles $\Iso$, $\gamma\in[0,\nicefrac{\pi}{3}]$,
  fulfills
  \[(4R(K)^2-D(K)^2)D(K)^4 =  4w(K)^2R(K)^4\]
  with equality. But as already described in \cite{Br} they are not the only ones.
  Since $r$ does not appear in this inequality any superset of an isosceles $\Iso$
  keeping the same circumradius, diameter, and width is mapped to the same facet. This is
  true, \eg for all bent trapezoids $\BTrap$ on the edges $[\L,\BT]$ and $[\BT,\FRT]$
  and surely also for any minimal version $\conv(\Iso, c_\gamma + r(\BTrap)\B)$, where $c_\gamma$
  denotes an incenter of $\BTrap$.
  Thus choosing any $r \in [r(\Iso),r(\BTrap)]$ and an appropriate incenter $c$
  (which in many cases will not be unique, as the centers $c_\gamma$ of $\BT_\gamma$
  where not always unique) the sets $\conv\left(\Iso, c+r\B\right)$
  would have inradius $r$ and the same circumradius, diameter, and width
  than $\Iso$ and $\BTrap$ (see Figure \ref{fig:lb2}). Hence the facet induced by \eqref{lb_2}
  is filled by those sets and bounded by the edges $(\L,\EqT)$
  (isosceles triangles with $\gamma \in[0,\nicefrac{\pi}{3}]$),
  $(\L,\BT)$,  and $(\BT,\FRT)$  (both kinds of bent trapezoids), as well as $(\EqT,\FRT)$
  (bent equilaterals with the inball being tangent to an edge of $\EqT$).

  Finally, one should recognize that for any fixed center $c$
  the sets $\conv\left(\Iso, c+r\B\right)$
  are minimal sets with respect to set inclusion mapped to these coordinates in the diagram
  and are constructed in the same way than the bent isosceles in $(lb_3)$ below.

  \begin{figure}
    \begin{center}
      \includegraphics[trim = 0cm 3cm 0cm 3cm, width=8cm]{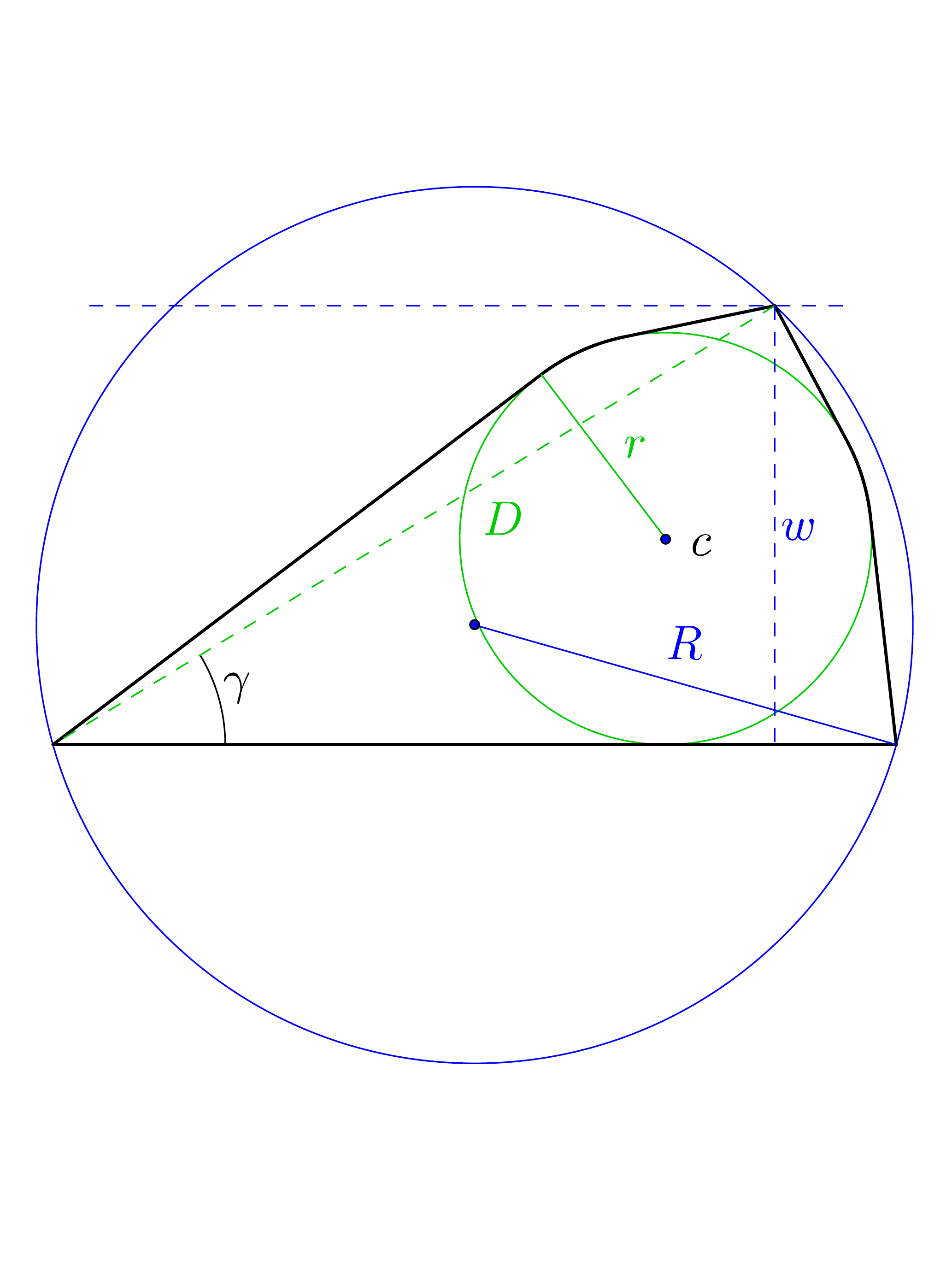}
      \caption{An example of a minimal set from $(lb_2)$.}\label{fig:lb2}
    \end{center}
  \end{figure}

\item[$(lb_3)$]
  For any $r\in[0,1]$ and $\gamma\in[0,\nicefrac{\pi}{3}]$ let
  \begin{enumerate}[(i)]
  \item $p^1,p^2,p^3 \in \S$, \st~$\Iso =\conv\{p^1,p^2,p^3\}$ with $D(\Iso) = \norm[p^1-p^2]=\norm[p^1-p^3]$,
  \item $c\in\R^2$, \st~the ball $c+r\B$ is tangent to the two arcs with centers $p^1,p^2$ and radius $D(\Iso)$,
  \item $L_1$ be the one of the two lines containing $p^2$ and supporting $c+r\B$ having the smaller angle
    with $[p^1,p^2]$ and
  \item $L_2$ be the parallel line of $L_1$ passing through $p^3$.
  \end{enumerate}
  Then a \cemph{dblue}{generalized bent pentagon} $\BPen$ is defined as the space contained between
  $L_1,L_2$ and the arcs of radius $D(\Iso)$ around the centers $p^1,p^2$, and $p^3$
  (see Figure \ref{fig:gbp}).

  If we can ensure that $\Iso \subset \BPen$, that $c+r\B$ is the inball of $\BPen$,
  and that $w(\BPen)=d(L_1,L_2)$, we simply call it a \cemph{dblue}{bent pentagon}
  (see Figure \ref{fig:lb_3}).

  \begin{figure}
    \centering
    \begin{subfigure}[b]{0.45\textwidth}
      \hspace{-0.5cm}\includegraphics[trim = 1cm 4cm 1cm 2.5cm, width =1.1\textwidth]{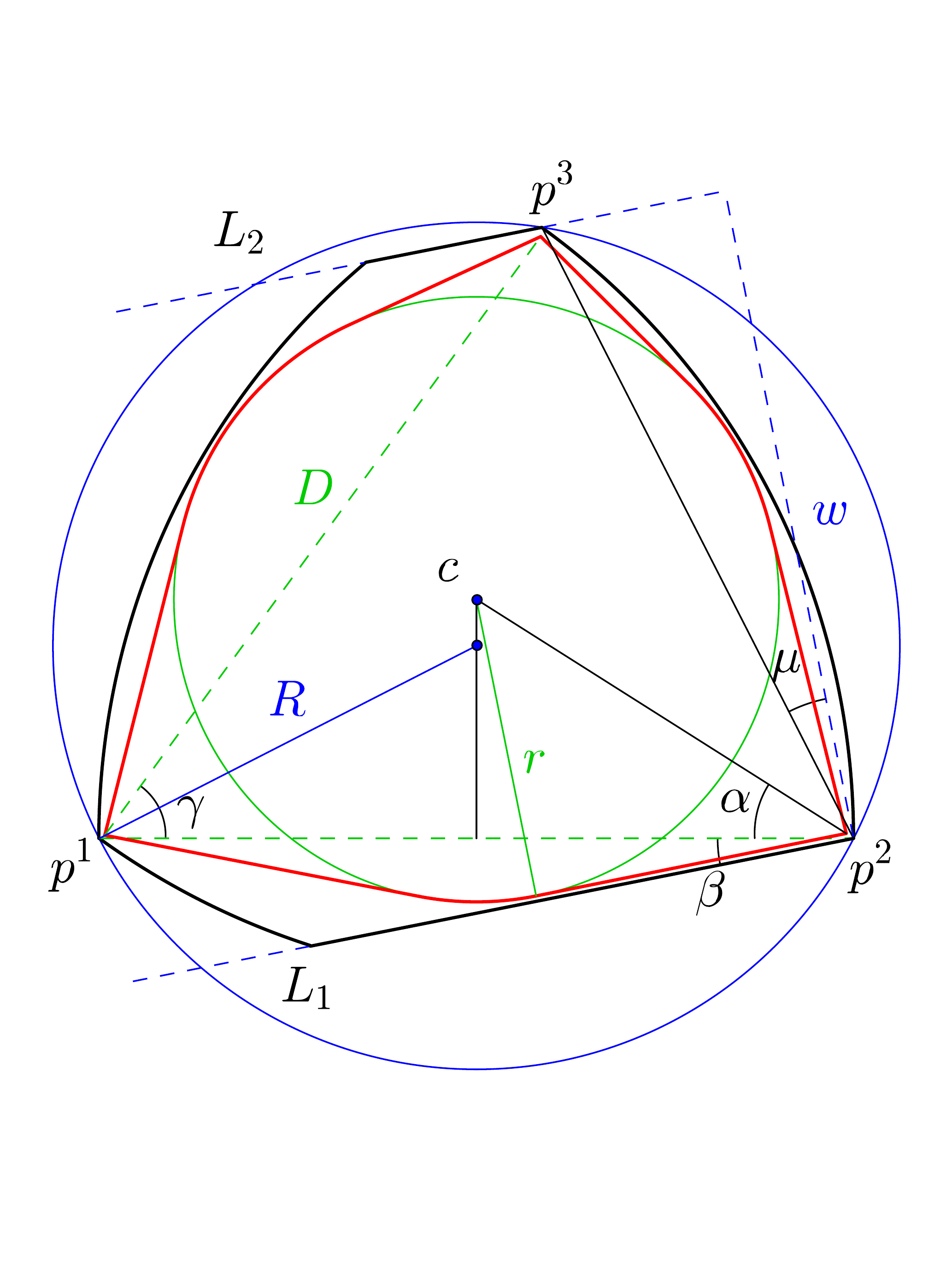}
      \caption{A bent pentagon $\BPen$ (black) and a bent isosceles $\BIso$ (red), the maximal and minimal
        sets mapped to the same coordinates in $(lb_3$).}
      \label{fig:lb_3}
    \end{subfigure} \hfill
    \begin{subfigure}[b]{0.45\textwidth}
      \hspace{-0.5cm}\includegraphics[trim = 1cm 4cm 1cm 2.5cm, width = 1.1\textwidth]{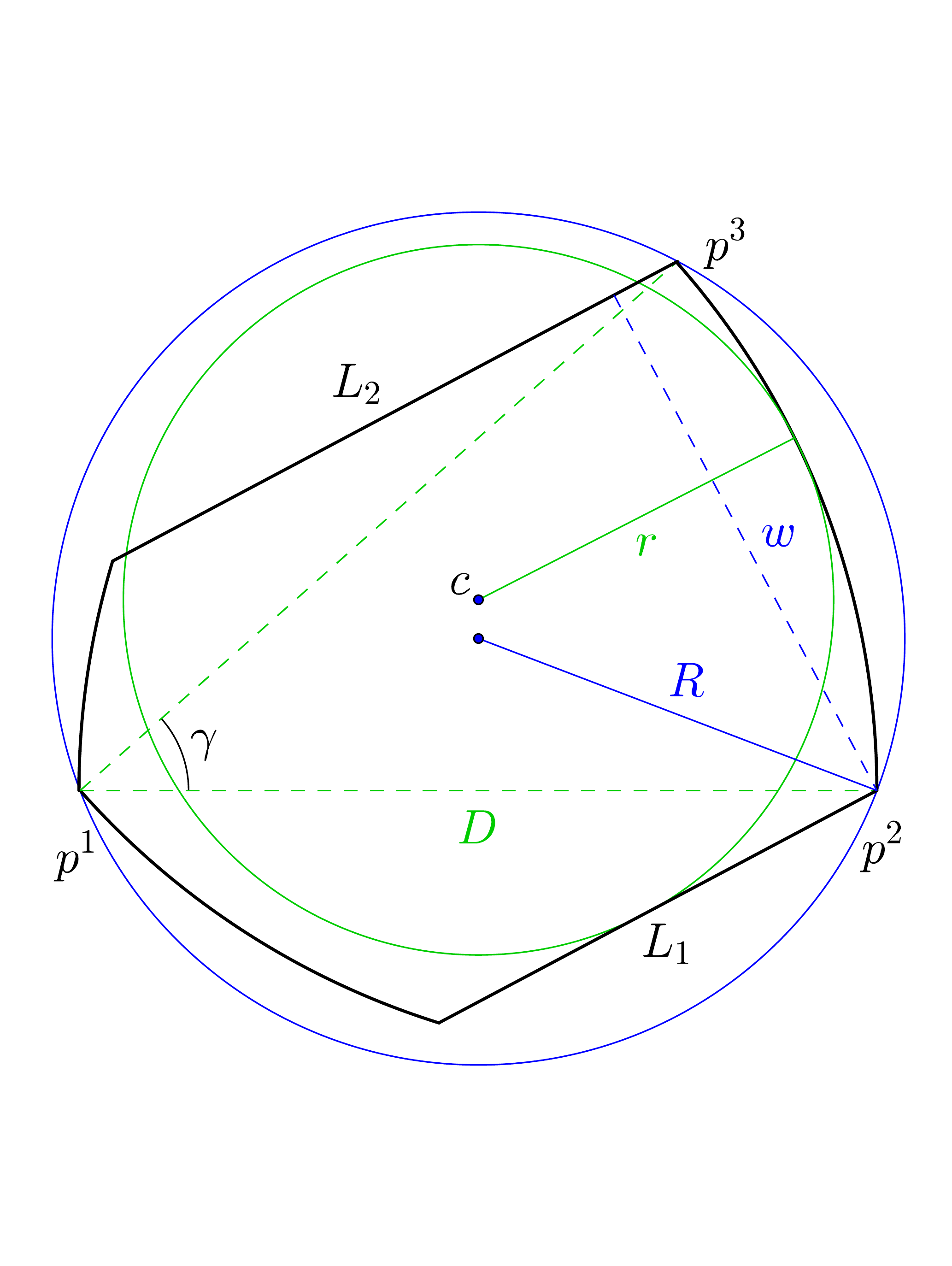}
      \caption{A generalized bent pentagon not being a bent pentagon as $r(\BPen)<r$.\\\hfill \\\hfill}
      \label{fig:lb_3NoN}
    \end{subfigure}
    \caption{Generalized bent pentagons}
    \label{fig:gbp}
  \end{figure}

  Recall the following edges:
  the bent trapezoids from $(\BT,\FRT)$, the bent pentagons from $(\BT,\H)$,
  the maximally-sliced Reuleaux triangles from $(\FRT,\SRT)$, and the general hoods from $(\SRT,\H)$.
  It is easy to check  from their construction that they all are particular cases of
  bent pentagons in the above sense.
  We will justify why they describe the boundaries of $(lb_3)$ in showing that
  they bound the range of the parameters $r, \gamma$, \st~a generalized bent pentagon is a bent pentagon
  (the bent pentagons
  and the bent trapezoids bound $\gamma$ from below while
  the general hoods and the maximally-sliced Reuleaux triangles
  bound $\gamma$ from above).

  \begin{lem}\label{lem:lb3properties}
    Let $r\in[0,1]$, $\gamma,\bar\gamma \in[0,\nicefrac{\pi}{3}]$ with $\gamma < \bar\gamma$,
      as well as $L_1, L_2$ and $\bar L_1,\bar L_2$ the corresponding parallels
      in the construction of the generalized bent pentagon $\BPen$ and $\BPen[r][\bar\gamma]$, respectively.
      Then
    \begin{enumerate}[a)]
    \item the ball $c+r\B$ used in the construction
      intersects (is tangent to) $[p^1,p^2]$, iff $8r \ge 3D(\Iso)$ ($8r = 3D(\Iso)$).
    \item if we restrict to the case $8r\ge 3D(\Iso)$, then it holds $d(L_1,L_2) < d(\bar L_1,\bar L_2)$.
    \end{enumerate}
  \end{lem}

  \begin{proof}
    \begin{enumerate}[a)]
    \item The distance from $c$ to $[p^1,p^2]$ is at most $r$, iff $[p^1,p^2]$ intersects $c+r\B$,
      that is, when
      $d(c,[p^1,p^2])^2 = (D(\BPen)-r)^2 - \nicefrac{1}{4} \, D(\BPen)^2 \le r^2$
      (\cf~the right-angled triangle $T=\conv\{c,p^2,\nicefrac12 \, (p^1+p^2)\}$ in Figure \ref{fig:lb_3}).
      From simplifying we obtain that this is equivalent to $\nicefrac{3}{4} \, D(\BPen)-2r \le 0$
      or $8r\geq 3D(\BPen)$ with equality, iff
      $r=d(c,[p^1,p^2])$,
      which means that the inball is tangent to $[p^1,p^2]$.
    \item We use the complete notation as in the construction of the bent pentagons, with a bar on top for
      $\BPen[r][\bar\gamma]$ and
      assume that $[p^1,p^2]$ as well as $[\bar{p}^1,\bar{p}^2]$ are horizontal, below $0$ with $p^1_1\leq p^2_1$
      and $\bar{p}^1_1\leq\bar{p}^2_1$. Then, it follows from Part (a) that
      all lines $L_i, \bar{L}_i$, $i=1,2$, have non-negative slope.
      Since the function $f(x)=(x-r)^2-\nicefrac{1}{4} \, x^2$ is increasing, if $x \geq 2r$ it follows
      \[\begin{split}
        d(c,[p^1,p^2]) &=\sqrt{(D(\BPen)-r(\BPen))^2-\nicefrac{1}{4} \, D(\BPen)^2} \\
        & > \sqrt{(D(\BPen[r][\bar\gamma])-r(\BPen[r][\bar\gamma]))^2-\nicefrac{1}{4} \, D(\BPen[r][\bar{\gamma}])^2}
        =d(\bar{c},[\bar{p}^1,\bar{p}^2]).
      \end{split} \]

      Using again the triangle $T$ defined above and the pythagorean theorem, we obtain
      \[p^2_2=-\sqrt{1-\nicefrac{1}{4} \, D(\BPen)^2} > -\sqrt{1-\nicefrac{1}{4} \, D(\BPen[r][\bar\gamma])^2}
      =\bar p^2_2.\]
      Moreover, since $\gamma<\bar\gamma$, rotating $\Iso$ around $\S$ until $p^1$ becomes $\bar p^1$,
      it follows $p^j, j=2,3$ belong to the smaller of the two arcs of $\S$ with endpoints $\bar p^j, j=2,3$
      and thus in particular it holds $p^3_2 < \bar p^3_2$ after the rotation.
      Undoing the rotation, \ie $p^1$ moves upward and $p^2, p^3$ downwards into their old positions,
      it still holds $p^3_2<\bar p^3_2$ and thus also
      both points $p^2,p^3$ still lie in the shorter arc of $\S$ with endpoints $\bar p^j, j=2,3$.
      Now, it follows from $\gamma <\bar\gamma$ that $\norm[p^1-p^2] > \norm[\bar{p}^1-\bar{p}^2]$,
      which together
      with $d(c,[p^1,p^2])> d(\bar{c},[\bar{p}^1,\bar{p}^2])$ means that the slope of
      $L_1$ is less than the one of $\bar{L}_1$.
      Using this fact, we see that if one rotates $\bar{L}_i$, $i=1,2$,
      around $\bar{p}^i$, $i=2,3$, \st~they become parallel
      to $L_i$, $i=1,2$, their distance decreases, but is still bigger than the distance between $L_1$ and $L_2$.
      Hence
      $w(\BPen)=d(L_1,L_2)< d(\bar{L}_1,\bar{L}_2)=w(\BPen[r][\bar{\gamma}])$.
    \end{enumerate}
  \end{proof}

  Hence we see that only if Part (a) of Lemma \ref{lem:lb3properties} holds
  (which is, because of $D(\Iso)=2 R(\Iso) \cos(\nicefrac{\gamma}{2})$, equivalent to
   $\gamma \ge 2\arccos\left(\nicefrac{4}{3}\,r\right)$), we have $\Iso\subset\BPen$, the latter
   implying that $R(\BPen)=R(\Iso)$ and $D(\BPen)=D(\Iso)$.

  Now considering $c+r\B$, we show that it is the inball of $\BPen$
  (which means that $r(\BPen)=r$), whenever $r, \gamma$ are in the range described by the edges above.
  To do so, it is enough to show that $L_2$ does not intersect the interior of $c+r\B$.
  However, using Part (b) of Lemma \ref{lem:lb3properties},
  it follows that if $r, \gamma$ determine a bent pentagon with maximal
  $\gamma$ depending on $r$
  (\ie $\BPen$ belongs to $(\FRT,\SRT)$ or $(\SRT,\H)$) then $L_2$ does not intersect $c+r\B$.
  Decreasing $\gamma$ decreases monotonously $d(L_1,L_2)$ until $\BPen$ becomes a set
  from $(\BT,\FRT)$  or
  $(\BT,\H)$ and in both cases $L_2$ does not intersect $c+r\B$ at any point of the transformation
  (except for the sets in $(\BT,\H)$, where it becomes tangent).

  Finally, from Part (b) of Proposition \ref{prop:exposed}, we know that the width of $\BPen$
  must be attained between two supporting parallel lines touching the endpoints of a
  perpendicular segment in $\BPen$.
  However, considering the construction of the generalized bent pentagons,
  any such pair of parallel supporting
  lines, except $L_1,L_2$, touches an arc of $\BPen$ and the vertex it is drawn around,
  therefore having a distance of $D(\BPen) \ge d(L_1,L_2)$ (\cf~Figure \ref{fig:lb_3}).
  (Observe that this argument fails if the pentagon does not fulfill Part (a)
  of Lemma \ref{lem:lb3properties}, as $p^1$ would not belong to $\BPen$ anymore.)
  Hence $w(\BPen)=d(L_1,L_2)$.

  \medskip

  The given boundaries for the bent pentagons are best possible. Considering the upper bounds first,
  on the one hand $\gamma\le\nicefrac{\pi}{3}$ by defintion
  and for all $r \in [r(\FRT),r(\SRT)]$ this bound is reached by a maximally-sliced
  Reuleaux triangle $\SliRT = \BPen[r][\nicefrac{\pi}{3}]$.

  On the other hand, in case of $r \in [r(\SRT),r(\H)]$, inequality \eqref{ib_2} implies that
  $D(\BPen)\ge r + R(\BPen) = D(\BPen[r][2\arccos(\nicefrac{(r+1)}{2})])$, which
  together with $D(\BPen) = D(\Iso) = 2 R(\Iso)\cos(\nicefrac{\gamma}{2})$
  and $D(\Iso)$ descending as a function of $\gamma$,
  implies that $\gamma\le 2\arccos(\nicefrac{(r+1)}{2})$. Equality in this situation is attained
  by the general hoods.

  Regarding the lower bounds, in both cases choosing $\gamma$ below the given bound yields a generalized
  bent pentagon not being a bent pentagon:
  As already mentioned, Part (a) of Lemma \ref{lem:lb3properties} implies
  $\gamma \ge 2\arccos\left(\nicefrac{4}{3} \, r\right)$ in general.
  And in case of $r \in [r(\BT),r(\H)]$ choosing
  $2\arccos\left(\nicefrac{4}{3}\,r\right) \le \gamma<\bar\gamma=\gamma_r$,
  Part (b) of Lemma \ref{lem:lb3properties} says that
  $d(L_1,L_2) < d(\bar L_1,\bar L_2)$. But since $L_1$ supports $c+r\B$ and both $\bar L_i, i=1,2$ support the inball
  of $\BPen[r][\gamma_r]$, it follows that $L_2$ would intersect the interior of $c+r\B$.

  For the computation of the radii we denote the angle in $p^2$ between $[p^1,p^2]$
  and $[c,p^2]$ by $\alpha$, the angle in $p^2$ between $[p^1,p^2]$ and $L_1$ by $\beta$,
  as well as the angle in $p^2$ between $[p^2,p^3]$ and the line perpendicular to
  $L_1$ by $\mu = \nicefrac{\gamma}{2} -\beta$ (\cf Figure \ref{fig:lb_3}).
  Omitting again the argument $\BPen$ in the radii functionals, it holds
    \[\text{(i)} \;~ \cos(\alpha)=\frac{D}{2(D-r)}, \quad \text{(ii)} \;~ \sin(\alpha+\beta)=\frac{r}{D-r},
     \quad \text{(iii)} \;~ \cos(\mu) = \frac{w}{\norm[p^2-p^3]}.
    \]
  From (i) and (ii) we obtain that
  $\beta=\arcsin\left(\nicefrac{r}{D-r}\right)-\arccos\left(\nicefrac{D}{2(D-r)}\right)$,
  which together with
  $\gamma=2\arccos(\nicefrac{D}{2R})$ implies
  $\mu=\frac \gamma 2 - \beta = \arccos\left(\nicefrac{D}{2R}\right) +
  \arccos\left(\nicefrac{D}{2(D-r)}\right)  -\arcsin\left(\nicefrac{r}{D-r}\right)$.
  Inserting $\mu$ and
  $\norm[p^2-p^3]=2D\sqrt{1-\left(\nicefrac{D}{2R}\right)^2}$ into (iii)
  results in
  \begin{equation} \label{eq:width-bentpent}
    \begin{split}
      & w = 2D\sqrt{1-\left(\nicefrac{D}{2R}\right)^2}  \\
      & \cos\left(\arccos\left(\frac{D}{2(D-r)}\right)
        + \arccos\left(\frac{D}{2R}\right)-\arcsin\left(\frac{r}{D-r}\right)\right).
    \end{split}
  \end{equation}

  Thus each $\BPen$ satisfies \eqref{eq:3PT} with equality.

  Again, we also define the bent isosceles $\BIso:=\conv(\Iso, c+r\B)$, which obviously
  fulfill  $R(\BIso)=R(\BPen)$, $D(\BIso)=D(\BPen)$, and $r(\BIso)=r(\BPen)$.
  Using Lemma \ref{lem:lb3properties},
  we know that $c+r\B$ intersects all three edges of $\Iso$. However,
  from Part (b) of Proposition \ref{prop:exposed} it follows, that
  the width of $\BIso$ is necessarily attained between a parallel pair of
  lines, from which one supports the inball and a vertex and the other a different vertex.
  Doing a direct comparison
  among the six pairs of such parallel supporting lines, we easily obtain
  $w(\BIso)=d(L_1,L_2)=w(\BPen)$ (\cf~Figure \ref{fig:lb_3}).
  Hence it holds $f(K)=f(\BPen)$, iff $\BIso \subset K \subset \BPen$.

  \begin{figure}
    \centering
    \includegraphics[trim = 1cm 11cm 1cm 10cm, width =13cm]{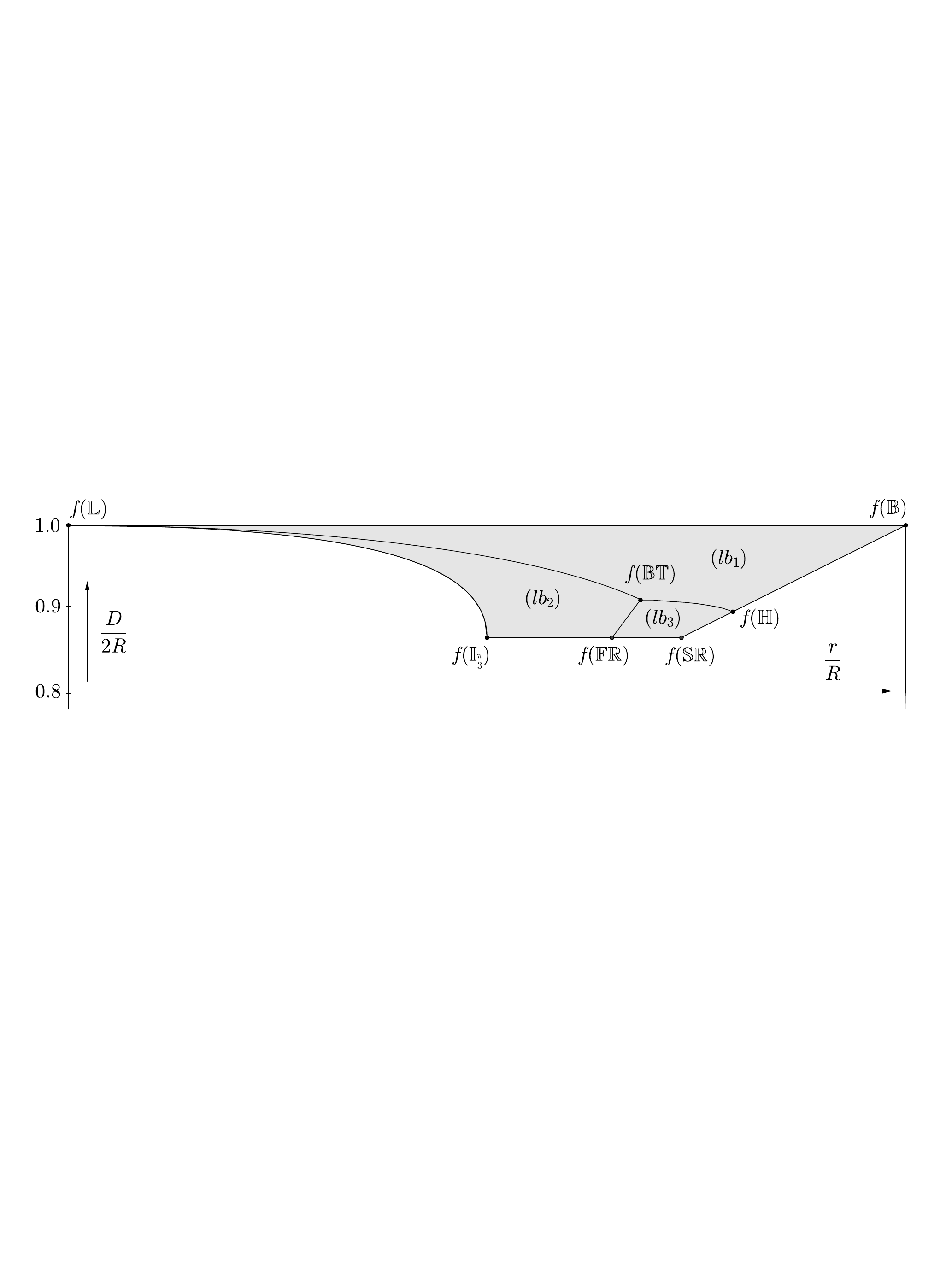}
    \caption{Bottom view of the diagram $f(\CK^2)$.}
    \label{fig:lowboundary}
  \end{figure}

\item[$(ub_2)$] \label{sailing-boat-facet}

  Let $\gamma\in[\nicefrac{\pi}{3},\nicefrac{\pi}{2}]$, $r\in[r(\Iso),r(\CSB)]$, and $p^1,p^2,p^3$,
  \st~$\conv\{p^1,p^2,p^3\}=\Iso$.
  Then $I_K = \frac{r}{r(\Iso)}(\Iso-p^3)  + p^3 = \conv\{q^1,q^2,p^3\}$
  is an isosceles triangle of inradius $r$, \st~$q^i = \frac{r}{r(\Iso)} p^i + (1-\frac{r}{r(\Iso)}) p^3$,
  $i=1,2$ (\cf~Figure \ref{fig:SBoat}).
  We call the sets $\SBoat=I_K \cap \B$
  \cemph{dblue}{(general) sailing boats},
  generalizing the concentric and right-angled
  sailing boats which are mapped to the edges $(\EqT,\RAT)$ and $(\EqT,\SB)$.

  It follows directly from the definition
  that $p^1 \in[q^1,p^3] \cap \S$ and $p^2\in[q^2,p^3] \cap \S$ and thus
    $R(\SBoat) = R(\Iso)$, $D(\SBoat) = D(\Iso) = 2R(\SBoat)\sin(\gamma)$ and $r(\SBoat) = r(I_K) = r$.
  Moreover, since $\Iso \subset \SBoat \subset \CSB$,
  the width of $\SBoat$ is obviously
  taken between $[q^1,q^2]$ and $p^3$,~\st
    \begin{equation*}
      w(\SBoat) = r \frac{w(\Iso)}{r(\Iso)}  = r \left(1+\frac{1}{\sin(\nicefrac{\gamma}{2})}\right) 
      = r\left(1+\frac{2\sqrt{2}R}{D}\sqrt{1+\sqrt{1-\left(\frac{D}{2R}\right)^2}}\right).
    \end{equation*}

  Thus all general sailing boats $\SBoat$ are extreme for the inequality \eqref{eq:sailing boats}.

  Since $\SBoat[r(\Iso)][\gamma]=\Iso$, $\SBoat[r(\CSB)][\gamma]=\CSB$,
  and $\SBoat[r][\nicefrac{\pi}{2}]$ a right-angled sailing-boat, the edges
  $(\EqT,\RAT)$, $(\EqT,\SB)$, and $(\RAT,\SB)$ form the boundaries of this facet.

  While it holds $K \subset \SBoat$ for all set $K$ with $f(K)=f(\SBoat)$,
    in general there do not exist unique minimal sets,
    as we have already discussed for the edge $(\RAT,\RSB)$. However,
    if $w(\SBoat) \le \norm[p^1-p^3]$, the minimal set
    $\conv\left(\Iso, (p^3 + w(\SBoat)\B)\cap\SBoat \right) $
    is unique (\cf~Figure \ref{fig:SBoat}).

  \begin{figure}
    \centering
    \begin{subfigure}[b]{0.45\textwidth}
      \includegraphics[trim = 1cm 3.5cm 1cm 2cm, width =\textwidth]{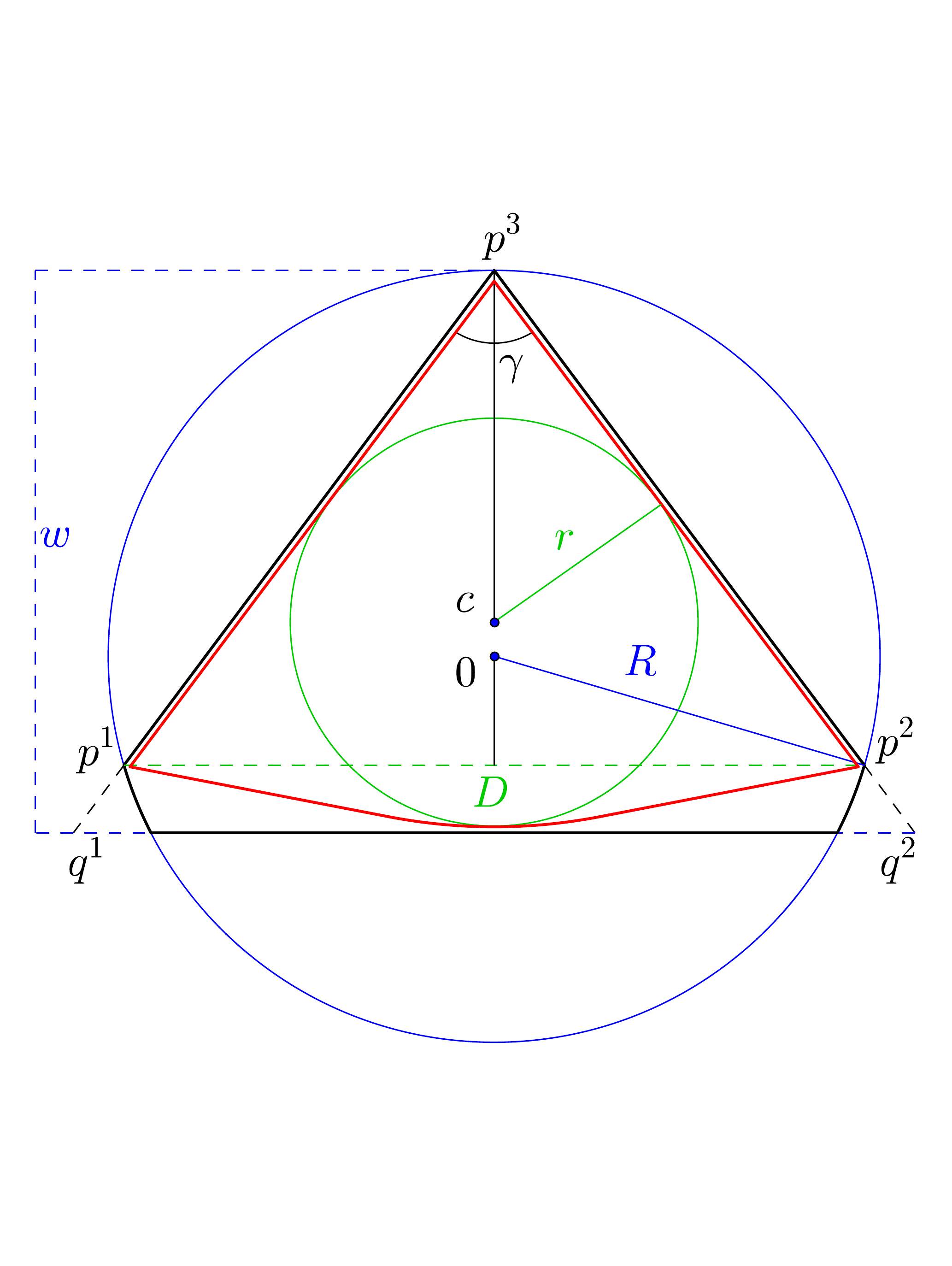}
      \caption{A general sailing boat $\SBoat$
        in black, and a corresponding minimal set in red.}
      \label{fig:SBoat}
    \end{subfigure} \hfill
    \begin{subfigure}[b]{0.45\textwidth}
      \includegraphics[trim = 1cm 3.5cm 1cm 2cm, width = \textwidth]{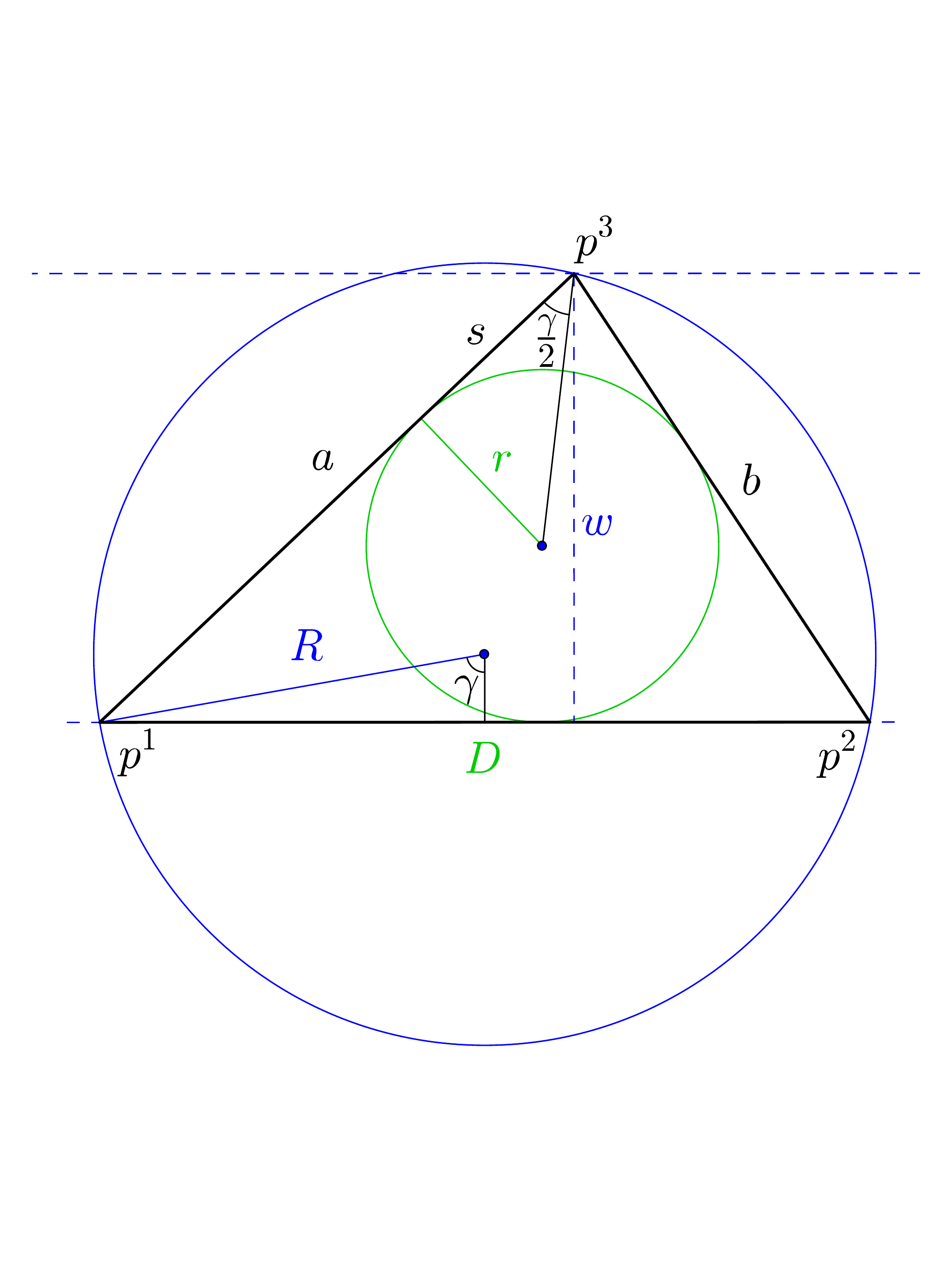}
      \caption{An acute triangle $\Triangle$.\\\hfill}
      \label{fig:Triangle}
    \end{subfigure}
    \caption{The general sailing boats and acute triangles fill the two remaining facets of the upper boundary.}
  \end{figure}

\item[$(ub_3)$] Any \cemph{dblue}{acute triangle} is circumspherical, \ie~all its vertices are situated on the
  circumsphere. For any $D\in[\nicefrac {\sqrt{3}}2, 1]$
  consider the two angles
  $0 \le \gamma_1 \le \nicefrac{\pi}{3} \le \gamma_2 \le \nicefrac{\pi}{2}$,
  \st~$D(\Iso[\gamma_1]) = D(\Iso[\gamma_2]) = D$.
  It is easy to see that for any $r\in[r(\Iso[\gamma_1]),r(\Iso[\gamma_2])]$ there exists an 
  acute triangle $\Triangle$ with the same circumradius and diameter as
  $\Iso[\gamma_1]$ and $\Iso[\gamma_2]$ and inradius $r$.

  Since every acute triangle is enclosed (in the above sense)
  between two isosceles triangles with the same diameter and circumaradius,
  the edges $(\L,\EqT)$, $(\EqT,\RAT)$ (both kinds of isosceles triangles) and the edge
  $(\L,\RAT)$ (right-angled triangles) form the relative boundary of this facet.

  Let $\gamma$ denote the angle of $\Triangle$ at the vertex $p^3$, opposing the diametral edge
  $[p^1,p^2]$ and $s$ the distance within the other two edges of $p^3$ to the touching points of the inball
  (see Figure \ref{fig:Triangle}).
  Then, as we have used already in the computations of the edge $(\L,\RAT)$
  in \ref{ss:edges},
  the perimeter of $\Triangle$ is $2(s+D)$. Thus using the semiperimeter formula for the area of a triangle,
  Proposition \ref{lem:angle}
  and simple trigonometry, we get
  \[\text{(i)} \;~  wD=2r(s+D), \quad \text{(ii)} \;~ D=2R\sin(\gamma) \quad \text{(iii)} \;~
  r=s\tan\left(\nicefrac{\gamma}{2}\right),\]
  (\cf~Figure \ref{fig:Triangle}).
  Now, substituting
  the value of $s$ in (i) by $s=\frac{r}{\tan\left(\nicefrac{\gamma}{2}\right)}$ obtained from (iii), while
  using (ii) to replace $\gamma$, we finally arrive in
  \[
  wD=2r\left(D+\frac{r}{\tan\left(\frac{1}{2}\arcsin\left(\frac{D}{2R}\right)\right)}\right) =
  2r \left(D + \frac{2rR}{D} \left(1 + \sqrt{1-\left(\frac{D}{2R}\right)^2}\right)
      \right).
  \]

\end{itemize}

\begin{figure}
  \centering
  \includegraphics[trim = 1cm 11cm 1cm 10cm, width =13cm]{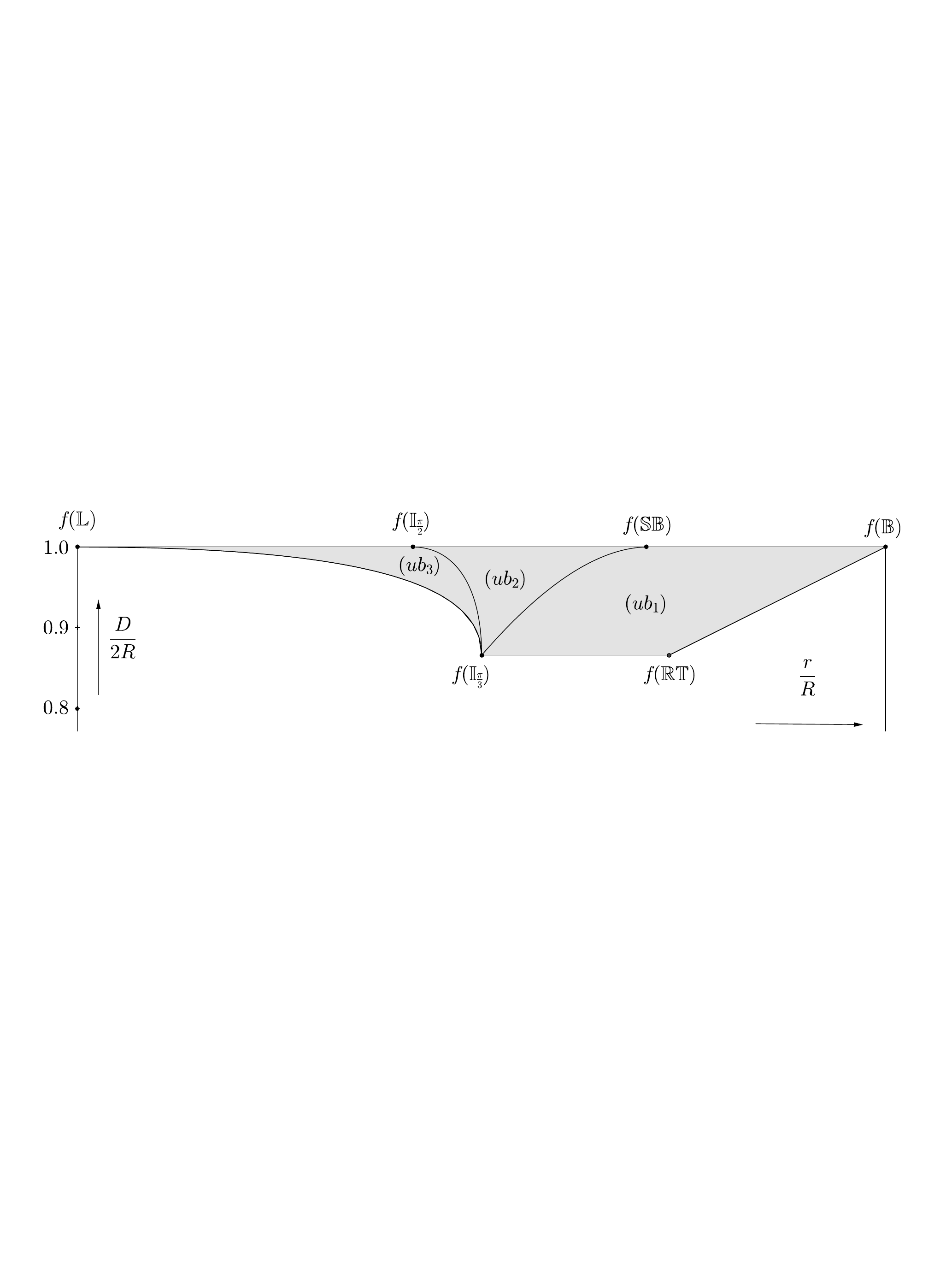}
  \caption{Top view of the diagram $f(\CK^2)$.}
\end{figure}

\section{Proofs of the main results}\label{s:proofs}

In this section we give the proofs of the main theorems. For preparation, we first state a corollary
and some technical lemmas.

\begin{cor} \label{cor:separate}
  Let $K\in\CK^n$ and $c\in\R^n$, \st~$c+\rho\B \subseteq K\subset \B$,
  $p^1,\dots,p^k$, $u^1,\dots,u^l$ as in Proposition \ref{prop:CONT},
  $T =c + \bigcap_{i=1}^l \{x\in\R^n:(u^i)^Tx \le \rho \}$
    and $T'= \conv\{p^1,\dots,p^k\}$.
  Then
  \begin{enumerate}[a)]
  \item at least two of the vertices of $T$ do not belong to $\inte(\B)$, and
  \item $T'$ separates $\bd(T)$ from 0.
  \end{enumerate}
\end{cor}

\begin{proof}
 Both statements follow directly from $0 \in T' \subset K \subset T$,
 recognizing that, if all but at most one vertex of $T$ would belong to $\inte(\B)$, it would follow that $R(K)<1$, a contradiction.
\end{proof}

While Proposition \ref{prop:CONT} in Section \ref{s:intro}
deduces properties of the inner and outer radii separately from their
definitions, Corollary \ref{cor:separate} already combines them.
In the following lemmas, we derive some properties from the
interaction between both of them and the diameter.

We recall that we always assume
$\B$ to be the circumball of $K$
even though keeping the value $R(K)$ in the equations.

\begin{lem}\label{lem:semisphere}
  Let $K\in\CK^n$ and $c\in\R^n$, \st~$c+r(K)\B \subset K \subset \B$.
  Then there exist $u^1,\dots,u^l$ and $T$ as in Corollary \ref{cor:separate},
  as well as $u \in \S$ \st~$\S^\geq_u \subset T\cap \S$ and  $\S^>_u \cap \bd(T) = \emptyset$,
  iff $K=\B$, $r(K)=1$ and $c=0$.
\end{lem}

\begin{proof}
  For the \enquote{if}-direction, we easily see that if $K=\B$ then choosing $l=2$ and $u^2=-u^1$ and any
  $u$ orthogonal to $u^1$, we obtain
  $T \cap\S = \S \supset S^\ge_u$ and $\S^>_u \cap \bd(T) = \emptyset$.

  For proving the \enquote{only if}-direction let us assume $r(K) < 1$. Then, however
  $u^1,\dots,u^l$ and $u$ are chosen, they must satisfy $0 \in \conv\{u^1,\dots,u^l\}$ and thus there exists $j \in [l]$,
  \st~$u^Tu^j \ge 0$ and $u^j\in\S^\ge_u$.

  Since $c+r(K)\B\subset\B$, it holds $\norm[c]+r(K) \leq 1$ and therefore
  $c^Tu^j + r(K) \leq \norm[c]\norm[u^j] + r(K) = \norm[c] + r(K) \leq 1$, which, as $u^j \in \S$, implies
  $(u^j-c)u^j \geq r(K)$ and \enquote{$=$} holds, iff $c=(1-r(K))u^j$,
  which means $u^j = c + r(K)u^j$.

  Now, in case of $(u^j-c)u^j > r(K)$, it follows $u^j \notin c+\{x\in\R^n : x^Tu^j \leq r(K)\} \supset
  T \supset S^\ge_u$, contradicting $u^j \in S^\ge_u$.

  On the other hand, if $(u^j-c)u^j = r(K)$, it holds $u^j=c+r(K)u^j \in c+\{x\in K : x^Tu^j = r(K)\} \subset \bd(T)$.
  However, since $\S^>_u \cap \bd(T) = \emptyset$, it follows
   $u^j \in \S^\ge_u\setminus\S^>_u = \S \cap \{x:u^Tx=0\}$ and therefore $u^Tu^j=0$.
  Now, since $0\in\conv\{u^1,\dots,u^l\}$,
  there exists $k \in [l] \setminus \{j\}$, \st~$u^Tu^k \ge 0$. But, since $c+r(K)u^k \in \S$
  would mean that there exist two different points of $c+r(K)\B$ in $\S$, contradicting $r(K)<1$,
    we must have $c+r(K)u^k \notin \S$. Hence
  $(u^k-c)u^k > r(K)$ as shown above with $j$ instead of $k$,
  contradicting $u^k \in S^\ge_u$.
\end{proof}

\begin{lem} \label{lem:max width}
  Let $K\in\CK^2$ and $c\in\R^2$, \st~$c+r(K)\B \subset K \subset \B$, as well as
  $p^1,p^2,p^3$
  (possibly with $p^2=p^3$), $u^1,u^2,u^3$ (possibly with $u^2=u^3$), $T$, and $T'$ as in Corollary
    \ref{cor:separate} for the case $n=2$.
    The common supporting lines of $K$ and $c+r(K)\B$ with outer normals $u^1,u^2,u^3$ are denoted
    by $L_1,L_2,L_3$, respectively, the halfspaces induced by these lines containing $K$
    by $L_1^-,L_2^-,L_3^-$ (thus $T':= \conv\{p^1,p^2,p^3\}$ and
    $T := L_1^- \cap L_2^-\cap L_3^-$). Finally, define $C := T \cap\B$,
    and $S_i := L_i \cap C$, $i=1,2,3$.
  Then
  \begin{enumerate}[a)]
  \item the line segments of $T'$ separate the line segments $S_i$ of $T$ from 0 within $\B$.
  \item the length of each line segment $S_i$, $i=1,2,3$, is at most $D(K)$.
  \item the diameter of $C$ is taken between two points on different arcs of $C \cap \S$ or $D(C)=2$.
  \item there exist $q^1,q^2 \in C \cap \S$ \st~$\norm[q^1-q^2] = D(K)$
    and the segment $[q^1,q^2]$ separates one of the segments $S_i$, $i=1,2,3$, from the other
    two segments and the origin $0$ (see Figure \ref{fig:Lem5_6} as an example).
  \end{enumerate}
\end{lem}

\begin{figure}
  \centering
  \includegraphics[trim = 1cm 3cm 1cm 3cm, width = 8cm]{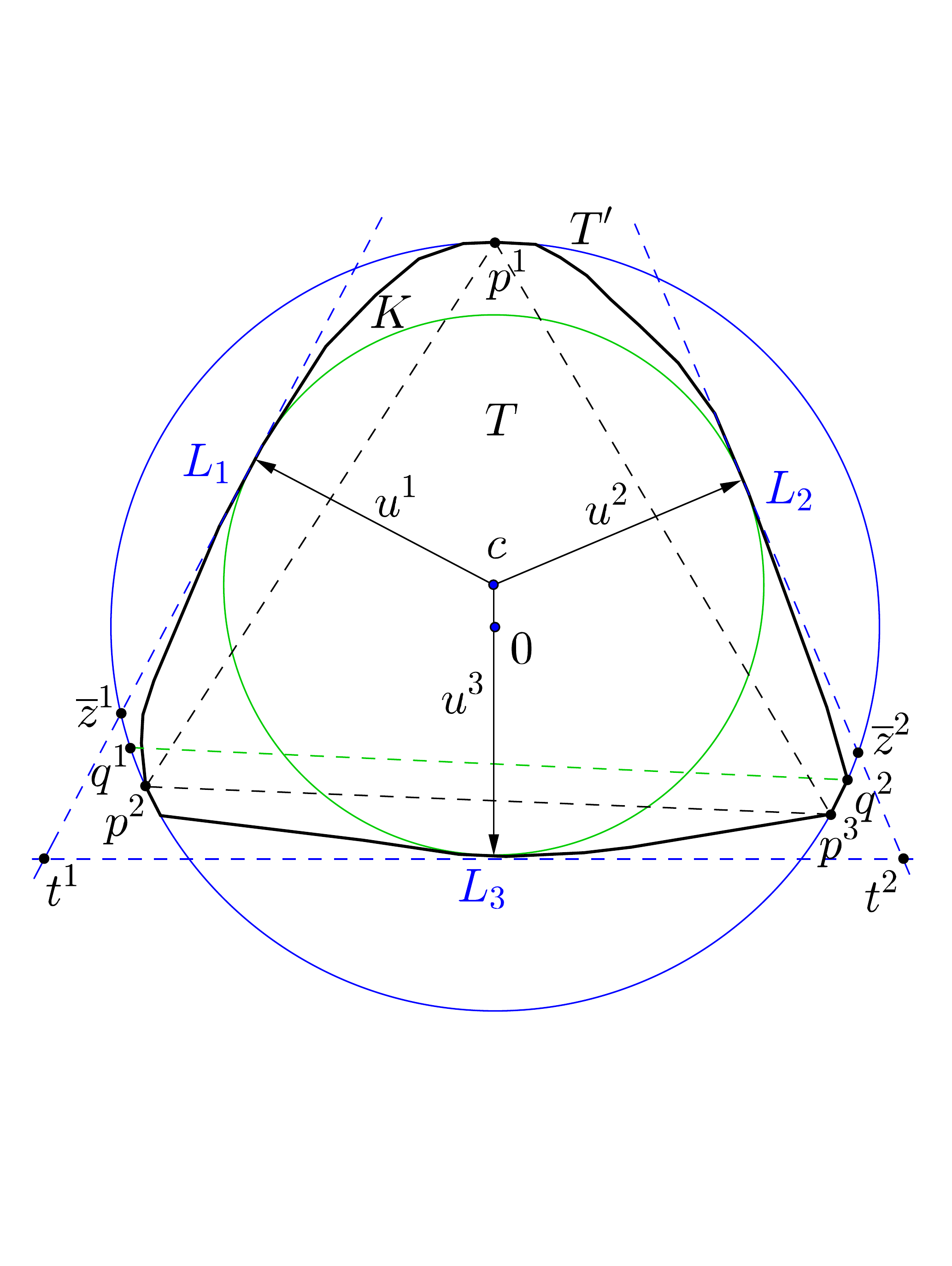}
  \caption{A convex set $K$ and all elements of Lemma \ref{lem:max width}. Observe that
    $q^1\notin K$.}
  \label{fig:Lem5_6}
\end{figure}

\begin{proof}
  \begin{enumerate}[a)]
  \item This is a direct interpretation of Part (b) of Corollary \ref{cor:separate} in $\R^2$ (but only there).

  \item If the length of $S_i$ would be greater than $D(K)$, the same would be true for the segment of $T'$
    separating $S_i$ from 0, a contradiction as $T' \subset K$.

   \item By Proposition \ref{prop:exposed}, there exist extreme points $z^1,z^2$ of $C$, \st~$\norm[z^1-z^2] = D(C)$.
    Using Part (a) of Corollary \ref{cor:separate}, we distinguish the cases where no or one vertex of
    $T$ belongs to $\inte(\B)$.

    In the first case, it holds $z^1,z^2 \in\ext(C) = C\cap\S$
    We show that if $K \neq \B$, then $z^1$ and $z^2$ do
    not belong to the same arc of $C\cap\S$.
    If they do, we denote by $\bar{z}^i$, $i=1,2$ the one of the two points in $L_i \cap \S$,
    which is in the same arc of $C\cap\S$ than $z^i$ (see again Figure \ref{fig:Lem5_6}).
    (defining that, if a line intersects $\S$ in a single point, then
    it separates two different arcs).

    Using Lemma \ref{lem:semisphere}, we know that if $K \neq \B$ the arc containing $z^1, z^2$
    is at most an open semisphere.
    Hence $D(C) = \norm[z^1-z^2] \le \norm[\bar{z}^1 - \bar{z}^2] <2$ and therefore $\bar{z}^i=z^i$,
    $i=1,2$.

    We first consider the case
    that $L_i \cap \inte(\B) \neq \emptyset$, $i=1,2$.
    Since $0 \in \conv\{u^1,u^2,u^3\}$ the lines $L_i$, $i=1,2$
    are parallel or intersect in a vertex of $T$ on the same side of 0 than the segment $[z^1,z^2]$.

    However, in any of the two cases,
    the distances between $z^1$ and any point in $L_2\cap\inte(\B)$ or the distance between $z^2$ and any point
    in $L_1\cap\inte(\B)$ is strictly bigger than $\norm[z^1-z^2]=D(C)$, a contradiction.

    Now turn to the case that
    $L_i \cap \inte(\B) = \emptyset$ for at least one $i \in \{1,2\}$, \Wlog for $i=1$,
    which means $\{z^1\} = L_1 \cap \S$ thus $L_1$ supports $\B$ in $z^1$.
    However, since $L_1$ supports $c+r(K)\B$ by definition, we obtain that it must support $c+r(K)\B$ in $z^1$.
    Using the fact that the arc containing $z^1,z^2$ is at most an open semisphere, we have
    $z^2 \neq -z^1$ and therefore $D(C) \ge D(\conv(\{z^2\},c+r(K)\B) > \norm[z^1-z^2] = D(C)$,
    again a contradiction.

    Finally consider the case that one vertex of $T$ belongs to $\inte(\B)$.
    Then $C \cap \S$ contains at most two arcs.
    Applying Part (a) of Proposition \ref{prop:CONT} for $C$, there exist $p^1,p^2,p^3$ in this two arcs,
    \st~$0 \in \conv\{p^1,p^2,p^3\}$. However, as two of the $p^i$ have to be on the same arc, the negative of the third
    has to be on that arc, too, proving $D(C)=2$ for that case.

  \item In case of $K=\B$ the claim is trivially true. Hence we may assume $K \neq \B$.

    Using Part (a) of Corollary \ref{cor:separate}, we distinguish again the cases with no or one vertex of
    $T$ belonging to $\inte(\B)$.

    In the first case, it was shown in Part (c) that any pair of diametral points
    $z^1$ and $z^2$ of $C$ lie in different arcs of $C\cap\S$.
    But this means that $[z^1,z^2]$ separates one of the segments $S_i$ from the other
      two segments and the origin $0$, say $S_3$ (\cf~Figure \ref{fig:Lem5_6}).
    From Part (b) we know that the length of $S_3$ is at most $D(K) \le D(C) = \norm[z^1-z^2]$.
    Hence there exist $q^1$ and $q^2$ in the same arcs as $z^1$ and $z^2$, respectively,
    \st~$\norm[q^1-q^2]=D(K)$ and $[q^1,q^2]$ still separates $S_3$
    in the same way as $[z^1,z^2]$ does.

    If a vertex of $T$ belongs to $\inte(\B)$, it follows from
        Part (c) that $D(C)=2$, which means $z^2=-z^1$.
    Thus $[z^1,z^2]$ separates the two segments intersecting in $\inte(\B)$ from the third.
    However, again because of Part (b) there must exist
    $q^1$ and $q^2$ with $\norm[q^1-q^2]= D(K)$ separating this third segment from the other two and 0.
  \end{enumerate}
\end{proof}

\begin{lem} \label{lem:max width 2}
  Consider the same setting and notation as in Lemma \ref{lem:max width}.
  In the following we assume that the single separated segment in Part (d) of Lemma \ref{lem:max width} is $S_3$
  and \Wlog that $S_3$ is horizontal below $0$ as well as separated by $[q^1,q^2]$ from $S_1$, $S_2$, and $0$.
  Moreover, we denote the point in $C$ farthest from $L_3$ by $y$, the intersection points of $L_i$, $i=1,2$
  with $L_3$ by $t^i$, $i=1,2$, respectively,
  and assume that $t^1_1 \le 0 \le t^2_1$ (which is possible when $S_3$ is horizontal and means
    that $L_1$ bounds $S_3$ on the left while $L_2$ bounds $S_3$ on the right, see~Figure \ref{fig:Lem5_6}).
  \begin{enumerate}[a)]
  \item The first coordinate of the intersection points of $L_1$ and $\S$ is bounded from above
      by $\nicefrac{D(K)}{2}$ while
      the first coordinate of the intersection points of $L_2$ and $\S$ is bounded from below
      by $\nicefrac{D(K)}{2}$.
  \item It holds $|y_1| \le \nicefrac{D(K)}{2}$.
  \item One can modify the choice of $q^1$ and $q^2$ satisfying Part (d) in  Lemma \ref{lem:max width},
    \st~the interior angles of $\conv\{y,q^1,q^2\}$ in $q^1$ and $q^2$ are at most
    $\nicefrac \pi 2$.
  \end{enumerate}
\end{lem}

\begin{proof}
  \begin{enumerate}[a)]
  \item It suffice to show the upper bound in case of $L_1$.
    Since $S_3$ is the separated segment, it follows $t^1,t^2 \not\in \inte(\B)$ and since $S_3$ is horizontal,
    (a) is obviously true for
    $\bar z^1$. Now we denote the further one by $x^1$ and assume $x^1 \ge 0$ as otherwise there is nothing to show.
    Since $S_3$ is horizontal, we have $t^1_1 \leq 0$ and $t^1_2\leq x^1_2$, which means
    $L_1$ has a positive slope.
    Now, assuming $x^1_1>\nicefrac{D(K)}{2}$ implies together with $\bar z^1_1 \le q_1^1 = - \nicefrac{D(K)}2$
    that the length of $S_1$ is strictly greater than $D(K)$,
    which contradicts Part (c) of Lemma \ref{lem:max width}.

  \item Again, it suffices to show $y_1 \le \nicefrac{D(K)}{2}$, because of symmetry in the argument.
    If $L_1$ and $L_2$ intersect within $\inte(\B)$,
    they must intersect in $y$.
    Hence $y_1 \le x^1_1 \le \nicefrac{D(K)}{2}$
    using the notation as in Part (a).
    Otherwise $y$ lies on the arc of $C \cap \S$ between $x^1$ and $x^2$, its corresponding point
      in $L_2 \cap \S$. However,
    with $e^2$ denoting the second unit vector,
    that would mean $y \in \{x^1,x^2,e^2\}$, which again proves the claim because of Part (a).

  \item We suppose \Wlog that $q^i$, $i=1,2$, belong to the arc of $\S$ induced by $S_i$
    and $S_3$, $i=1,2$, respectively. First, it follows from Thales' theorem
    that the region $R \subset\B$ on the same side as 0 of
    $[q^1,q^2]$, for which one of the angles would be bigger than $\nicefrac \pi 2$, is
    the union of the caps of $\B$ induced by $\aff\{q^1,-q^2\}$ and $\aff\{-q^1,q^2\}$,
    without $\aff\{q^1,-q^2\}$ and $\aff\{-q^1,q^2\}$ themselves.

    As we have seen in Part (b), using the notation there
    and denoting the intersection of $L_1$ and $L_2$ by $t^3$, it holds $y \in \{x^1,x^2,e^2,t^3\}$.
    In any of the four cases $[q^1,q^2]$ separates $y$ from $S_3$.

    Now, we describe the choice of $q^1$ and $q^2$ satisfying Part (c) for each of the possible $y$'s:
    Since $q^1_1 < 0$ and $q^2_1 > 0$, we obviously have $y \notin R$, if $y=e^2$.
    To obtain $y=x^i$, $i=1,2$, it must hold that $e^2$ does not belong to $C$, which means
    either $x_1^1 > 0$ and $y=x^1$ or $x_2^1 < 0$ and $y=x^2$. Hence we may assume \Wlog that $y=x^1$ or
    $y=t^3$ and that, if $y \in R$, then it is contained in the cap induced by
    $\aff\{-q^1,q^2\}$. Using Part (c) of Lemma \ref{lem:max width}, we know that the length of
    $S_1$ is at most $D(K)=\norm[q^1-q^2]$.

    Since $y$ is contained in the cap of $\B$ between $q^2$ and $-q^1$
    this is only possible if $S_1$ cuts through one of the segments $[q^1,q^2]$ or $[-q^2,-q^1]$
    (otherwise $S_1$ would be longer than $\norm[q^1-q^2]=D(K)$, a contradiction).
    However, the first case would contradict that $[q^1,q^2]$ separates $S_1$ from $S_3$.
    Hence $S_1$ must intersect $[-q^2,-q^1]$.
    Using $S_3$ being horizontal and therfore $S_1$ being ascending (and $S_2$ descending) as well as
    the fact that $-q^1$ is above $y$,
    $[q^1,q^2]$ must be ascending
    even with a bigger slope than $S_1$, since othwerwise they could not intersect.
        But then we may move $q^1, q^2\in C\cap\S$ inside the same arcs and keeping their distance, until $[q^1,q^2]$
    becomes parallel to $S_1$, but stays ascending, and therefore not annihilating the separation of $S_3$.
    Rebuilding $R$ from the new vectors $q^1,q^2$, we obtain $y \notin R$.
  \end{enumerate}
\end{proof}

Before we state the follwing Lemma, remember that we know from $(ub_2)$ and $(ub_3)$
in Subsection \ref{ss:facets}, that for every diameter $D \in [\sqrt 3,2]$
and inradius $r \in  [r(\Iso[2\arccos(\nicefrac{D}{2})]),r(\CSB[\arcsin(\nicefrac{D}{2})])]$
there exist triangles $\Triangle$ -- in case of
$r \leq r(\Iso[\arcsin(\nicefrac{D}{2})])$ --
or sailing boats $\SBoat[r][\arcsin(\nicefrac{D}{2})]$ -- in case of $r \geq r(\Iso[\arcsin(\nicefrac{D}{2})])$.

\begin{lem}\label{lem:triangles and sail boats}
Let  $D \in [\sqrt 3,2]$ and
$r(\Iso[2\arccos(\nicefrac{D}{2})]) \le r \le r(\CSB[\arcsin(\nicefrac{D}{2})])$.
Then for all $K\in\CK^2$, \st~$D(K)=D$ and $r(K)=r$ there exists
\begin{enumerate}[a)]
\item a triangle $\Triangle$, \st~$w(K) \le w(\Triangle)$, if $r \le r(\Iso[\arcsin(\nicefrac{D}{2})])$,
\item a sailing boat $\SBoat[r][\arcsin(\nicefrac{D}{2})]$, \st~$w(K) \le w(\SBoat[r][\arcsin(\nicefrac{D}{2})])$,  if $r \ge r(\Iso[\arcsin(\nicefrac{D}{2})])$.
\end{enumerate}
\end{lem}

\begin{proof}
Let $c\in\R^2$, \st~$c+r(K)\B$ and $\B$ are the in- and circumball of $K$, respectively.
Using the notation as given in Lemma \ref{lem:max width},
remember that $R(C)=R(K)$ and $r(C)=r(K)$, whereas
the monotonicity of the radii with respect to set inclusion implies $D(C) \geq D(K)$ and
$w(C)\geq w(K)$.

The idea of the proof is to transform $C$ in several steps into some
triangle or sailing boat $\bar C$
satisfying $R(\bar C)=R(K)$, $r(\bar C)=r(K)$, $D(\bar C) = D(K)$, and $w(\bar C) \ge w(K)$.

More precisely, denoting the breadth of $C$ in direction of $u^3$ by $b_{u^3}(C)$, we know
from the definitions of the width and the point $y$ in Part (b) of Lemma \ref{lem:max width 2}, which is the
farthest point in $C$ from $L_3$,
that $w(C) \le b_{u^3}(C)=\dist{y}{L_3}$.

Now, in every step of the transformation of $C$, we will increase the breadth in direction of $u^3$, but arriving
in $\bar C$ it even holds $w(\bar C) = b_{u^3}(\bar C)$ (as we have seen when defining the triangle and sailing boat families).

\begin{figure}
  \centering
  \begin{subfigure}[b]{0.45\textwidth}
    \includegraphics[trim = 1cm 3.5cm 1cm 2cm, width =\textwidth]{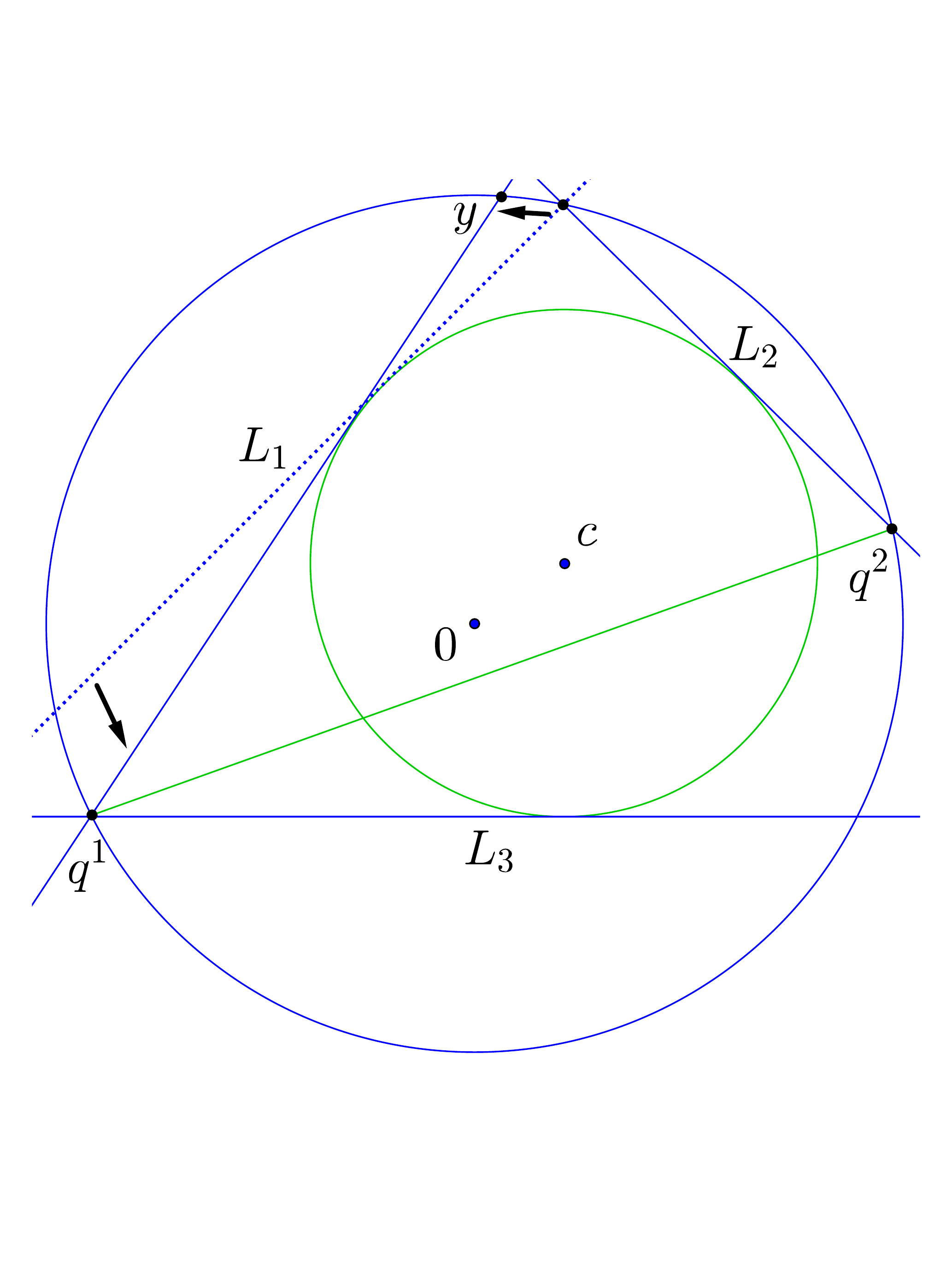}
    \caption{In (i) the lines $L_1, L_2$ may be rotated around $c+r\B$.\\ \hfill}
    \label{fig:Lem5.8Fig1}
  \end{subfigure} \hfill 
  \begin{subfigure}[b]{0.45\textwidth}
    \includegraphics[trim = 1cm 3.5cm 1cm 2cm, width = \textwidth]{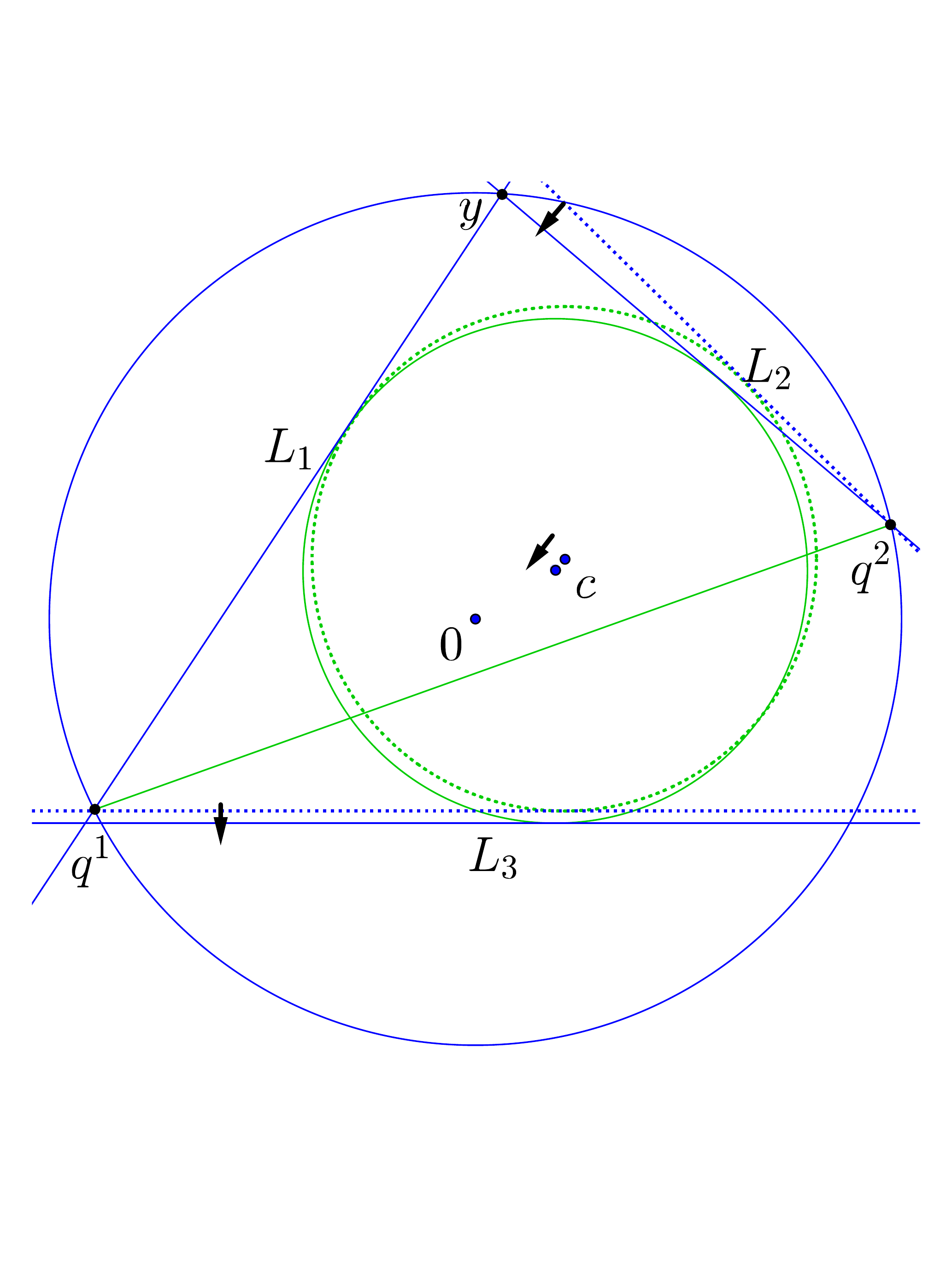}
    \caption{In (ii) we may move $c+r\B$ downwards while $L_1$ (and/or $L_2$)
        may be rotated around $q^1$ (and/or $q^2$).}
    \label{fig:Lem5.8Fig2}
  \end{subfigure}
  \caption{Examples for (i) and (ii) in Lemma \ref{lem:triangles and sail boats}.
    Here and in Figures 21 to 23 the start and the end of a movement are indicated by dotted and,
    respectively, full lines.}
  \label{fig:Lem5.8Step12}
\end{figure}

\begin{enumerate}[(i)]
\item Rotate the lines $L_1$ and $L_2$, \st~they keep supporting $c+r(K)\B$ and contain $q^1$ and $q^2$,
  respectively, thus also keeping the separation of  $S_1,S_2$ from $S_3$ by $[q^1,q^2]$.
  In the degenerate case of only two supporting parallel hyperplanes to $c+r(K)\B$ (which means
    by the choices in the proof of Lemma \ref{lem:max width} that $u^1=u^2$),
    we substitute $L_1$ by two lines $L_1$ and $L_2$ supporting $c+r(K)\B$ and containing $q^1$ and $q^2$,
    respectively, \st~$0 \in \inte(\conv\{u^1,u^2,u^3\})$ and arrive in the same situation then in the
    non-degenerate case.

  Thus, the $y$ we have before the change still belongs to $C$ afterwards and
  therefore the new $y$ (the point at maximum distance from $L_3$ within the new $C$)
  is not closer to $L_3$ than before. Applying Parts (b) and (c) of Lemma
  \ref{lem:max width 2} for the new $C$, we still have that $|y_1|\leq\nicefrac{D(K)}{2}$
  and that $\conv\{y,q^1,q^2\}$ has interior angles in $q^1$ and $q^2$
  at most $\nicefrac \pi 2$.

\item This step is only needed, if $L_1\cap L_2 \notin \B$, which means that $y \in \{x^1,x^2,e^2\}$.
  First, as long as $y \notin L_2$ we translate $c+r(K)\B$
  downwards, parallel to $L_1$, rotate $L_2$ around the point $q^2$
  and move $L_3$ parallel to its prior position, \st~$L_2$ and $L_3$
  keep supporting $c+r(K)\B$. Afterwards, as long as $y \notin L_1$
  we translate $c+r(K)\B$ downwards, parallel to $L_2$, rotate $L_1$
  around the point $q^1$ and move $L_3$ again parallel to its prior position,
  \st~$L_1$ and $L_3$  keep supporting $c+r(K)\B$.
  In the end $y$ is in $L_1 \cap L_2 \cap \S$
  and since $L_3$ moves always vertically downwards, but $y$ stays equal,
  the distance $\dist{y}{L_3}$ does not decrease.

\item Since the inner angles of $\conv\{y,q^1,q^2\}$ in $q^1$ and $q^2$ are at most
  $\nicefrac{\pi}{2}$, the distance of the intersection points
  $t^i$, $i=1,2$, of $L_3$ with $L_i$, $i=1,2$,
  is at least $\norm[q^1-q^2] = D(K)$.

  Hence there exist $\bar q^i \in L_i$, $i=1,2$, \st~$[\bar q^1, \bar q^2]$ is
  parallel to $L_3$ and $\norm[\bar q^1 - \bar q^2] = D(K)$. We
  move $T$ until $\bar q^i = q^i$, $i=1,2$,
  (differently to $\bar q^i$, $i=1,2$, the $q^i$'s are not fixed to the construction of $T$
    and therefore are not affected by the rotation)
  and afterwards
  rotate everything to have $L_3$ again horizontal (see Figure \ref{fig:Lem5.8Fig3}).

  From (ii) it follows $y \in L_1\cap L_2 \subset \B$ and since $T$ is only moved in (iii)
  the angle in $y$ of $T$ does not change. Hence Proposition \ref{lem:angle} with $q^3=y$
  implies that after the movement the vertex $y$ is still in $\B$, and
  moreover, if $y\in\S$ before the movement, it will be in $\S$ after, too.
  Since we just do a solid motion on $T$, the distance $\dist{y}{L_3}$ keeps constant
  in (iii).

\item[(iv)] If $y \in \S$, we move $\{y\} = L_1 \cap L_2$ around $\S$ towards $e^2$ and
  $L_1$ and $L_2$ with it.
  The inball is moved, \st~it remains tangent to both $L_1$ and $L_2$ and the
  line $L_3$ parallel to its prior position to keep tangent
  to the inball.
  We stop, when $y=e^2$ (see Figure \ref{fig:Lem5.8Fig4}) or $L_3$ contains the segment $[q^1,q^2]$
  (whichever comes first -- the first meaning that we arrive in a sailing boat, the latter
    that we arrive at a triangle).

  Before and after the transformation the inradius and the angles in the points $y$
  of the two triangles coincide (see Proposition \ref{lem:angle}),
  while the line passing through the incenter $c$
  and $y$ becomes closer to be perpendicular to $L_3$.
  Hence the distance $\dist{y}{L_3}$ does not decrease under this movement.

  If $y=e^2$, then $C=\SBoat[r(K)][\gamma(D(K))]$, otherwise, if $L_3$ contains the segment $[q^1,q^2]$,
  $C =\Triangle[r(K)][D(K)]$. In both cases
  $r(C)=r(K), D(C)=D(K)$, $R(C)=R(K)$,
  and $w(C) = \dist{y}{L_3} \ge b_{u^3}(K) \ge w(K)$ holds.

\item[(v)] If $y \in \inte(\B)$,
  we rotate the lines $L_1, L_2$ around $q^1,q^2$, repectively,
  \st~$y = L_1 \cap L_2$ moves along $\aff\{y,c\}$ away from $c$.
  The inball moves, \st~it remains tangent to $L_1$ and $L_2$, and
  $L_3$ is shifted upwards, parallel to its original position to
  remain tangent to the inball.

  The change finishes when $y \in \S$ or the line
  $L_3$ contains the segment $[q^1,q^2]$.

  Before and after the movement the triangle $T$ has the same inradius and
  $\aff\{y,c\}$ has the same angle with respect to $L_3$, but the angle in
  $y$ decreases. Hence the distance $\dist{y}{L_3}$ does
  not decrease.

  If we arrive in $y \in \S$,
  we are in a situation to apply (iv) again.
  If $L_3$ contains the segment $[q^1,q^2]$, then we may roll the inball along $L_3$ and
  rotate $L_i$, $i=1,2$, \st~they keep supporting the inball, until $y \in \S$.
  Hence the inball of $\conv\{y,q^1,q^2\}$ stays equal and
  it can easily be checked that the width of the triangle does not decrease
  under this change. In fact, we again arrive in the situation $C=\Triangle[r(K)][D(K)]$ as
  after (iv), when $y \neq e^2$.
\end{enumerate}
\end{proof}

\begin{figure}
  \centering
  \begin{subfigure}[b]{0.45\textwidth}
    \includegraphics[trim = 1cm 3.5cm 1cm 2cm, width =\textwidth]{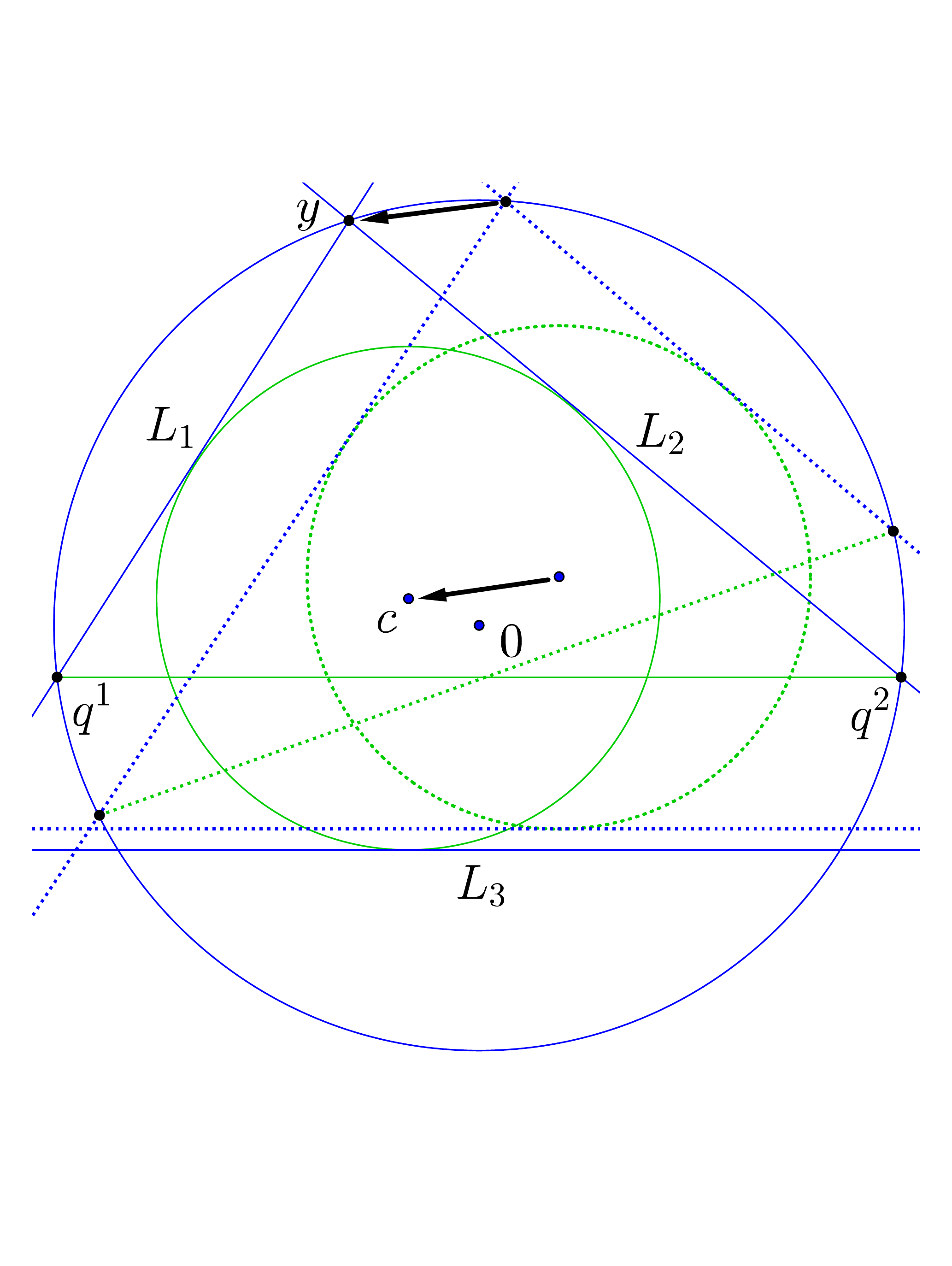}
    \caption{In (iii) the set rotates until $[q^1,q^2]$ becomes parallel to $L_3$.}
    \label{fig:Lem5.8Fig3}
  \end{subfigure} \hfill 
  \begin{subfigure}[b]{0.45\textwidth}
    \includegraphics[trim = 1cm 3.5cm 1cm 2cm, width = \textwidth]{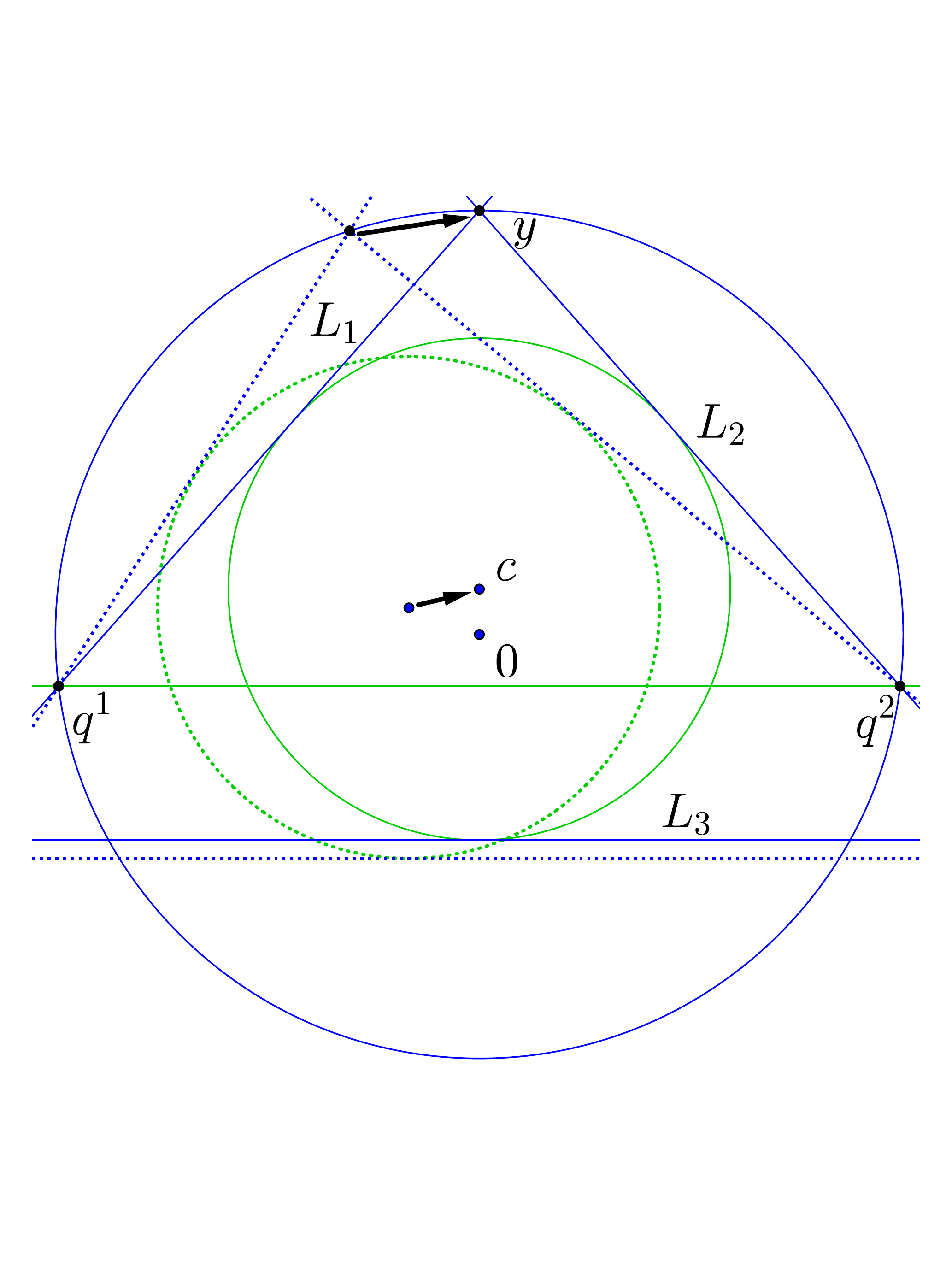}
    \caption{In (iv) (if $y\in\S$) $y$ moves inside $\S$ and may become $e^2$
      whereas $C=\SBoat$.}
    \label{fig:Lem5.8Fig4}
  \end{subfigure}
  \caption{Examples of (iii) and (iv) from Lemma \ref{lem:triangles and sail boats}.}
  \label{fig:Lem5.8Step34}
\end{figure}

\begin{proof}[Proof of Theorem \ref{thm:2}]
  The part of Theorem \ref{thm:2} that all sailing boats fulfill equality for \eqref{eq:sailing boats}
  directly follows from Lemma \ref{lem:triangles and sail boats}. Thus
  it only remains to show the general validity of the inequality \eqref{eq:sailing boats}.
  
  Since there exist isosceles triangles $\Iso$ and concentric sailing boats $\CSB$
  of the same diameter and circumradius as a given $K$ for an appropriate choice of
  $\gamma \in[\nicefrac{\pi}{3},\nicefrac{\pi}{2}]$,
  we only have to distinguish the cases
  \[\text{(i)} \quad  r(K) \leq r(\Iso), \qquad \text{(ii)} \quad r(\Iso) \leq r(K) \leq r(\CSB),
  \qquad \text{(iii)} \quad  r(K) \ge r(\CSB).
  \]
  
  
  Again we abbreviate $r=r(K), w=w(K), D=D(K)$, and $R=R(K)=1$.
  
  In case of (ii), $K$ fulfills the conditions of Lemma \ref{lem:triangles and sail boats} and
  we obtain $w\leq w(\SBoat[r][D])$, which suffices as mentioned above.
  For the other two cases we \emph{extend} the construction of the general sailing boats from
  Subsection \ref{sailing-boat-facet}:
  
  For any pair $r,D$ obtained from $K$, let $\Iso=\conv\{p^1,p^2,p^3\}$,
  $\gamma \in [\nicefrac{\pi}{3},\nicefrac{\pi}{2}]$ be the isosceles triangle with circumball
  $\B$ and diameter $D=\norm[p^1-p^2]=$ as well as $I_K:=\nicefrac{r}{r(\Iso)}(\Iso-p^3)+p^3$
  the rescaled copy with inradius $r$, keeping the vertex $p^3$.
  
  By construction, $I_K$ belongs to the general sailing boats, $D(I_K) = \nicefrac{r}{r(\Iso)} \, D$,
  and $R(I_K) = \nicefrac{r}{r(\Iso)} \, R$. Hence it fulfills \eqref{eq:sailing boats} with equality.
  However, since
  \[r(I_K) \left( 1 + \frac{2\sqrt{2}R(I_K)}{D(I_K)}
    \sqrt{1 + \sqrt{1-\left(\frac{D(I_K)}{2R(I_K)}\right)^2}}\right)
  = r \left( 1 + \frac{2\sqrt{2}R}{D} \sqrt{1 + \sqrt{1-\left(\frac{D}{2R}\right)^2}}\right)\]
  it suffices to show that $w \le w(I_K)$.
  
  Now, we first consider case (i).
  
  Using Lemma \ref{lem:triangles and sail boats} and the notation used there, we know there exists a
  triangle $\Triangle[r][D] = \conv\{q^1,q^2,y\}$ in the triangle face, \st~$\norm[q^1-q^2]=D$
  and $w \leq w(\Triangle[r][D]) = \dist{y}{[q^1,q^2]}$.
  Hence we just need to prove $w(\Triangle[r][D]) \leq w(I_K)$.

Similar to (iv) of Lemma \ref{lem:triangles and sail boats}, we now transform $\Triangle[r][D]$
by moving $y$ within $\S$ until $y=p^3$, ignoring the stopping condition \enquote{when $L_3$ contains
  $[q^1,q^2]$}.
Because of ignoring the stopping condition, the inball will not touch $[q^1,q^2]$ anymore, but
a line $L$ parallel to $[q^1,q^2]$, which means that we arrived at a triangle congruente with $\Iso$
and inradius $r$, which is $I_K$.
Thus $\dist{p^3}{L} = w(I_K)$ and we may argue as in (iv) of Lemma \ref{lem:triangles and sail boats}
that $w(\Triangle[r][D]) \leq  w(I_K)$, which shows the assertion.

Finally, assume we are in case of (iii).
We know from Subsection \ref{ub1-facet} that the outer parallel bodies $K'$ of
a concentric saling boat or a Reuleaux blossom satisfy
$r(K')=r$, $D(K')=D$, $R(K')=R$, and $w \le w(K') = r(K') + R(K')$.
Hence we just need to show that $w(K') \leq w(I_K)$ again.

Now, consider the concentric sailing boat $\CSB$. It shares $p^3$ and its inside angle $\gamma$ with
$I_K$ and has a smaller inradius. Thus it follows from the concentricity of the in- and circumradius
of $\CSB$ that $c_2<0$ holds for the incenter $c$ of $I_K$
has a negative second component $\eta$. Hence
$w(I_K) = r + R + |c_2| \ge r(K') + R(K') =w(K')$
which finishes the proof.
\end{proof}

\medskip

\begin{proof}[Proof of Theorem \ref{thm:4}]
In case of $r(K) \le r(\Iso[\arcsin(\nicefrac{D(K)}{2R(K)})])$ Part (a)
  of Lemma \ref{lem:triangles and sail boats}
implies $w(K) \le w(\Triangle)$, proving the validity of \eqref{eq:triangles} in that case.

Thus we may assume \Wlog that $r(K) \ge r(\Iso[\arcsin(\nicefrac{D(K)}{2R(K)})])$.
Observe two facts: first, if $r(K) = r(\Iso[\arcsin(\nicefrac{D(K)}{2R(K)})])$, then the two right hand sides of
\eqref{eq:sailing boats} and \eqref{eq:triangles} coincide and equal
$w(\Iso[\arcsin(\nicefrac{D(K)}{2R(K)})])$. Omitting again the argument $K$, we obtain
\begin{equation} \label{eq:ub2=ub3}
\frac{w}{r}=1+\frac{2\sqrt{2}R}{D}\sqrt{1+\sqrt{1-\left(\frac{D}{2R}\right)^2}}=
\frac{2}{D}\left(D+\frac{2rR}{D}\left(1+\sqrt{1-\left(\frac{D}{2R}\right)^2}\right)\right)
\end{equation}
in that case.
The second fact to be observed is that in \eqref{eq:ub2=ub3} the middle expression does not depend on $r$,
while the right hand part is increasing in $r$. Hence knowing the general validity of \eqref{eq:sailing boats},
we may conclude
\[
\frac{w}{r} \le 1+\frac{2\sqrt{2}R}{D}\sqrt{1+\sqrt{1-\left(\frac{D}{2R}\right)^2}} \le
\frac{2}{D}\left(D+\frac{2rR}{D}\left(1+\sqrt{1-\left(\frac{D}{2R}\right)^2}\right)\right).
\]
\end{proof}

Now, we turn to the open part of the lower boundary and
start with a technical corollary needed in order to prove Theorem \ref{thm:3}.

\begin{cor}\label{cor:min width}
  Let $K \in \CK^n$ and $c \in\R^n$, \st~$c+r(K)\B$ and $\B$ are the in- and circumball of $K$,
  respectively, $p^1,\dots,p^k \in K \cap \S$ be the points given by Part (a) of Proposition \ref{prop:CONT},
  $T':= \conv\{p^1,\dots,p^k\}$, and $C:=\conv(T', c+r(K)\B)$.
  Then $D(C) = \max\{D(T'),\norm[p^i-c]+r(C),i \in [k]\}$.
\end{cor}

\begin{proof}
  Since the statement is obviously true if $K=\B$, we may assume \Wlog that $K \neq \B$
  and therefore $C \neq \B$. This means the diameter of $C$ is
  bigger than $2r(C)$, the distance of two antipodal points of the inball,
  but due to Proposition \ref{prop:exposed} attained between two extreme points.

  Thus if it is not attained between a pair of the vertices $p^1,\dots,p^k$, it must be between one of
  them and its antipodal on the insphere.
\end{proof}

\begin{rem}\label{rem:min width}
  Let $K \in \CK^2$, $c \in\R^2$, $p^1,\dots,p^k \in K \cap \S$, $T'$, and $C$
  be given as in Corollary \ref{cor:min width}.
  Denoting by $L_1,L_2$ a pair of parallel supporting lines of $C$, \st~$w(C)=d(L_1,L_2)$
  we may assume \Wlog due to Proposition \ref{prop:exposed} that $L_1$ has at least two contact
  points with $C$ and
  (by renaming and defining $p^3=p^2$ if neccessary) that $p^1$ is situated in one of
  the arcs in $\S$ between $L_1$ and $L_2$,
  while $p^2$ and $p^3$ belong to the other with $p^2$ closer to $L_1$ and $p^3$ closer to $L_2$.
  With this assumptions one of the following cases holds:
  \begin{enumerate}[(i)]
  \item $L_1$ contains $p^2$ but not $p^1$ and supports $c+r(K)\B$, whereas $L_2$ supports $c+r(K)\B$.
  \item $L_1$ contains $p^2$ but not $p^1$ and supports $c+r(K)\B$, whereas $L_2$ contains only $p^3$.
  \item $L_1$ contains $p^1$ but not $p^2$ and supports $c+r(K)\B$, whereas $L_2$ contains only $p^3$.
  \item $L_1$ contains $[p^1,p^2]$, whereas $L_2$ contains $p^3$ or supports $c+r(K)\B$.
  \end{enumerate}
  Due to Proposition \ref{prop:exposed} one of the sets $L_i\cap C$, $i=1,2$,
  say $L_1\cap C$ contains a smooth boundary point of $C$.
  Hence $L_1 \cap C$ is either a segment containing at least one of the points $p^1,p^2$,
  which means we are in Case (ii),(iii), or (iv), or $L_1$ supports the inball in a unique point
  (see Figure \ref{fig:Step1lb3} for an example of Case (ii)). However, in case $L_2 \cap C$ is a segment,
  we may interchange the roles of $L_1$ and $L_2$ arriving again in Case (ii), (iii), or (iv), or if $L_2$
  also supports $C$ only in a single boundary point of the inball, we may rotate
  $L_1$ and $L_2$ around it, \st~we may assume Case (i).
\end{rem}

The following lemma proves Theorem \ref{thm:3} apart from the general validity of the inequality.

\begin{lem}\label{lem:3 Bent Isosceles}
  Let $K\in\CK^2$ be \st~there exists a bent pentagon $\BPen$ from the facet
  $(lb_3)$ with the same inradius, circumradius, and diameter as $K$.
  Then $w(K) \geq w(\BPen)$.
\end{lem}

\begin{proof}
Using the same notation as in Corollary \ref{cor:min width}
we have $r(C)=r(K)$ and $R(C)=R(K)$ by definition as well as
$D(C)\leq D(K)$ and $w(C)\leq w(K)$
because of the monotonicity of the radii with respect to set inclusion.

The idea of the proof is to transform $C$ in several steps into a bent isosceles $\BIso$
from $(lb_3)$ of Subsection \ref{ss:facets}
keeping the same in- and circumradius at all time and guaranteeing that $D(\BIso) = D(K)$
and $w(\BIso) \leq w(K)$
at the end of the transformation (and obtaining the corresponding solution for $\BPen$).
More precisely, we know that the parallel supporting lines $L_1$ and $L_2$ of $C$ from
Remark \ref{rem:min width} satisfy $w(C)=d(L_1,L_2)$. Now, in every step of the transformation of $C$,
$d(L_1,L_2)$ will be decreased, however when arriving at $\BIso$ it again holds
$w(\BIso)=d(L_1,L_2)$ (as shown in $(lb_3)$).

To reduce notation formalities we assume \Wlog that $L_1,L_2$ are embedded horizontally
and we denote the inball by $\rB$.

\begin{enumerate}[(a)]
\item The first step is only needed in case of $c_1 < 0$. In this step all radii except the diameter
  of $C$ are kept constant, while the diameter may be reduced but not raised.

  If $L_1$ and $L_2$ are arranged as in Case (iii) of Remark \ref{rem:min width},
  then using Part (b) of Proposition \ref{prop:exposed}, we see that $c_1 <0$ is not possible as
  $c_1 \geq p^3_1\geq 0$ holds.
  In case of (ii) or (iv), we may translate $B$ parallel to $L_1$ until $c_1=0$.
  Because of Corollary \ref{cor:min width} this transformation does not increase
  $D(C)$: in both cases the only candidate
  distance for the diameter which is raised is $\norm[p^1-c]+r(C)$, but in case of (iv) $\norm[p^1-c]$ is
  before and after the
  transformation bounded from above by $\norm[p^2-c]$ and in case of (ii) it is bounded from above by
  $\norm[p^2-c]$ or by $\norm[p^3-c]$.

  Finally, Case (i) can be handled almost the same. If $p^3$ is closer to
  $L_2$ than $p^1$, again $\norm[p^2-c]$ is bounded from above by $\norm[p^2-c]$
  or by $\norm[p^3-c]$ and we may directly move $B$ parallel to $L_1$ until $c_1=0$.
  If, on the contrary, $p^1$ is the point closer to $L_2$, then we first rotate $C$ between
  $L_1$ and $L_2$ until $p^1$ and $p^3$ get into same distance to $L_2$ and then we may do the movement of $B$.

\item Translate $p^2$ and $p^3$ on $\S$ within the arcs between $L_1$ and $L_2$ they belong to,
  until $\norm[p^1-p^2]=\norm[p^1-p^3]=D(K)$. Since $c_1 \geq 0$ and $D(C) \le D(K)$ before the transformation
  we have $D(C)=D(K)$ after the transformation and we keep at least $p^1$ on $L_1$ or $L_1$ tangent
  to the unit ball.
  Moreover, if neccessary, moving $L_2$ parallel to its prior position until it supports $C$ again,
  $L_2$ touches $p^3$ or $\rB$ (see Figure \ref{fig:Step1lb3}),
  $d(L_1,L_2)$ does not increase, and $r(C)$ and $R(C)$ stay constant.

  However, in each situation where we only touch two points after the transformation,
  we may additionally rotate $L_1$ and $L_2$ around the vertices or along the insphere, respectively,
  not increasing their distance, until we obtain a third touching point of the two lines with $C$.
  In the following we distinguish the following cases,
  (exchanging, if necessary, the roles of $L_1,L_2$ and $p^2,p^3$, respectively,
    to attain one of them)
  \begin{enumerate}[(i)]
  \item $L_1$ contains $p^1$, or
  \item $L_1$ contains $p^2$ and supports $\rB$ and $L_2$ contains $p^3$, or
  \item $L_1$ does not contain any of the points $p^i$, while $L_2$ contains $p^3$ but no other.
  \end{enumerate}
  In all three cases we will search for a situation, in which the angle between
  $L_1$ with one of the edges of $\Iso$ is acute.
  If (i) holds, the angle between $L_1$ and $[p^1,p^3]$ must always be acute as $p^3$ lies on the right side
  of $p^1$.

  In case of (ii),
  the angle between $L_1$ and $[p^2,p^3]$ can either be acute (see Figure \ref{fig:Step2lb3})
  or obtuse, as $p^3$ could even be on the right of $p^2$.
  If the latter happens we rotate
  $L_1$ around $p^2$ (thus possibly loosing contact with $\rB$)
  and $L_2$ around $p_3$, keeping them parallel, until $L_2$ supports $\rB$, allowing a zero degree
  rotation in the case that $L_2$ supported $\rB$ from the beginning. This does not increase $d(L_1,L_2)$.

  Finally if (iii) holds, the angle between $L_2$ and $[p^2,p^3]$ could be obtuse.
  Then we rotate both lines $L_1, L_2$ along $\rB$, loosing contact with $p^3$,
  until $L_2$ touches $p^1$ or $L_1$ touches $p^2$ (whichever comes first).
  In case $L_2$ touches $p^1$ first we are back in (i). Thus assume
  $L_1$ touches $p^2$ first.
  Compare with the bent isosceles $\BIso$ we want to arrive at:
  Because of our movement in the beginning of (b), we have that $\conv\{p^1,p^2,p^3\}$
  is an isosceles triangle with inball $\rB$ contained in $C$. Thus identifying it
  with
  $\Iso \subset \BIso$ yields that the
  supporting lines $L_1', L_2'$ of $\BIso$ contain $p^2$ and $p^3$, respectively,
  and contain between them the inball of radius $r$. Thus it holds $\norm[p^2-p^3]\ge 2r$.
  Considering $C$ again, since the angle between $L_2$ and $[p^2,p^3]$ was obtuse
  before the rotation of $L_1, L_2$ in (iii), the incenter $c$ is closer to $p^3$ than
  to $p^2$. But since  $\norm[p^2-p^3]\ge 2r$, this means after the rotation
  the angle between $L_1$ and $[p^2,p^3]$ must be acute.
  Exchanging if necessary $L_1$ with $L_2$ and $p^2$ with $p^3$,
  we are back again in the cases (i), (ii), (iii) or (iv)
  of Remark \ref{rem:min width}, also not guaranteeing that the distance between those lines
  defines the width of $C$, but knowing that the angle between $L_1$ and $[p^2,p^3]$
  (if (i), (ii) or (iv) hold) or $L_1$ and $[p^1,p^3]$ (if (iii) holds) is
  acute (see Figure \ref{fig:Step3lb3}).

  \begin{figure}
    \centering
    \begin{subfigure}[b]{0.45\textwidth}
      \includegraphics[trim = 1cm 3.5cm 1cm 2cm, width =\textwidth]{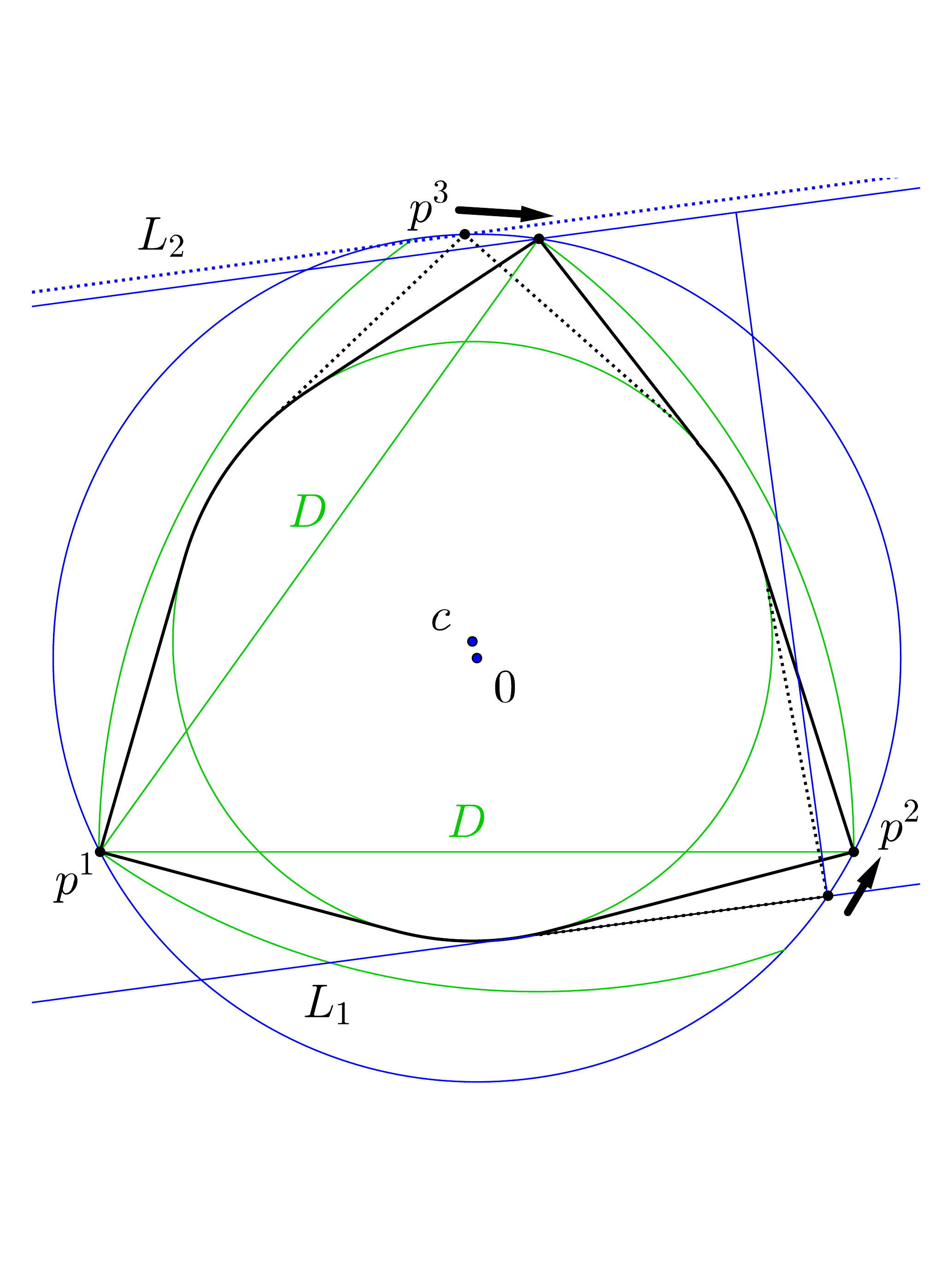}
      \caption{In (b), the tangencies correspond to Part (ii) of Remark \ref{rem:min width}.
        While moving $p^2, p^3, L_1$ and $L_2$ some of the tangencies can be lost but we obtain that
        $\Iso$ is contained in $C$. \\ \hfill
      }
      \label{fig:Step1lb3}
    \end{subfigure} \hfill 
    \begin{subfigure}[b]{0.45\textwidth}
      \includegraphics[trim = 1cm 3.5cm 1cm 2cm, width = \textwidth]{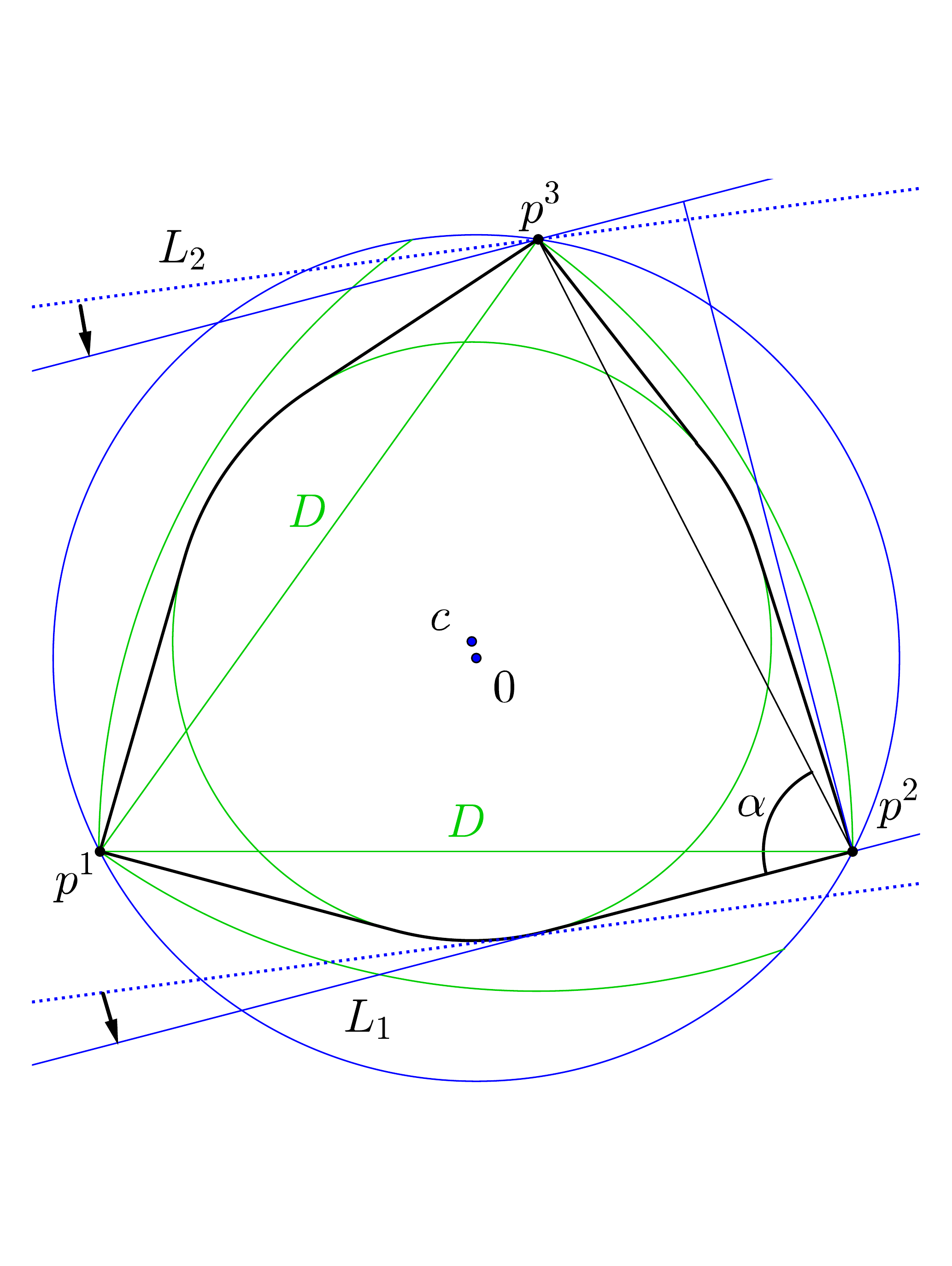}
      \caption{We rotate $L_1, L_2$ around $C$ until $L_1$ or $L_2$ supports more than one point of $C$,
          arriving, \eg, in the situation in
          which $L_1$ contains $p^2$ and supports
          $\rB$, $L_2$ supports $p^3$, and $\alpha$ is acute here.
          Then we are back into the tangencies of Part (ii).}
      \label{fig:Step2lb3}
    \end{subfigure}
    \caption{Transformations of $C$ during (b).
      Here the green bows indicate arcs of radius $D$ and centers in $p^1, p^2, p^3$, defining a
        region in which both, $\rB$ and the $p^i$, have to be contained.}
    \label{fig:sc5lb3}
    \end{figure}

\item In the last step of the transformation of $C$ we only move $\rB$ and $L_1, L_2$,
  keeping $r(C), D(C)$, and $R(C)$ constant.
  Independently of the tangencies (i)--(iv) of Remark \ref{rem:min width},
  $\rB$ is translated until it becomes tangent with the pair of arcs with centers
  $p^1, p^2$ and radius $D(K)$, finishing the transformation of $C$ into $\BIso$.
  Finally $L_1$ is rotated around $p^1$ (in case of Part (iii)
  of Remark \ref{rem:min width}) or around $p^2$ (in all other cases)
  keeping it tangent to $\rB$ and $L_2$ keeping it parallel to $L_1$ and supporting $C$.
  A simple but crucial observation is the following: assuming that $L_1$ contains $p^2$, it was shown
  in $(lb_3)$ of Subsection \ref{ss:facets} that $\rB$ is the inball of $\BPen$, touching its boundary in
  the diametrical arcs around $p^1$,$p^2$ and in $L_1$. Therefore
  any translation of $\rB$ within the region spanned by the two arcs would
  lead to an intersection of $L_1$ with the interior of $\rB$,
  which means that before the rotation of $L_1$, its angle with $[p^2,p^3]$
  was not smaller than after.
  This observation implies that the breadth $b_{s}([p^2,p^3])$ with $s$ orthogonal to the two lines is reduced
  by the rotation.

  \begin{figure}
    \begin{center}
      \includegraphics[trim = 0cm 4cm 0cm 4cm, width=9cm]{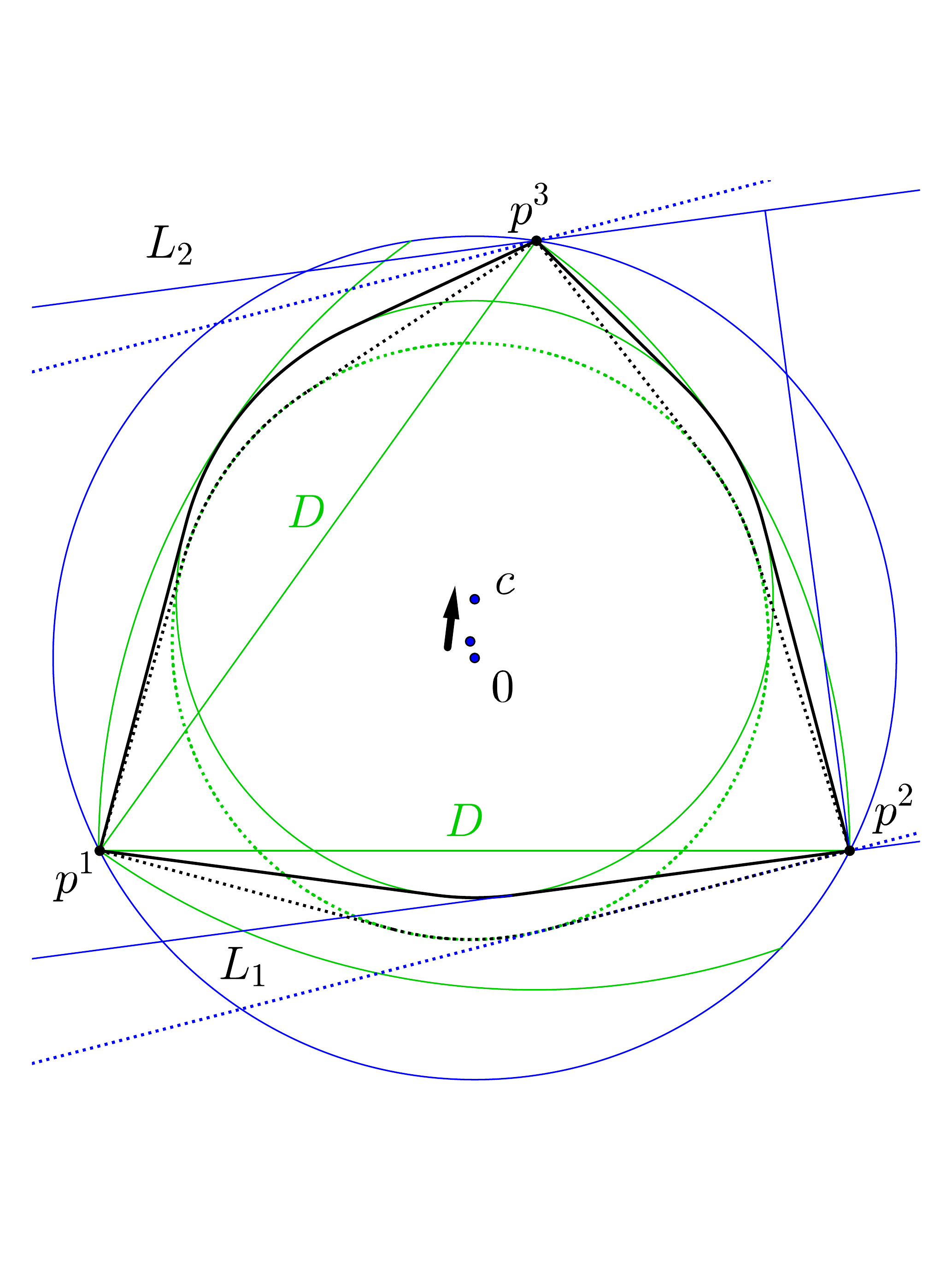}
      \caption{In (c) the inball moves, $L_1$ is rotated around $p^2$ reducing
        the angle with $[p^2,p^3]$  and $L_2$ to keep parallel with $L_1$
        until $C=\BIso$.}
      \label{fig:Step3lb3}
    \end{center}
  \end{figure}

  However, in Cases (ii) and (iii) of Remark \ref{rem:min width} it obviously
  holds $b_{s}([p^2,p^3])=d(L_1,L_2)$
  and since $d(L_1,L_2)$ did not increase in any step of the transformation we obtain
  $w(K) \ge d(L_1,L_2) \ge w(C)$.

  Finally, consider the remaining cases, (i) and (iv): they describe the extremal
  situation when $C$ shares radii with a set from the edges $(\BT,\H)$ or $(\BT,\FRT)$.
  In case of (i) $L_1$ and $L_2$ support $\rB$, which means $d(L_1,L_2)=2r$ and therefore that
  $w(C)=d(L_1,L_2)\le w(K)$. In case of (iv) we have $L_1 \supset [p^1,p^2]$ and $p^3 \in L_2$ parallel to $L_1$.
  Hence $w(C) \le d(L_1,L_2) = w(\Iso)$, within the given parameters for $\gamma$, and since $\Iso \subset C$
  it follows $w(C) = d(L_1,L_2)$.
\end{enumerate}
\end{proof}

\begin{proof}[Proof of Theorem \ref{thm:3}]
  As before we abbreviate $r(K)=r$ and the same for the other radii.
  In order to show the general validity of the inequality \eqref{eq:3PT}, we split the proof into
  the following cases:
  \[
  \begin{split}
    \text{(i)} \quad 8r \ge 3D, \gamma \ge \gamma_r, r \le r(\H), \qquad
    & \text{(ii)} \quad 8r < 3D \\
    \text{(iii)} \quad \gamma < \gamma_r, r(\BT) \le r \le r(\H), \qquad
    & \text{(iv)} \quad r > r(\H).
  \end{split}
  \]
  Recognize that in case of (i) there exists a bent pentagon $\BPen$, as we have shown with
  the help of Lemma \ref{lem:lb3properties} in $(lb_3)$. Thus we are under the conditions of
  Lemma \ref{lem:3 Bent Isosceles} in that case.

  In the remaining cases, consider the generalized bent pentagon $\BPen$ as defined in the
  description of the facet $(lb_3)$ (together with all the notation used there)
  and observe that the distance $d(L_1,L_2)$ may in any case be computed as the width in \eqref{eq:width-bentpent}.
  The angle $\beta$ may become $-\beta$ in Case (ii) (\cf~~Figure \ref{fig:GenBP})
  or the angle $\mu$ may become $-\mu$ in Cases (iii) or (iv),
  whenever the angle between $L_1$ and $[p^2,p^3]$ is bigger than
  $\nicefrac{\pi}{2}$.
  In both cases this change of sign does not affect the final value of the right hand side of the
  inequality \eqref{eq:3PT} to coincide with $d(L_1,L_2)$.

  \begin{figure}
    \begin{center}
      \includegraphics[trim = 0cm 5cm 0cm 5cm, width=10cm]{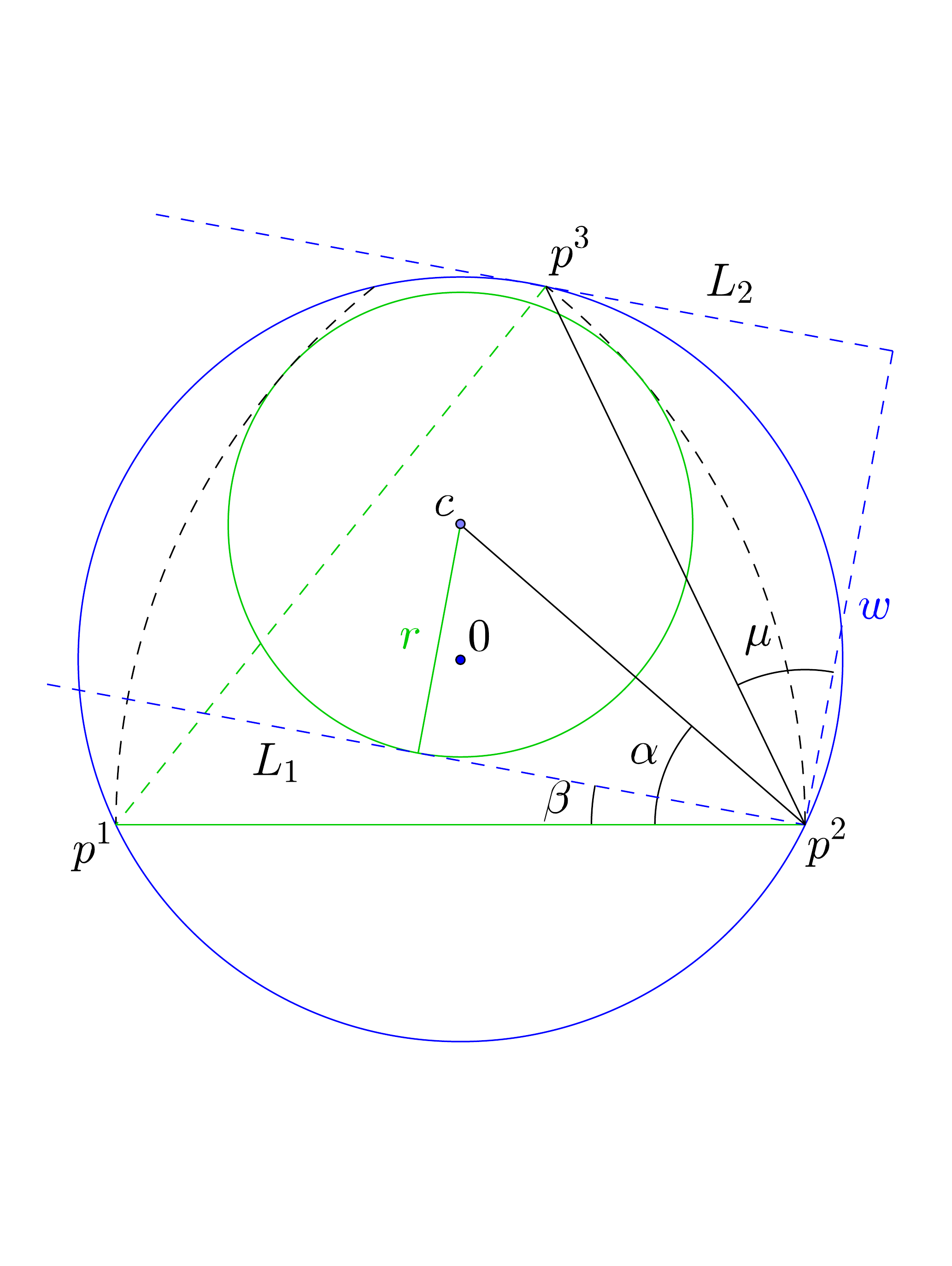}
      \caption{If $3 D > 8 r$, the angle $\beta$ in the computations in $lb_3$ (\cf~Figure \ref{fig:lb_3})
        becomes $-\beta$, but does not change the final equation for $d(L_1,L_2)$.}
      \label{fig:GenBP}
    \end{center}
  \end{figure}

  Hence it suffices to show $w \geq d(L_1,L_2)$.
  For this assume \Wlog that $[p^1,p^2]$ is horizontal and below $0$, that $p^1_1\le p^2_1$, and that
  $p^3_2\ge 0$.

  In case of (ii), Part (a) of Lemma \ref{lem:lb3properties} ensures that $[p^1,p^2]$
  does not intersect $\rB$.
  Thus the slope of the $L_i$'s is negative,
  and considering the line $L$ containing $[p^1,p^2]$, the angle between $L_1$ and $[p^2,p^3]$ is
  smaller than the angle between $L$ and $[p^2,p^3]$ (\cf~Figure \ref{fig:GenBP}).
  Hence, denoting the line containing $p^3$ and parallel to $L$ by $L'$, their distance satisfies
  $d(L_1,L_2) \le d(L,L') = w(\Iso) \le w$.

  Now, let us assume that (iii) is true, but not (ii). Then we know from
  Part (b) of Lemma \ref{lem:lb3properties}, that the distance $d(L_1,L_2)$ decreases if $\gamma$ decreases.
  Using $\gamma \leq \gamma' = \gamma_r$ we obtain $d(L_1,L_2) \le d(L_1',L_2') =
  w(\BIso[r][\gamma_r]) =2r \le w$.

  Finally, for the treatment of (iv), one should first observe two easy facts: first, since
  $p^2 \in L_1$ and $p^3 \in L_2$ we have $d(L_1,L_2) \le \norm[p^2-p^3]$ and second
  if $r=r(\H),\gamma=2\arccos(\nicefrac{D(\H)}{2})$, and $L_1^{\H}, L_2^{\H}$ are the according support lines of $\H$,
  then $L_1^{\H}, L_2^{\H}$ are perpendicular to $[p^2,p^3]$ and thus $\norm[p^2-p^3]=w(\H)=2r(\H)$
  (\cf~the description of $\H$ in Subsection \ref{ss:vertices}).
  From (iv) and inequality \eqref{ib_2} we obtain that $D(\H)=r(\H)+1\le r+1\le D$ and
  since $[p^2,p^3]$ is the shorter edge of $\Iso$ we have $\norm[p^2-p^3]=\nicefrac{D}{R}\,\sqrt{4R^2-D^2}$
  which is a decreasing function on $D$.
  Hence $\norm[p^2-p^3]$ is maximized, when $\gamma=\gamma_r$, \ie when $\norm[p^2-p^3] = w(\H)$.
  Thus using inequality \eqref{lb_1} we obtain
  \[ d(L_1,L_2)\le\norm[p^2-p^3]\leq w(\H)=2r(\H)\leq 2r\leq w, \]
  which completes the proof.
\end{proof}

\section{Final remarks} \label{s:outview}

For finishing the paper, let us give two final remarks:

\medskip

First, for some practical purposes it could be of some value to be able
to replace the sometimes quite unhandy non-linear
inequalities by linear ones. Thus knowing the full extend of the diagram know,
it would be worthwhile to develope a complete system of linear inequalities supporting the diagram.
Since the convex hull of the vertices does not contain the full diagram (the supporting plane of $\L,\EqT$,
and $\RAT$ separates $\CSB$ from major parts of the diagram) and since all edges and facets are smooth,
this system cannot be finite.

\medskip

Second, especially considering the application of Blaschke-Santal\'o diagrams given in \cite{RDP6, RDP5, Pr},
  consider the following problem: suppose two convex sets $K$ and $K'$ are mapped to the same point in the diagram,
  how \enquote{different} may $K$ and $K'$ be? Before giving any answer to this question, we should first develope an
  idea, how to measure this \enquote{difference}.
  For this neither the usual Hausdorff nor the Banach-Mazur distance can be taken.
  For the Hausdorff distance any $K$ and some of its rotations may be quite far from each other, while
  the Banach-Mazur distance would mark (\eg) all simplices equal.
  A good choice for this task could be taking the Hausdorff distance within the class of
  similarities of the two sets. However, to the best of our knowledge, this distance is not considered in
  literature so far.

\bigskip

{\bf Acknowledgements:} We would like to thank Viviana Ghiglione and Evgeny Zavalnyuk
for giving crucial hints, as well as Peter Gritzmann, Maria Hern\'andez Cifre, and Salvador Segura Gomis
for always supporting us.

\bibliographystyle{amsplain}

\end{document}